\documentclass[12pt,a4paper]{amsart}
\usepackage{amsthm,amsfonts,amsmath,amssymb,latexsym}
\usepackage{epsfig,graphics,color}
\usepackage[all]{xy}
\usepackage[breaklinks=true]{hyperref}
\usepackage{stmaryrd}
\usepackage{verbatim}
\usepackage{bm}
\usepackage{mathabx}
\usepackage{enumitem}
\usepackage{pgf,amsmath,tikz,type1cm,fix-cm}
\usetikzlibrary{positioning}
 
\newlength{\XHeight}
\newlength{\XWidth}
\setlist[itemize,1]{leftmargin=\dimexpr 26pt-.1in}

\newtheorem{PARA}{}[section]
\newtheorem{theorem}[PARA]{Theorem}
\newtheorem{corollary}[PARA]{Corollary}
\newtheorem{lemma}[PARA]{Lemma}
\newtheorem{proposition}[PARA]{Proposition}
\newtheorem{definition}[PARA]{Definition}
\newtheorem{definition-proposition}[PARA]{Definition-Proposition}

\newtheorem{question}[PARA]{Question}
\theoremstyle{definition}
\newtheorem{remark}[PARA]{Remark}
\theoremstyle{theorem}
\newtheorem{example}[PARA]{Example}
\newcommand{\para}{\begin{PARA}\rm}
\newcommand{\arap}{\end{PARA}\rm}
\newcommand{\dfn}{\begin{definition}\rm}
\newcommand{\nfd}{\end{definition}\rm}
\newcommand{\rmk}{\begin{remark}\rm}
\newcommand{\kmr}{\end{remark}\rm}
\newcommand{\xmpl}{\begin{example}\rm}
\newcommand{\lpmx}{\end{example}\rm}
\newcommand{\cA}{\mathcal{A}}
\newcommand{\cB}{\mathcal{B}}

\newcommand{\cF}{\mathcal{F}}
\newcommand{\cE}{\mathcal{E}}

\newcommand{\cM}{\mathcal{M}}

\newcommand{\cN}{\mathcal{N}}
\newcommand{\cO}{\mathcal{O}}
\newcommand{\cP}{\mathcal{P}}

\newcommand{\cS}{\mathcal{S}}

\newcommand{\cU}{\mathcal{U}}

\newcommand{\cZ}{\mathcal{Z}}
\newcommand{\BB}{\mathcal{B}}
\newcommand{\EE}{\mathcal{E}}
\newcommand{\OO}{\mathcal{O}}
\newcommand{\ZZ}{\mathcal{Z}}
\newcommand{\FF}{\mathcal{F}}
\newcommand{\one}
{{{\mathchoice \mathrm{ 1\mskip-4mu l} \mathrm{ 1\mskip-4mu l}
\mathrm{ 1\mskip-4.5mu l} \mathrm{ 1\mskip-5mu l}}}}

\newcommand{\C}{{\mathbb{C}}}

\renewcommand{\H}{{\mathbb{H}}}

\newcommand{\N}{{\mathbb{N}}}

\newcommand{\R}{{\mathbb{R}}}

\newcommand{\Z}{{\mathbb{Z}}}
\newcommand{\coker}{\mathrm{ coker }}  
\newcommand{\im}{\mathrm{im}\,}        
\newcommand{\sign}{\mathrm{ sign }}    
\newcommand{\id}{\mathrm{ id}}         

\newcommand{\supp}{\mathrm{ supp}}     
\newcommand{\ind}{\mathrm{ind}}
\renewcommand{\Re}{\mathrm{ Re\,}}       
\renewcommand{\Im}{\mathrm{ Im\,}}       

\newcommand{\ev}{\mathrm{ev}}
\newcommand{\Aut}{\mathrm{ Aut}}          

\newcommand{\Hom}{\mathrm{Hom}}
\newcommand{\Fix}{\mathrm{Fix}}
\newcommand{\rk}{\mathrm{rk}}
\newcommand{\Max}{\mathrm{Max}}
\newcommand{\eps}{{\varepsilon}}
\newcommand{\om}{{\omega}}
\newcommand{\Om}{{\Omega}}

\newcommand{\CZ}{\mathrm{CZ}}
\def\NABLA#1{{\mathop{\nabla\kern-.5ex\lower1ex\hbox{$#1$}}}}
\def\Nabla#1{\nabla\kern-.5ex{}_{#1}}
\def\Tabla#1{\Tilde\nabla\kern-.5ex{}_{#1}}
\renewcommand{\Tilde}{\widetilde}

\newcommand{\p}{{\partial}}
\newcommand{\dbar}{{\bar\partial}}

\newenvironment{enum}
{\begin{enumerate}}
{\end{enumerate}}

\newcommand{\la}{\langle}
\newcommand{\ra}{\rangle}
\newcommand{\wh}{\widehat}
\newcommand{\ol}{\overline}
\newcommand{\ul}{\underline}
\newcommand{\MM}{\mathcal{M}}

\newcommand{\PP}{\mathcal{P}}

\newcommand{\Crit}{{\rm Crit}}
\newcommand{\wt}{\widetilde}

\newcommand{\K}{\mathbb{K}}
\newcommand{\into}{\hookrightarrow}
\newcommand{\pb}{\overline{\partial}}
\newcommand{\st}{\mathrm{st}}

\newcommand{\pt}{{\rm pt}}
\newcommand{\Skel}{{\rm Skel}}
\newcommand{\boldtau}{\tau} 

\begin{document}

\title[Loop coproduct in Morse and Floer homology]{Loop coproduct in Morse and Floer homology}
\author{Kai Cieliebak}
\address{Universit\"at Augsburg \newline Universit\"atsstrasse 14, D-86159 Augsburg, Germany}
\email{kai.cieliebak@math.uni-augsburg.de}
\author{Nancy Hingston}
\address{Department of Mathematics and Statistics, College of New Jersey \newline Ewing, New
Jersey 08628, USA}
\email{hingston@tcnj.edu} 
\author{Alexandru Oancea}
\address{
Institut de recherche math\'ematique avanc\'ee, IRMA \newline
Universit\'e de Strasbourg and CNRS\newline
Strasbourg, France}
\email{oancea@unistra.fr}
\date{\today}


\maketitle

\begin{center}
{\it To Claude Viterbo on the occasion of his 60th birthday, with admiration}
\end{center}

\begin{abstract}
By a well-known theorem of Viterbo, the symplectic homology of the cotangent bundle of a closed manifold is isomorphic to the homology of its loop space. In this paper we extend the scope of this isomorphism in several directions. First, we give a direct definition of {\em Rabinowitz loop homology} in terms of Morse theory on the loop space and prove that its product
agrees with the pair-of-pants product on Rabinowitz Floer homology. The proof uses compactified moduli spaces of punctured annuli. Second, we prove that, when restricted to {\em positive} Floer homology, resp.~loop space homology relative to the constant loops,
the Viterbo isomorphism intertwines various constructions of secondary
pair-of-pants coproducts with the loop homology coproduct. Third, we introduce {\em reduced loop homology}, which is a common domain of definition for a canonical reduction of the loop product and for extensions of the loop homology coproduct which together define the structure of a commutative cocommutative unital infinitesimal anti-symmetric bialgebra. Along the way, we show that the Abbondandolo-Schwarz quasi-isomorphism going from the Floer complex of quadratic Hamiltonians to the Morse complex of the energy functional can be turned into a filtered chain isomorphism by using linear Hamiltonians and the square root of the energy functional. 
\end{abstract}



\section{Introduction}\label{sec:introduction}

For a closed manifold $M$ there are canonical isomorphisms
\begin{equation}\label{eq:isos}
   H_*(\Lambda,\Lambda_0;\eta) \cong FH_*^{>0}(T^*M) \cong SH_*^{>0}(D^*M) \cong SH^{1-*}_{<0}(S^*M).
\end{equation}
Here we use coefficients in any commutative ring $R$, twisted in the first group by a suitable local system $\eta$ which restricts to the orientation local system on the space $\Lambda_0\subset \Lambda$ of constant loops (see Appendix~\ref{sec:local-systems}). 
The groups in the above chain of isomorphisms are as follows: $H_*(\Lambda,\Lambda_0)$ denotes the homology of the free loop space $\Lambda=C^\infty(S^1,M)$ relative to $\Lambda_0$;
$FH_*^{>0}(T^*M)$ the positive action part of the Floer homology of a fibrewise quadratic Hamiltonian on the cotangent bundle; 
$SH_*^{>0}(D^*M)$ the positive symplectic homology of the unit cotangent bundle $D^*M$; 
and $SH^{1-*}_{<0}(S^*M)$ the negative symplectic cohomology of the trivial Liouville cobordism $W=[1,2]\times S^*M$ over the unit cotangent bundle $S^*M$. 
The first isomorphism is the result of work of many people starting with Viterbo (see~\cite{Viterbo-cotangent,AS,AS-corrigendum,AS2,SW,Ritter,Cieliebak-Latschev,Kragh,Abouzaid-cotangent}); the second one is obvious; and the third one is a restriction of the Poincar\'e duality isomorphism from~\cite{CO}. 

Restricting to field coefficients, all the groups in~\eqref{eq:isos} carry natural coproducts of degree $1-n$:
\begin{itemize}
\item the {\em loop homology coproduct} (in the sequel simply called {\em loop coproduct}) $\lambda$ on $H_*(\Lambda,\Lambda_0;\eta)$ defined by Sullivan~\cite{Sullivan-open-closed} and further studied by Goresky and the second author in~\cite{Goresky-Hingston}, see also~\cite{Hingston-Wahl}; 
\item the {\em (secondary) pair-or-pants coproduct} $\lambda^{AS}$ on $FH_*^{>0}(T^*M)$ defined by Abbondandolo and Schwarz~\cite{AS-product-structures}; 
\item the {\em varying weights coproduct} $\lambda^w$ on $SH_*^{>0}(D^*M)$ first described by Seidel and further explored in~\cite{Ekholm-Oancea};
\item the {\em continuation coproduct} $\lambda^F$ on $SH_*^{>0}(D^*M)$ described in~\cite{CO-cones};
\item the {\em Poincar\'e duality coproduct} 
$\lambda^{PD}$ on $SH^{1-*}_{<0}(S^*M)$ dual to the pair-of-pants product on $SH_{1-*}^{<0}(S^*M)$, described in~\cite{CHO-PD}. 
\end{itemize}
The first result of this paper is 

\begin{theorem}\label{thm:main1}
With field coefficients, all the above coproducts are equivalent under the isomorphisms in~\eqref{eq:isos}. 
\end{theorem}

\begin{remark}[Coproducts and field coefficients]
There is a formal algebraic reason why we need to restrict to field coefficients when speaking about homology coproducts. Given a chain complex $C=C_*$ and a chain map $C\to C\otimes C$, we obtain a map $H_*(C) \to H_*(C\otimes C)$. However, the latter factors through $H_*(C)\otimes H_*(C)$ only if the K\"unneth isomorphism $H_*(C)\otimes H_*(C)\stackrel\simeq\to H_*(C\otimes C)$ holds, which is the case with field coefficients. All our coproducts are defined at chain level with arbitrary coefficients, and we would not need to restrict to field coefficients if we carried the discussion at chain level.
\end{remark}

\begin{remark}[coefficients twisted by local systems]
The chain of isomorphisms~\eqref{eq:isos} also holds if one further twists each of the factors by an additional local system. If the latter is compatible with products in the sense of Appendix~\ref{sec:spaces-of-loops}, then all groups still carry natural coproducts of degree $1-n$ and Theorem~\ref{thm:main1} continues to hold. This is particularly relevant when $M$ is orientable: the constant local system on $\Lambda$ is indeed of the form $\sigma^{-1}\otimes \eta$, where $\sigma$ is the transgression of the second Stiefel-Whitney class, so that $H_*(\Lambda,\Lambda_0)\simeq SH_*^{>0}(D^*M;\sigma^{-1})$. See Appendix~\ref{sec:iso-sympl-loop}. 
\end{remark}

\begin{remark}
All our statements have counterparts for open strings, in which the free loop space is replaced by the based loop space and symplectic homology of $T^*M$ is replaced by wrapped Floer homology of the cotangent fiber $T^*_qM$. See~\cite{CHO-PD}. We do not spell out these results and focus on closed strings in this paper. 
\end{remark}

The first two isomorphisms in~\eqref{eq:isos} are obtained by dividing
out the constant loops, resp.~the action zero part in the chain of isomorphisms

\begin{equation}\label{eq:isos2}
   H_*(\Lambda;\eta) \cong FH_*(T^*M) \cong SH_*(D^*M).
\end{equation}

According to Abbondandolo and Schwarz~\cite{AS2}, these
isomorphisms intertwine the Chas--Sullivan loop product~\cite{CS} on
the first group with the pair-of-pants products on the other two groups.
On the other hand, according to~\cite{CHO-PD,CHO-reduced}, the product on
$SH_*(D^*M)$ and the coproduct on $SH_*^{>0}(D^*M)$ are related to the
pair-of-pants product and coproduct on $SH_*(S^*M)$ by the ``almost split" exact sequence  
\begin{equation}\label{eq:les+}
{\scriptsize
\xymatrix
@C=20pt
{
& & & & SH^{1-*}_{>0}(D^*M) \ar[dl]_-{i} \ar[d]^{j} & \\
    \ar[r]& SH^{-*}(D^*M) \ar[r]^-\eps & SH_*(D^*M) \ar[d]^-q \ar[r]^-\iota &
    SH_*(S^*M) \ar[dl]_-p \ar[r]^-\pi & SH^{1-*}(D^*M)\ar[r]&  \\
    & & SH_*^{>0}(D^*M) & & 
}
}
\end{equation}
where the maps have the following properties. 
\begin{itemize} 
\item The map $\iota$ intertwines the pair-of-pants products, and the map $\pi$ intertwines the pair-of-pants coproducts.
\item The ``almost splitting" $i$ satisfies $\pi\circ i = j$ and
  intertwines the product dual to $\lambda^F$ on
  $SH^{1-*}_{>0}(D^*M)$ with the pair-of-pants product on $SH_*(S^*M)$. 
\item The ``almost splitting" $p$ satisfies $p\circ \iota = q$ and intertwines the coproduct on $SH_*(S^*M)$ with the continuation coproduct $\lambda^F$ on $SH_*^{>0}(D^*M)$.
\item The map $\eps$ lives only in degree $0$ and factors through the constant loops as the connecting map in the Gysin sequence for the cohomology $H^{n-*}(S^*M)$ 
\begin{equation}\label{eq:eps}
\xymatrix{
  SH^{-*}(D^*M) \ar[r]^-\eps \ar[d]& SH_*(D^*M) \\
  H^{-*}M \ar[r]^-e & H_*M \ar[u]
}
\end{equation}
\end{itemize} 
Here the map $e$ is multiplication with the Euler characteristic of $M$ in degree $0$.  
From this perspective, and up to some discrepancy at the constant loops, both the
pair-of-pants product on $SH_*(D^*M)$ and the product dual to
$\lambda^F$ on $SH^{1-*}_{>0}(D^*M)$ appear as ``components'' of
the pair-of-pants product on $SH_*(S^*M)$. See~\cite[\S7]{CHO-reduced}. 

Our second goal is to define a topological counterpart of $SH_*(S^*M)$.\footnote{In~\cite{CO} the group $SH_*(S^*M)$ was called \emph{symplectic homology of (the trivial cobordism over) $S^*M$}, and in~\cite{Cieliebak-Frauenfelder-Oancea} it was proved to be isomorphic to the \emph{Rabinowitz Floer homology group} $RFH_*(S^*M)$. In the sequel we will allow ourselves to use both names. The isomorphism $\wh H_*\Lambda\simeq SH_*(S^*M)$ motivates our terminology \emph{Rabinowitz loop homology} for $\wh H_*\Lambda$.}
The starting point is the topological counterpart of diagram~\eqref{eq:eps}:
\begin{equation}\label{eq:eps-top}
\xymatrix{
  H^{-*}\Lambda \ar[r]^-\eps \ar[d]& H_*\Lambda \\
  H^{-*}M \ar[r]^-e & H_*M \ar[u]
}
\end{equation}
Here the map $\eps$ is induced by a chain map on the Morse complex
(with respect to the energy functional)
$$
   c:MC^{-*}(\Lambda)\to MC^{-*}(M)\to MC_*(M)\to MC_*(\Lambda),
$$
where the exterior maps are induced by the inclusion of constant loops,
and the middle map lives in degree zero and is given by multiplication
with the Euler characteristic of $M$. 
We define the {\em Rabinowitz loop homology} as the homology of the cone of $c$,
$$
   \wh H_*\Lambda := H_*(Cone(c)). 
$$
By general properties of the cone construction (see e.g.~\cite{CO}),
this fits into a long exact sequence 
\begin{equation}\label{eq:les-Lambda}
\xymatrix
@C=20pt
{
    \ar[r]& H^{-*}\Lambda \ar[r]^-\eps & H_*\Lambda \ar[r]^-\iota &
    \wh H_*\Lambda \ar[r]^-\pi & H^{1-*}\Lambda\ar[r]&  
}
\end{equation}
Our second result is

\begin{theorem}\label{thm:main2}
The Rabinowitz loop homology $\wh H_*\Lambda$ carries a natural product
of degree $-n$ such that the map $\iota$ in~\eqref{eq:les-Lambda} is a ring homomorphism. Moreover, for $n\neq 2$ there exists an isomorphism of rings
$SH_*(S^*M)\cong \wh H_*\Lambda$ such that the following diagram commutes:
\begin{equation*}
{\scriptsize
\xymatrix
@C=20pt
{
    \cdots SH^{-*}(D^*M) \ar[r]^-\eps & SH_*(D^*M) \ar[d]^\cong \ar[r]^-\iota &
    SH_*(S^*M) \ar[d]^\cong\ar[r]^-\pi & SH^{1-*}(D^*M)\cdots \\
    \cdots H^{-*}\Lambda \ar[u]_\cong \ar[r]^-\eps & H_*\Lambda \ar[r]^-\iota &
    \wh H_*\Lambda \ar[r]^-\pi & H^{1-*}\Lambda\cdots \ar[u]_\cong 
}
}
\end{equation*}
\end{theorem}

\begin{remark}
(a) In~\cite{CHO-PD} we {\em defined} $\wh H_*\Lambda$ as $SH_*(S^*M)$,
and with this definition Theorem~\ref{thm:main2} is a tautology. 
The point of the present paper is to define $\wh H_*\Lambda$ in purely
topological terms as above, in which case Theorem~\ref{thm:main2}
becomes an actual theorem. It can be seen as an upgrade of Viterbo's
isomorphism~\cite{Viterbo-cotangent} from symplectic homology to Rabinowitz Floer homology. 

(b) The hypothesis $n\neq 2$ is only an artefact of our proof and can be removed by upgrading the theory of $A_2^+$-structures in~\cite{CO-cones} to a theory of $A_3^+$-structures, which would take into account arity 3 operations.
 
\end{remark}

One difficulty with the proof of Theorem~\ref{thm:main2} is the lack
of an obvious chain map inducing the isomorphism $SH_*(S^*M)\cong \wh H_*\Lambda$,
due to the fact that the natural chain maps inducing Viterbo's
isomorphisms on homology and cohomology go in opposite directions. 
We overcome this difficulty using the theory of {\em $A_2^+$-structures}
from~\cite{CO-cones}. We will prove that the Abbondandolo-Schwarz map on chain
level yields a quasi-isomorphism of $A_2^+$-structures, and then
appeal to algebraic results from~\cite{CO-cones} concerning such structures and their
associated cones.

Starting from the exact sequence~\eqref{eq:les-Lambda} we define in this paper \emph{reduced loop homology and cohomology}
$$
\ol H_*\Lambda = \coker\,\eps, \qquad \ol H^*\Lambda= \ker\eps.
$$ 
\begin{theorem}[{\cite{CHO-reduced}}]\label{thm:main3} 
  The loop product on $H_*\Lambda$ descends to $\ol H_*\Lambda$ and the loop coproduct on $H_*(\Lambda,\Lambda_0)$ extends to $\ol H_*\Lambda$ (canonically if we have $H_1M=0$). Each such extension $\lambda$ defines together with the loop product $\mu$ the structure of a commutative cocommutative unital infinitesimal anti-symmetric
bialgebra on $\ol \H_*\Lambda=\ol H_{*+n}\Lambda$. In particular, the following relation holds 
$$
\lambda\mu = (\mu\otimes \one)(\one\otimes\lambda) + (1\otimes\one)(\lambda\otimes \one) - (\mu\otimes\mu)(\one\otimes\lambda 1 \otimes\one),
$$ 
where we denote $\one$ the identity map and $1$ the unit for the product $\mu$. 
\end{theorem}

We refer to~\cite{CHO-reduced,CHO-algebra} for the definition of a
commutative cocommutative unital infinitesimal anti-symmetric bialgebra. 
The extensions of the coproduct depend on the choice of auxiliary data consisting of a Morse function on $M$ with a unique maximum, a Morse-Smale gradient vector field, and a vector field with nondegenerate zeroes located away from the $(n-1)$-skeleton. We discuss this dependence in~\S\ref{sec:reduced}. The coproduct is independent on all choices when $H_1M=0$ (Proposition~\ref{prop:coproduct-choices}), and in that case it also vanishes on the unit $1$ (Corollary~\ref{cor:lambda1=0}), so that the above relation becomes the \emph{unital infinitesimal relation}
$$
\lambda\mu = (\mu\otimes \one)(\one\otimes\lambda) + (1\otimes\one)(\lambda\otimes \one).
$$ 

\smallskip

{\bf Structure of the paper. }
In~\S\ref{sec:A2+loop} we define the notion of a special
$A_2^+$-structure and prove that the Morse complex of the energy
functional on loop space carries such a structure. In particular, this
includes a Morse theoretic definition of the loop coproduct.  

In~\S\ref{sec:A2+symp} we construct a special $A_2^+$-structure on the
chain complexes underlying symplectic homology of $D^*M$. 

In~\S\ref{sec:reduced} we discuss extensions of the loop coproduct to reduced homology, and also the dependence of these extensions on choices.

In~\S\ref{sec:symp-loop-iso} we revisit the Viterbo isomorphism between
symplectic homology of the cotangent bundle and loop space homology.
We show that the Abbondandolo-Schwarz map 
$$
\Psi:SH_*(D^*M)\stackrel\simeq\to H_*(\Lambda;\eta),
$$ 
which was originally constructed using asymptotically quadratic
Hamiltonians and as such did not preserve the natural filtrations (at
the source by the non-Hamiltonian action, and at the target by the
square root of the energy), can be made to preserve these filtrations
when implemented for the linear Hamiltonians used in the definition of
symplectic homology. As such, $\Psi$ becomes an isomorphism at chain
level. This uses a length vs.~action estimate inspired
by~\cite{Cieliebak-Latschev}.  

In~\S\ref{sec:pop-GH-iso} we prove that the isomorphism $\Psi$ intertwines
the special $A_2^+$-structures of~\S\ref{sec:A2+loop} and~\S\ref{sec:A2+symp},
which together with algebraic results from~\cite{CO-cones} yields Theorem~\ref{thm:main2}.
Our proof uses homotopies in certain compactified moduli spaces of punctured annuli. In Remark~\ref{rmk:PhiPsiAS} we discuss some related open questions involving the two chain level isomorphisms between Morse and Floer complexes constructed by Abbondandolo-Schwarz in~\cite{AS,AS-Legendre}.

In~\S\ref{sec:other-coproducts} we restrict to positive symplectic
homology on the symplectic side, respectively to loop homology rel
constant loops on the topological side. We relate there the coproduct
$\lambda^F$ resulting from~\S\ref{sec:A2+symp} to the other secondary
coproducts mentioned above, thus proving Theorem~\ref{thm:main1}. 
In particular, this implies that the secondary coproduct defined by
Abbondandolo and Schwarz in~\cite{AS-product-structures} corresponds
under the isomorphism $\Psi$ (restricted to the positive range) to the
loop coproduct. For completeness, we also give a direct proof of this
last fact in~\S\ref{sec:symplectic-Morse-coproducts}.  

In~\S\ref{sec:spheres} we compute the extended coproducts on reduced loop homology of odd-dimensional spheres $S^n$. For $n\ge 3$ these coproducts are canonical, but for $n=1$ one sees explicitly the dependence on the choice of auxiliary data discussed in~\S\ref{sec:reduced}. 

The Appendix contains a complete discussion of local systems on free
loop spaces and their behaviour with respect to the loop product and
coproduct. Local systems are unavoidable in the context of manifolds which are not orientable~\cite{Laudenbach-CS,Abouzaid-cotangent}, and
also in the context of the correspondence between symplectic homology
of $D^*M$ and loop space homology of
$M$~\cite{Kragh,Abouzaid-cotangent,AS-corrigendum}. They
also proved useful in applications~\cite{Albers-Frauenfelder-Oancea}.  
 
{\bf Acknowledgements. } 
The first author thanks Stanford University, Institut Mittag--Leffler,
and the Institute for Advanced Study for their hospitality over the
duration of this project. 
The second author is grateful for support over the years from the Institute for Advanced Study, especially from Helmut Hofer, and in particular during the academic year 2019-2020.
The third author thanks Helmut Hofer and the Institute for Advanced Study for their hospitality over the
duration of this project. The third author was partially funded by
the Agence Nationale de la Recherche, France under the grants
MICROLOCAL ANR-15-CE40-0007 and ENUMGEOM ANR-18-CE40-0009. In its late stages, this work has also benefited from support provided to the third author by the University of Strasbourg Institute for Advanced Study (USIAS) for a Fellowship, within the French national programme "Investment for the future" (IdEx-Unistra).

\section{$A_2^+$-structure for loop space homology}\label{sec:A2+loop}

\subsection{$A_2^+$-algebras}

In this subsection we recall from~\cite{CO-cones} the definition and
basic properties of $A_2^+$-algebras. We will restrict to the case of
{\em special} $A_2^+$-algebras which suffices for our purposes.

Let $R$ be a commutative ring
with unit, and $(\cA,\p)$ a differential graded $R$-module.
Let $\cA^\vee_*=\mathrm{Hom}_R(\cA_{-*},R)$ be its graded dual, and
$\ev :\cA^\vee\otimes \cA\to R$ the canonical evaluation map.  
We denote
$$
  \boldtau:\cA\otimes\cA\to\cA\otimes\cA,\qquad
  a\otimes b\mapsto(-1)^{\deg a\deg b}b\otimes a.
$$
   
\begin{definition}\label{def:A2+}
A \emph{special $A_2^+$-structure} on $(\cA,\p)$ consists of the following
$R$-linear maps: 
\begin{itemize}
\item \emph{the continuation quadratic vector} $c_0:R\to \cA\otimes \cA$, of degree $0$;
\item \emph{the secondary continuation quadratic vector} $Q_0:R\to\cA\otimes \cA$, of degree $1$; 
\item \emph{the product} $\mu:\cA\otimes \cA\to \cA$, of degree $0$;
\item \emph{the secondary coproduct} $\lambda:\cA\to \cA\otimes \cA$, of degree $1$.
\end{itemize}
The continuation quadratic vector $c_0$ gives rise to the {\em continuation map} 
$$
c:=(\ev\otimes 1)(1\otimes c_0) 
:\cA^\vee\to\cA.
$$
These maps are subject to the following conditions: 
\begin{enumerate}
\item $c_0$ is a cycle; 
\item $c_0$ is symmetric up to a homotopy given by $Q_0$, i.e.
$$
\boldtau c_0-c_0=[\p,Q_0];
$$
\item $\mu$ is a chain map; 
\item $\lambda$ satisfies the relation 
$$
[\p, \lambda]=(\mu\otimes 1)(1\otimes c_0) - (1\otimes \mu)(\boldtau c_0\otimes 1);
$$ 
\item Denoting $\lambda=\lambda_{c_0,c_0}$ and 
\begin{align*}
\lambda_{\boldtau c_0,\boldtau c_0} & = \lambda_{c_0,c_0} + (\mu\otimes 1)(1\otimes Q_0) - (1\otimes \mu)(\boldtau Q_0\otimes 1),\\
\lambda_{c_0,\boldtau c_0} & = \lambda_{c_0,c_0} + (\mu\otimes 1)(1\otimes Q_0),\\
\lambda_{\boldtau c_0,c_0} & = \lambda_{c_0,c_0} - (1\otimes \mu)(\boldtau Q_0\otimes 1),
\end{align*}
we require that 
$$
(\lambda_{c_0,\boldtau c_0} \otimes 1) \boldtau c_0 = (\lambda_{\boldtau c_0,c_0}\otimes 1)\boldtau c_0=(\lambda_{\boldtau c_0,\boldtau c_0}\otimes 1) c_0 =0.
$$
\end{enumerate}
\end{definition}

We call the tuple $(\cA,\p,c_0,Q_0,\mu,\lambda)$ a {\em special
  $A_2^+$-algebra}. 

\begin{proposition}[\cite{CO-cones}]\label{prop:TQFT+}
Let $(\cA,\p,c_0,Q_0,\mu,\lambda)$ be a special $A_2^+$-algebra. Then
the cone $Cone(c)$ carries a canonical product $\boldsymbol{\mu}$ which commutes with
the boundary operator and thus descends to homology. Moreover, in the
long exact sequence
\begin{equation*}
\xymatrix
@C=20pt
{
    \ar[r]& H^{-*}(\cA) \ar[r]^-{c_*} & H_*(\cA) \ar[r]^-\iota &
    H_*(Cone(c)) \ar[r]^-\pi & H^{1-*}(\cA)\ar[r]&  
}
\end{equation*}
the map $\iota$ is a ring map with respect to $\mu$ and $\boldsymbol{\mu}$.
\end{proposition}

Next we discuss morphisms (again only a special case). 

\begin{definition}\label{def:A2+mor}
A \emph{special morphism of special $A_2^+$-algebras} $(\Psi,\Gamma,\Theta):\cA\to\cA'$
consists of the following $R$-linear maps:  

(i) a degree $0$ chain map $\Psi:\cA\to\cA'$ satisfying
$$
  c_0'=(\Psi\otimes\Psi)c_0,\qquad Q'_0=(\Psi\otimes\Psi)Q_0;
$$
  (ii) a degree $1$ bilinear map $\Gamma:\cA\otimes\cA\to\cA'$ satisfying
$$
  [\p,\Gamma] = \mu'(\Psi\otimes\Psi)-\Psi\mu;
$$
(iii) a degree $2$ bilinear map $\Theta:\cA\to\cA'\otimes\cA'$
  satisfying $\Theta c=0$ and
$$
  [\p,\Theta] = \lambda'\Psi-(\Psi\otimes\Psi)\lambda 
  - (\Gamma\otimes\Psi)(1\otimes c_0) + (\Psi\otimes\Gamma)(\boldtau c_0\otimes 1).
$$
\end{definition}

\begin{proposition}[\cite{CO-cones}]\label{prop:A2+toA2-mor}
Let $(\Psi,\Gamma,\Theta):\cA\to\cA'$ be a special morphism of special
$A_2^+$-algebras such that the induced map $\Psi_*:H_*(\cA)\to H_*(\cA')$
is an isomorphism. Then there exists a canonical ring isomorphism
$H_*(Cone(c))\cong H_*(Cone(c'))$ such that the following
diagram commutes:
\begin{equation*}
\xymatrix
@C=20pt
{
    \cdots H^{-*}(\cA) \ar[r]^-{c_*} & H_*(\cA) \ar[d]^\cong_{\Psi_*} \ar[r]^-\iota &
    H_*(Cone(c)) \ar[d]^\cong\ar[r]^-\pi & H^{1-*}(\cA)\cdots \\
    \cdots H^{-*}(\cA') \ar[u]_\cong^{\Psi^*} \ar[r]^-{c_*'} & H_*(\cA') \ar[r]^-{\iota'} &
    H_*(Cone(c')) \ar[r]^-{\pi'} & H^{1-*}(\cA')\cdots \ar[u]_\cong^{\Psi^*} 
}
\end{equation*}
\end{proposition}

\begin{remark}
The word ``special'' refers to the conditions~(5) in
Definition~\ref{def:A2+} and $\Theta c=0$ in Definition~\ref{def:A2+mor}.
These conditions are imposed in order to simplify the algebra in~\cite{CO-cones}. 
These conditions, as well as the hypothesis
$n\neq 2$ from Theorems~\ref{thm:main2} and~\ref{thm:cont-loop-iso}, can be removed by upgrading the theory of $A_2^+$-structure to a theory of $A_3^+$-structures, which would include arity 3 operations.
\end{remark}

\begin{remark}
(a) The conditions in Definition~\ref{def:A2+} imply (Proposition~\ref{prop:TQFT+})
that $\mu$ and $\lambda$ induce a product $\boldsymbol{\mu}$ on $H_*(Cone(c))$.
Associativity of $\boldsymbol{\mu}$ requires further compatibilities between
$\mu$ and $\lambda$, one of them being the ``unital infinitesimal relation''~\cite{CHO-PD,CO-cones,CHO-reduced}.\\
(b) The conditions in Definition~\ref{def:A2+} imply that $\lambda$
descends to the ``reduced homology'' $H_*(\cA/\im c)$ and the map $\pi$ in~\eqref{eq:les-Lambda} intertwines it with a naturally defined coproduct $\boldsymbol{\lambda}$ on the cone; see~\cite{CHO-reduced} for further details. 
\end{remark}

\subsection{$A_2^+$-structure on the Morse complex of the loop space}\label{ss:A2+loop}

Let now $M$ be a closed connected manifold of dimension $n$. 
For simplicity, we assume that $M$ is oriented and we use untwisted
coefficients in a commutative ring $R$; the necessary adjustments with
twisted coefficients are explained in Appendix~\ref{sec:local-systems}.
We denote 
$$
   S^1:=\R/\Z\quad\text{and}\quad\Lambda :=W^{1,2}(S^1,M). 
$$
Our goal in this subsection is to construct an $A_2^+$-structure on
the Morse complex of $\Lambda$. The analysis underlying the Morse
complex is identical to the one in~\cite{AS,Cohen-Schwarz} and we
refer to there for details. 

{\bf The Morse complex. }
Consider a smooth Lagrangian $L:S^1\times TM\to\R$ which
outside a compact set has the form $L(t,q,v)=\frac12|v|^2-V_\infty(t,q)$ for
a smooth potential $V_\infty:S^1\times M\to\R$. 
It induces an action functional
$$
   S_L:\Lambda\to\R,\qquad q\mapsto\int_0^1L(t,q,\dot q)dt, 
$$
which we can assume to be a Morse function. This functional is continuously differentiable and twice G\^ateaux-differentiable  on the space of loops of class $W^{1,2}$, but in general it is not smooth (unless $L$ is everywhere quadratic)~\cite{AS-pseudo-gradient}. Abbondandolo and Schwarz proved in~\cite{AS-pseudo-gradient} that it admits a negative pseudo-gradient vector field
which is smooth and Morse--Smale. The latter condition means that for all $a,b\in\Crit(S_L)$ the unstable
manifold $W^-(a)$ and the stable manifold $W^+(b)$ with respect to
the negative pseudo-gradient 
intersect transversely in a manifold of dimension
$\ind(a)-\ind(b)$, where $\ind(a)$ denotes the Morse index with
respect to $S_L$. 

Let $(MC_*,\p)$ be the Morse complex of $S_L$ with $R$-coefficients. It is 
graded by the Morse index and the differential is given by
$$
   \p:MC_*\to MC_{*-1},\qquad a\mapsto\sum_{\ind(b)=\ind(a)-1}\#\MM(a;b)\,b,
$$
where $\#\MM(a;b)$ denotes the signed count of points in the oriented
$0$-dimensional manifold
$$
   \MM(a;b) := \bigl(W^-(a)\cap W^+(b)\bigr)/\R.
$$
Then $\p\circ\p=0$ and its homology $MH_*$ is isomorphic to the
singular homology $H_*\Lambda$. We will assume in addition that
near the zero section $L(t,q,v)=\frac12|v|^2-V(q)$ for a
time-independent Morse function $V:M\to\R$ such that all nonconstant
critical points of $S_L$ have action larger than $-\min V$. Then the
constant critical points define a subcomplex $MC_*^{=0}$ of $MC_*$
which agrees with the Morse cochain complex of $V$ on $M$, with
degrees of $q\in\Crit(V)$ related by $\ind(q)=n-\ind_V(q)$. 

We assume that $L|_M$ has a unique minimum $q_0$ and a unique maximum $q_\Max$. We denote by
$\chi=\chi(M)$ the Euler characteristic of $M$ and
define the $R$-linear map $c_0:R\to MC_0\otimes MC_0$ by
$$
   c_0(1) := \chi\,q_0\otimes q_0.
$$
The element $c_0$ is clearly a cycle and we actually have $\boldtau c_0=c_0$. Note however that the secondary continuation element $Q_0$ that we construct in the sequel may be nonzero. See also~\S\ref{sec:reduced}.

\begin{remark}
The operation $c_0$ can also be defined by a count of pairs of semi-infinite gradient lines with common starting point.  
\end{remark}  
   
{\bf The product $\mu$. }
For a path $\alpha:[0,1]\to M$ and $\tau\in[0,1]$ we define the restrictions
$\alpha|_{[0,\tau]},\,\alpha|_{[\tau,1]}:[0,1]\to M$ by
\begin{equation}\label{eq:restr}
   \alpha|_{[0,\tau]}(t) := \alpha(\tau t),\qquad
   \alpha|_{[\tau,1]}(t) := \alpha\bigl(\tau+(1-\tau)t\bigr).
\end{equation}
For paths $\alpha,\beta:[0,1]\to M$ with $\alpha(1)=\beta(0)$ we define
their concatenation $\alpha\#\beta:[0,1]\to M$ by
$$
   \alpha\#\beta(t) := \begin{cases}
      \alpha(2t) & t\leq 1/2, \\
      \beta(2t-1) & t\geq 1/2.
   \end{cases}
$$
For $a,b,c\in\Crit(S_L)$ set
\begin{align*}
   \MM(a,b;c) := \bigl\{(\alpha,\beta,\gamma)\in
   W^-(a)\times W^-(b)\times W^+(c)\mid \gamma=\alpha\#\beta\bigr\},
\end{align*}
which is a transversely cut out manifold of dimension 
$$
   \dim\MM(a,b;c) = \ind(a)+\ind(b)-\ind(c)-n. 
$$
If its dimension equals zero this manifold is compact and defines a map
$$
   \mu:(MC\otimes MC)_*\to MC_{*-n},\qquad
   a\otimes b\mapsto\sum_{c}\#\MM(a,b;c)\,c.
$$
If the dimension equals $1$ it can be compactified to a compact
$1$-dimensional manifold with boundary 
\begin{align*}
   \p\MM(a,b;c) 
   &= \coprod_{\ind(a')=\ind(a)-1}\MM(a;a')\times\MM(a',b;c) \cr
   &\amalg\coprod_{\ind(b')=\ind(b)-1}\MM(b;b')\times\MM(a,b';c) \cr
   &\amalg\coprod_{\ind(c')=\ind(c)+1}\MM(a,b;c')\times\MM(c';c).
\end{align*}
corresponding to broken pseudo-gradient lines. So we have
\begin{equation}\label{eq:mu-chain-map-loop}
   \mu(\p\otimes\id+\id\otimes\p) - \p\mu = 0,
\end{equation}
i.e.~$\mu$ satisfies condition~(3) in Definition~\ref{def:A2+}. The induced map on homology
$$
   \mu_*:(MH\otimes MH)_*\to MH_{*-n}
$$
agrees with the loop product under the canonical isomorphism
$MH_*\cong H_*\Lambda$.
The loop product is associative, and this is reflected at chain level by the fact that $\mu$ is associative up to chain homotopy. 

The critical point $q_\Max$ is a cycle which is a two-sided unit for $\mu$ up to homotopy. Moreover, the subcomplex of constant loops $MC_*^{=0}\subset MC_*$ is stable under $\mu$ and we can choose the Morse data such that $q_\Max$ is a strict unit for the restriction of $\mu$ to $MC_*^{=0}$. 

{\bf The coproduct $\lambda$. }
We fix a small vector field $v$ on $M$ with nondegenerate
zeroes such that the only periodic orbits of $v$ with period $\leq 1$
are its zeroes. (The last property can be arranged, e.g., by choosing
$v$ gradient-like near its critical points; then the periods of
nonconstant periodic orbits are uniformly bounded from below by a
constant $c>0$, so $v/2c$ has the desired property.) Denote by 
$$f_t:M\stackrel{\cong}\longrightarrow M, \qquad t\in\R$$ 
the flow of $v$, i.e. the solution of the ordinary differential equation $\frac{d}{dt}f_t=v\circ f_t$. 
It follows that the only fixed points of $f=f_1$ are
the zeroes of $v$, each zero $q$ is nondegenerate as a fixed point, and
$$
   \sign\det(T_qf-\id) = \ind_v(q),
$$
where $\ind_v(q)$ is the index of $q$ as a zero of $v$. The map 
$$
   f\times\id:M\to M\times M,\qquad q\mapsto\bigl(f(q),q\bigr)   
$$
is transverse to the diagonal $\Delta\subset M\times M$ and
$$
   (f\times\id)^{-1}(\Delta) = \{q\in M\mid f(q)=q\} = \Fix(f). 
$$
Since for $q\in\Fix(f)$ the map $T_qM\to T_qM\times T_qM$,
$w\mapsto\bigl((T_qf-\id)w,0\bigr)$ fills up the complement to $T_{(q,q)}\Delta$, 
the induced orientation on $\Fix(f)=(f\times\id)^{-1}(\Delta)$ endows
$q$ with the sign $\ind_v(q)$.  

\begin{remark}\label{rem:ft}
Alternatively, we could use the exponential map of some Riemannian
metric to define a map $M\to M$ by $q\mapsto\exp_qtv(q)$. Although
this map differs from $f_t$ above, for $v$ sufficiently small it
shares its preceding properties and could be used in place of $f_t$. 
\end{remark}
 
Consider now a generic family of vector fields $v^\tau$, $\tau\in[0,1]$ which interpolates between $v^0=v$ and $v^1=-v$. We denote $f_t^\tau$, $t\in\R$ the flow of $v^\tau$ and $f^\tau=f_1^\tau$. Note that, while $v$ and $-v$ have nondegenerate zeroes, this condition cannot be guaranteed for $v^\tau$. Genericity of the family means that the maps $f^0\times 1:M\to M\times M$, $f^1\times 1:M\to M\times M$, and $[0,1]\times M\to M\times M$, $(\tau,p)\mapsto (f^\tau(p),p)$, are transverse to the diagonal. 

For each $q\in M$ and $\tau\in[0,1]$ we denote the induced path from $q$ to $f^\tau(q)$ by
$$
   \pi_q^\tau:[0,1]\to M,\qquad \pi_q^\tau(t):=f_t^\tau(q),
$$
and the inverse path by
$$
   (\pi_q^\tau)^{-1}:[0,1]\to M,\qquad (\pi_q^\tau)^{-1}(t):=f_{1-t}^\tau(q).
$$
Recall from above the restriction and concatenation of paths. 
Now for $a,b,c\in\Crit(S_L)$ we set
\begin{align*}
   \MM^1(a;b,c) := \bigl\{&(\tau,\alpha,\beta,\gamma)\in[0,1]\times
   W^-(a)\times W^+(b)\times W^+(c)\mid \cr
   & \beta = \alpha^\tau_1,\ 
   \gamma = \alpha^\tau_2\bigr\}
\end{align*}
with
\begin{align*}
   \alpha^\tau_1(t) &:= \alpha|_{[0,\tau]}\#(\pi_{\alpha(0)}^\tau)^{-1}
   = \begin{cases}
      \alpha(2\tau t) & t\leq 1/2, \\
      f_{2-2t}^\tau\bigl(\alpha(0)\bigr) & t\geq 1/2, 
   \end{cases} \cr
   \alpha^\tau_2(t) &:= \pi_{\alpha(0)}^\tau\#\alpha|_{[\tau,1]}
   = \begin{cases}
      f_{2t}^\tau\bigl(\alpha(0)\bigr) & t\leq 1/2, \\
      \alpha\bigl(2\tau-1+(2-2\tau)t\bigr) & t\geq 1/2.
   \end{cases}
\end{align*}
See Figure~\ref{fig:lambda-new}. 

\begin{figure} [ht]
\centering
\input{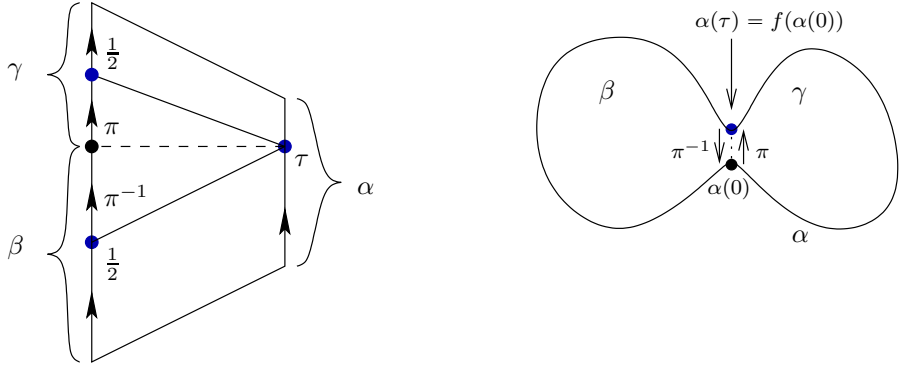}
\caption{Matching conditions for the definition of the 
loop coproduct via Morse chains.}
\label{fig:lambda-new}
\end{figure} 

Note that the matching conditions imply $\alpha(\tau)=f^\tau\circ\alpha(0)$. 
This is a codimension $n$ condition and, as the family $v^\tau$ is generic, $\MM^1(a;b,c)$ is a transversely cut out manifold of dimension 
$$
   \dim\MM^1(a;b,c) = \ind(a)-\ind(b)-\ind(c)+1-n. 
$$
If its dimension equals zero this manifold is compact and defines a map
$$
   \lambda:MC_*\to(MC\otimes MC)_{*+1-n},\qquad
   a\mapsto\sum_{b,c}\#\MM^1(a;b,c)\,b\otimes c.
$$
If the dimension equals $1$ it can be compactified to a compact
$1$-dimensional manifold with boundary 
\begin{align*}
   \p\MM^1(a;b,c) 
   &= \coprod_{\ind(a')=\ind(a)-1}\MM(a;a')\times\MM^1(a';b,c) \cr
   &\amalg\coprod_{\ind(b')=\ind(b)+1}\MM^1(a;b',c)\times\MM(b';b) \cr
   &\amalg\coprod_{\ind(c')=\ind(c)+1}\MM^1(a;b,c')\times\MM(c';c) \cr
   &\amalg\MM^1_{\tau=1}(a;b,c) \amalg\MM^1_{\tau=0}(a;b,c).
\end{align*}
Here the first three terms correspond to broken pseudo-gradient lines and the
last two terms to the intersection of $\MM^1(a;b,c)$ with the sets
$\{\tau=1\}$ and $\{\tau=0\}$, respectively. So we have
\begin{equation}\label{eq:lambda-chain-map-loop}
   (\p\otimes\id+\id\otimes\p)\lambda + \lambda\p = \lambda_1-\lambda_0,
\end{equation}
where for $i=0,1$ we set
$$
   \lambda_i:MC_*\to(MC\otimes MC)_{*-n},\qquad
   a\mapsto\sum_{b,c}\#\MM^1_{\tau=i}(a;b,c)\,b\otimes c.
$$
Let us look more closely at the map $\lambda_1$. For $\tau=1$ the
matching conditions in $\MM^1(a;b,c)$ imply that $\alpha(0)=q$ is a
fixed point of $f^1$, the time-one flow of $-v$, and $\gamma=q$ is the constant loop at $q$.
Assuming that $L|_M$ has a unique minimum $q_0$
and the fixed points of $f^1$ are in general position with respect to
the stable and unstable manifolds of $L|_M$, the condition $q\in W^+(c)$
is only satisfied for $c=q_0$. Thus $\MM^1_{\tau=1}(a;b,c)$ is
empty if $c\neq q_0$ and 
\begin{align*}
   \MM^1_{\tau=1}(a;b,q_0) \cong \coprod_{q\in\Fix(f^1)}\bigl\{(\alpha,\beta)\in
   W^-(a)\times W^+(b)\mid \beta = \alpha\# q\bigr\}.
\end{align*}
Choosing all fixed points of $f^1$ closely together, we can achieve that
the terms on the right hand side corresponding to different $q\in\Fix(f^1)$ 
are in canonical bijection to each other. By the discussion
before Remark~\ref{rem:ft}, the terms corresponding to a fixed point $q$ 
come with the sign $\ind_{-v}(q)$.
Since $\sum_q \ind_{-v}(q) = \chi$ and $W^-(q_0)=\{q_0\}$ we obtain 
\begin{align*}
  \#\MM^1_{\tau=1} & (a;b,q_0) \cr
  &= \sum_{q\in\Fix(f^1)}\#\bigl\{(\alpha,\beta)\in W^-(a)\times
  W^+(b)\mid \beta = \alpha\# q\bigr\} \cr
  &= \sum_{q\in\Fix(f^1)}(-1)^{\ind_{-v}(q)}\#\bigl\{(\alpha,\beta)\in W^-(a)\times
  W^+(b)\mid \beta = \alpha\# q_0\bigr\} \cr
  &= \chi\,\#\bigl\{(\alpha,\gamma,\beta)\in W^-(a)\times W^-(q_0)\times
  W^+(b)\mid \beta = \alpha\# \gamma\bigr\} \cr
  &= \chi\,\#\MM(a,q_0;b). 
\end{align*}
Since the last moduli space is the one in the definition of $\mu$, we conclude
$$
\lambda_1(a)=\chi\,\mu(a\otimes q_0)\otimes q_0, 
$$
or equivalently 
$$
\lambda_1=(\mu\otimes 1)(1\otimes c_0). 
$$
Similarly we have 
$$
\lambda_0(a)=\chi q_0\otimes \mu(q_0\otimes a),
$$
or equivalently 
$$
\lambda_0=(1\otimes \mu)(c_0\otimes 1) =(1\otimes \mu)(\boldtau c_0\otimes 1). 
$$
In conclusion we obtain condition (4) in Definition~\ref{def:A2+},
$$
[\p,\lambda]= (\mu\otimes 1)(1\otimes c_0) - (1\otimes \mu)(\boldtau c_0\otimes 1).
$$
We define the secondary continuation quadratic vector $Q_0$ by 
$$
Q_0=-\lambda(q_\Max).
$$
Condition~(2), i.e. $\boldtau c_0-c_0=[\p,Q_0]$, follows by inserting $q_\Max$ into the relation for $[\p,\lambda]$ and using that $q_\Max$ is a strict two-sided unit for $\mu$ on $MC_*^{=0}$. This is an instance of \emph{unital $A_2^+$-structure}~\cite{CO-cones}. Note that $Q_0\in MC_*^{=0}\otimes MC_*^{=0}$ for energy reasons. 

We now prove condition (5) in Definition~\ref{def:A2+}. For $n=1$ it holds because $\chi=0$, hence $c=0$. We therefore assume w.l.o.g. $n\ge 2$ and give the proof in two steps. 
\begin{enumerate}
\item We first prove $\lambda c=0$.
This follows from $\lambda(q_0)=0$, which is seen as follows.
The coefficient $\la\lambda(q_0),x\otimes y\ra$
can only be nonzero if $x,y$ are critical points of $K$. Since
$\lambda$ has degree $1-n$, we must have
$1-n = \ind(x)+\ind(y)-\ind(q_0) = n-\ind_V(x)-\ind_V(y)+\ind_V(q_0)$,
hence $\ind_V(x)+\ind_V(y)=\ind_V(q_0)+2n-1=3n-1$. Since
$\ind_V(x)+\ind_V(y)\leq 2n$, this is impossible for $n\ge 2$. 
\item We now show that $(1\otimes \mu\otimes 1)(\boldsymbol{a}\otimes\boldsymbol{b})=0$, where $(\boldsymbol{a},\boldsymbol{b})=(\boldtau Q_0,c_0), (\boldtau c_0,Q_0), (\boldtau Q_0,\boldtau c_0), (c_0,\boldtau Q_0)$. We identify $MC_*^{=0}(L)$ with the Morse cochain complex $MC^{n-*}(V)$. The cohomological index of $Q_0$ is $2n-1$, so its components must have degrees $n-1$ and $n$. If $n\ge 2$ these degrees are both positive, and therefore any component of $Q_0$ is killed by multiplication with $q_0$ because the latter has cohomological index $n$.    
\end{enumerate} 

In summary we have shown

\begin{proposition}\label{prop:A2+loop}
Each vector field $v$ on $M$ satisfying the preceding conditions gives
rise to a special $A_2^+$-structure $(c_0,Q_0,\mu,\lambda)$ on the Morse complex $MC_*$ of the
functional $S_L:\Lambda\to\R$. 
\end{proposition}

\begin{remark}
In the previous construction we used an interpolating family of vector fields $v^\tau$ such that $v^1=-v^0$. This choice is important because it ensures that the product on the Rabinowitz loop homology obtained from the $A_2^+$-structure via the cone construction is associative, and much more: in view of the isomorphism with the $A_2^+$-structure on symplectic homology proved in~\S\ref{sec:pop-GH-iso} and in view of~\cite{CHO-PD,CO-cones}, the resulting product fits into a graded Frobenius algebra structure on $\wh H_*\Lambda$. 

While the construction of an $A_2^+$-structure would have worked with \emph{any} choice of nondegenerate vector fields $v^0$ and $v^1$ at the endpoints of the parametrizing interval, the necessity of the condition $v^1=-v^0$, which ensures these fine properties of the product, would become visible at chain level within a theory of $A_3^+$-structures. The development of such a theory is a matter for further study. 
\end{remark}

\begin{remark}
The description of $\MM^1_{\tau=1}(a;b,q_0)$ and $\MM^1_{\tau=0}(a;q_0,c)$
above implies that $\lambda_1,\lambda_0:MC_*\to(MC\otimes MC)_{*-n}$ are
chain maps. By equation~\eqref{eq:lambda-chain-map-loop} they are
chain homotopic, hence they induce the same ``primary" coproduct
$$
   [\lambda_0]=[\lambda_1]:MH_*\to(MH\otimes MH)_{*-n}
$$
and the preceding discussion recovers~\cite[Lemma 5.1]{AS-product-structures}.
\end{remark}

\begin{remark}\label{rem:GH-alternative}
Alternatively, we could define the loop coproduct using the spaces 
\begin{align*}
   \wt\MM^1(a;b,c) := \bigl\{&(\tau,\alpha,\beta,\gamma)\in[0,1]\times
   W^-(a)\times W^+(b)\times W^+(c)\mid \cr
   & \beta(t) = (f_t^\tau)^{-1}\circ\alpha(\tau t),\cr 
   & \gamma(t) = (f_{1-t}^\tau)^{-1}\circ\alpha(\tau+(1-\tau)t)\bigr\}.
\end{align*}
Again the matching conditions imply $\alpha(\tau)=f^\tau\circ\alpha(0)$,
and $\wt\MM^1(a;b,c)$ is a transversely cut out manifold of dimension 
$\ind(a)-\ind(b)-\ind(c)+1-n$ whose rigid counts define a map
$$
   \wt\lambda:MC_*\to(MC\otimes MC)_{*+1-n},\qquad
   a\mapsto\sum_{b,c}\#\wt\MM^1(a;b,c)\,b\otimes c.
$$
A discussion analogous to that for $\lambda$ shows that
Proposition~\ref{prop:A2+loop} also holds with $\wt\lambda$ in place
of $\lambda$. 
The obvious homotopies between the loops $\alpha_1^\tau$ and
$\alpha_2^\tau$ in the definition of $\lambda$ and the loops 
$t\mapsto (f_t^\tau)^{-1}\circ\alpha(\tau t)$ and
$t\mapsto (f_{1-t}^\tau)^{-1}\circ\alpha(\tau+(1-\tau)t)$ in the definition of
$\wt\lambda$ provide a
special morphism between $(c_0,Q_0,\mu,\lambda)$ and $(c_0,\wt Q_0,\mu,\wt\lambda)$, where $\wt Q_0=-\wt \lambda (q_\Max)$. 
We will use the restriction of the map $\wt\lambda$ to Morse chains
modulo constants in the proof of Proposition~\ref{prop:iso-ol-FH-MH}.
\end{remark}

\section{$A_2^+$-structure for symplectic homology}\label{sec:A2+symp}

As in the previous section, let $M$ be a closed oriented manifold. We
pick a Riemannian metric on $M$ and denote by $S^*M\subset D^*M\subset T^*M$ its
unit sphere resp.~unit disc cotangent bundle. The latter is a Liouville domain whose
completion is $T^*M$. Its {\em symplectic homology} $SH_*(D^*M)$ is defined
as the direct limit of the Floer homologies $FH_*(K)$ over
Hamiltonians $K:S^1\times T^*M\to\R$ that are negative on $D^*M$ and
linear outside a compact see; see~\cite{CO} for general background on
symplectic homology. The goal of this section is to construct a special
$A_2^+$-structure on the chain complex underlying symplectic homology. 

\subsection{The continuation map $c^F$}\label{sec:cont-map}

Recall from~\cite{CO} that for Hamiltonians $H\leq K$ we have a continuation map
$c_{H,K}:FC_*(H)\to FC_*(K)$, defined by counting Floer cylinders for
an $s$-dependent Hamiltonian $\wh H(s,\cdot)$ which agrees with $K$
for small $s$, with $H$ for large $s$, and which satisfies $\p_s\wh
H\leq 0$. In this subsection we will describe the continuation map
$$
   c^F=c_{-K,K}:FC_*(-K)\to FC_*(K)
$$
for a smooth Hamiltonian $K:T^*M\to\R$ of the form
$$
   K(q,p) = k(|p|) + V(q)
$$
for a convex function $k$ with $k(0)=0$ and $k(r)=\mu r$ for large
$r$, with $\mu>0$ not in the length spectrum, and a potential $V:M\to\R$ which has a unique maximum $q_0$ and a unique minimum $q_\Max$.
For $1$-periodic orbits $x$ of $-K$ and $y$ of $K$, the coefficient
$\la c^Fx,y\ra$ is given by the count of solutions $u:\R\times S^1\to T^*M$
of the Floer equation
\begin{equation}\label{eq:c-Floer}
   \p_su+J(u)\bigl(\p_tu-\phi(s)X_K(u)\bigr)=0
\end{equation}
converging to $x$ as $s\to+\infty$ and to $y$ as $s\to-\infty$. Here
$\phi:\R\to[-1,1]$ is a nonincreasing smooth function which equals $1$
for small $s$ and $-1$ for large $s$. For action reasons the
coefficient can only be nonzero if $x,y$ are constant solutions
corresponding to critical points of $V$ on the zero section $M\subset
T^*M$, in which case the solutions $u$ are $t$-independent and the
Floer equation becomes the Morse equation
\begin{equation}\label{eq:c-Morse}
   \p_su+\phi(s)\nabla K(u)=0. 
\end{equation}
The Fredholm index of this problem is
$$
   \CZ_{-K}(x)-\CZ_K(y) = -\CZ_K(x)-\CZ_K(y) =
   \ind_V(x)+\ind_V(y)-2n\leq 0,
$$
with equality iff $x=y=q_0$ for the maximum $q_0\in M$ of $V$. 
On the other hand, solutions $u$ of~\eqref{eq:c-Morse} are in one-to-one
correspondence to points $u(0)\in W_x^-\cap W_y^-$, where $W_x^-$
denotes the stable manifold of $x$ with respect to $\nabla V$. 
This shows that the Fredholm problem given by~\eqref{eq:c-Floer}
resp.~\eqref{eq:c-Morse} is degenerate.

To perturb it, we denote by $\Skel_k(V)\subset M$ the $k$-skeleton, i.e. the
union of the descending manifolds $W_x^-$ of critical points of index
$\leq k$. We pick a $1$-form $\eta$ on $M$ satisfying the following condition:
\begin{equation}\label{eq:c-eta}
   \text{All zeroes of $\eta$ are nondegenerate and lie in $M\setminus\Skel_{n-1}(V)$.}
\end{equation}
It gives rise to the flow
$$
   F^\eta_t:T^*M\to T^*M,\qquad F^\eta_t(q,p):=\bigl(q,p+t\eta(q)\bigr)
$$
generated by the vector field $\wh\eta$ on $T^*M$,
$$
   \wh\eta(q,p) := \frac{d}{dt}\Bigl|_{t=0}F_t^\eta(q,p).
$$
We pick a compactly supported function $\rho:\R\to[0,\infty)$ with
$\int_\R\rho=1$ and perturb equations~\eqref{eq:c-Floer} and~\eqref{eq:c-Morse} to
\begin{equation}\label{eq:c-Floer-pert}
   \p_su+J(u)\bigl(\p_tu-\phi(s)X_K(u)\bigr) = \rho(s)\wh\eta(u).
\end{equation}
and
\begin{equation}\label{eq:c-Morse-pert}
   \p_su+\phi(s)\nabla K(u) = \rho(s)\wh\eta(u). 
\end{equation}
To understand solutions of the perturbed Morse equation, we choose
$\phi,\rho$ such that $\phi\equiv 0$ on $[-1,1]$ and $\supp(\rho)\subset[-1,1]$.
Then solutions $u$ of~\eqref{eq:c-Morse-pert} are in one-to-one
correspondence to points
$$
   u(1)\in W_x^-\cap F^\eta_1(W_y^-),
$$
where the intersection is taken in $T^*M$. By condition~\eqref{eq:c-eta}
this intersection is empty unless $x=y=q_0$, in which case
intersection points correspond to zeroes of $\eta$ and their signed
count equals the Euler characteristic $\chi$ of $M$. This shows that
the only nontrivial term in the continuation map $c^F:FC_*(-K)\to FC_*(K)$ is
$$
   c^F q_0=\chi q_0 
$$
and the corresponding quadratic vector $c_0^F$ is given by
$$
   c_0^F(1) = \chi q_0\otimes q_0.
$$
In particular, $c_0^F$ satisfies the closedness condition in
Definition~\ref{def:A2+}, and it is also symmetric $\boldtau c_0^F=c_0^F$. Note that this holds without any symmetry
assumptions on the data such as $\phi(-s)=-\phi(s)$ or $\rho(-s)=\rho(s)$.
Note also that, although the definition of $c_0^F$ on the chain level
requires the choice of a pair $(V,\eta)$ consisting of a Morse function
$V:M\to\R$ and a $1$-form $\eta$ on $M$ subject to
condition~\eqref{eq:c-eta}, the result does not depend on this choice. In contrast, the secondary continuation quadratic vector $Q_0^F$ which we construct below may depend on this choice. See also~\S\ref{sec:reduced}.

\subsection{The product $\mu^F$ and coproduct $\lambda^F$}\label{sec:cont-coproduct}

The pair-of-pants product $\mu^F:FC_*(K)\otimes FC_*(K)\to FC_*(2K)$
(of degree $-n$) counts maps from a pair-of-pants satisfying a Floer equation with
weights $1$ at the two positive punctures and weight $2$ at the
negative puncture. The definition is entirely analogous to the
one for the coproduct $\lambda^F$ given below, without the additional
parameter $\tau$. It is well-known that $\mu^F$ is a chain map which
is associative and graded commutative up to chain homotopy (see e.g.~\cite{AS2}), 
so condition (2) in Definition~\ref{def:A2+} holds. 

The critical point $q_\Max$ is a constant orbit and is a cycle which is a two-sided unit for $\mu^F$ up to homotopy. The subcomplex $FC_*^{=0}(K)\subset FC_*(K)$ generated by small action orbits is stable under $\mu^F$ and we can choose the auxiliary data such that $q_\Max$ is a strict unit for the restriction of $\mu^F$ to $FC_*^{=0}(K)$. 

In~\cite{CHO-reduced}
a secondary coproduct $\lambda$ 
is defined in terms of continuation maps on the reduced symplectic
homology of a large class of Weinstein domains which includes cotangent bundles. See also~\cite{CHO-PD, CO-cones}. In this
subsection we recall its definition for $D^*M$; we will call it the
{\em continuation coproduct} and denote it by $\lambda^F$. 

The definition in~\cite{CHO-PD, CO-cones} is described in terms of real
parameters $\lambda_1,\lambda_2<0<\mu_1,\mu_2,\mu$ satisfying
$\mu\leq\min(\lambda_1+\mu_2,\mu_1+\lambda_2)$. For simplicity, we
choose the parameters as $\lambda_1=\lambda_2=-\mu$ and
$\mu_1=\mu_2=2\mu$ for some $\mu>0$. We assume that $\mu$ and $2\mu$
do not belong to the action spectrum of $S^*M$. 

As before, we denote by $r=|p|$ the radial coordinate on $T^*M$. Let
$K=K_\mu$ be a convex smoothing of the Hamiltonian which is zero on
$D^*M$ and equals $r\mapsto\mu r$ outside $D^*M$. Then $2K=K_{2\mu}$
and $-K=K_{-\mu}$ are the corresponding Hamiltonians of slopes $2\mu$
and $-\mu$, respectively.

Let $\Sigma$ be the $3$-punctured Riemann sphere, where we view one
puncture as positive (input) and the other two as negative (outputs). We fix
cylindrical coordinates $(s,t)\in[0,\infty)\times S^1$ near the positive puncture
and $(s,t)\in(-\infty,0]\times S^1$ near the negative punctures. Consider a
$1$-form $\beta$ on $\Sigma$ which equals $B\,dt$ near the positive
puncture and $A_idt$ near the $i$-th negative puncture $(i=1,2$) for
some $A_i,B\in\R$. We say that {\em $\beta$ has weights $B,A_1,A_2$}.
We moreover require $d\beta\leq 0$, which is possible iff
$$
   A_1+A_2\geq B.
$$
We consider maps $u:\Sigma\to T^*M$ satisfying the perturbed
Cauchy-Riemann equation
$$
   (du-X_K\otimes\beta)^{0,1}=0. 
$$
Near the punctures this becomes the Floer equation for the Hamiltonians 
$BK$ and $A_iK$, respectively, and the algebraic count of such maps
defines a (primary) coproduct
$$
   FC_*(BK)\to FC_*(A_1K)\otimes FC_*(A_2K)
$$
which has degree $-n$ and decreases the Hamiltonian action. 

To define the secondary coproduct $\lambda^F$, we choose a
$1$-parameter family of $1$-forms $\beta_\tau$, $\tau\in(0,1)$, with
the following properties (see Figure~\ref{fig:cont-coproduct}):
\begin{itemize}
\item $d\beta_\tau\leq 0$ for all $\tau$;
\item $\beta_\tau$ equals $dt$ near the positive puncture and $+2dt$
  near each negative puncture, i.e., $\beta_\tau$ has weights $1,2,2$;
\item as $\tau\to 0$, $\beta_\tau$ equals $-dt$ on cylinders near the
  first negative puncture whose length tends to $\infty$, so that
  $\beta_0$ consists of a $1$-form on $\Sigma$ with weights $1,-1,2$
  and a $1$-form with weights $-1,2$ on an infinite cylinder attached
  at the first negative puncture;  
\item as $\tau\to 1$, $\beta_\tau$ equals $-dt$ on cylinders near the
  second negative puncture whose length tends to $\infty$, so that
  $\beta_1$ consists of a $1$-form on $\Sigma$ with weights $1,2,-1$
  and a $1$-form with weights $-1,2$ on an infinite cylinder attached
  at the second negative puncture.
\end{itemize}

\begin{figure}
\begin{center}
\includegraphics[width=.7\textwidth]{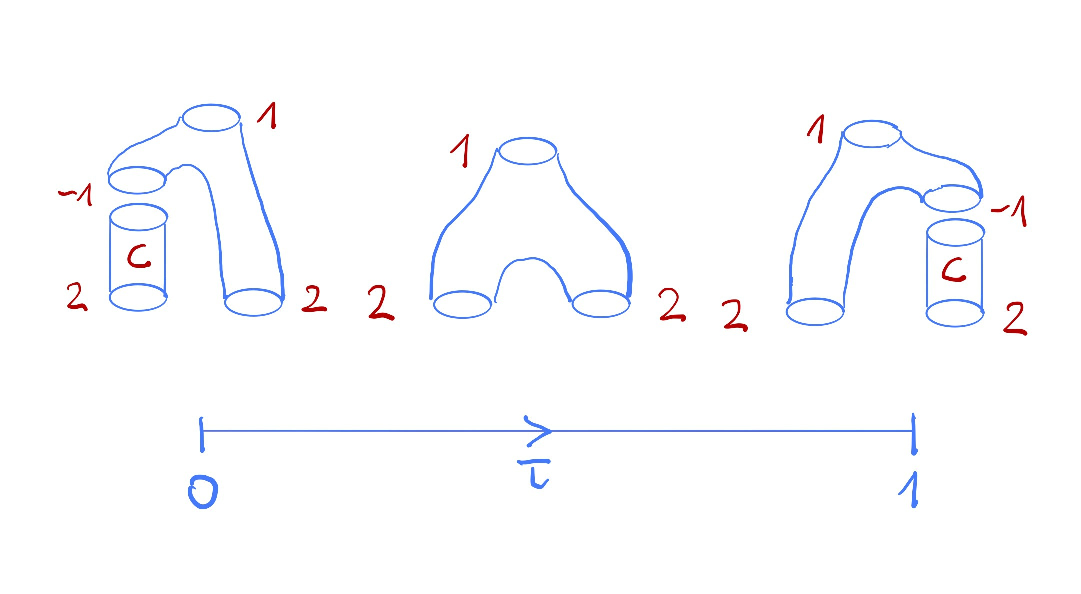}
\caption{The continuation coproduct $\lambda^F$}
\label{fig:cont-coproduct}
\end{center}
\end{figure}

Now we consider pairs $(\tau,u)$ where $\tau\in[0,1]$ and $u:\Sigma\to
T^*M$ satisfies the perturbed Cauchy-Riemann equation
$$
   (du-X_K\otimes\beta_\tau)^{0,1}=0. 
$$
The algebraic count of such pairs defines a (secondary) coproduct
$$
   \lambda^F:FC_*(K)\to FC_*(2K)\otimes FC_*(2K)
$$
which has degree $1-n$ and decreases the Hamiltonian action. 

Let us analyze the contributions from $\tau=0,1$. The algebraic count
of cylinders with weights $-1,2$ defines the {\em continuation map}
(of degree $0$) 
$$
   c^F=c_{-K,2K}:FC_*(-K)\to FC_*(2K). 
$$
As explained in the previous subsection, to define $c^F$ we perturb $K$
by a Morse function $V:M\to\R$ with a unique maximum at
$q_0$ and a unique minimum at $q_\Max$. Moreover, we choose a family of $1$-forms $\eta_\tau$ on $M$
such that $\eta_0=\eta_1$ satisfies condition~\eqref{eq:c-eta} (for all practical purposes one can think of $\eta_\tau$ as being constant).
Finally, we choose a family of compactly supported $1$-forms
$\alpha_\tau$ on $\Sigma$ which for $\tau=0,1$ agree with $\rho(s)ds$
supported in the split off cylinder, for a function
$\rho:\R\to[0,\infty)$ satisfying $\int_\R\rho=1$.
For example, we can take
$\alpha_\tau=(1-\tau)\alpha_0+\tau\alpha_1$ where $\alpha_0=\rho(s)ds$
supported on the first negative end, and $\alpha_1=\rho(s)ds$
supported on the second negative end (below the level where the
splitting happens at $\tau=0,1$).
With $\wh\eta_\tau$ the vector field on $T^*M$ corresponding to $\eta_\tau$, we replace
the Cauchy-Riemann equation in the definition of $\lambda^F$ by
\begin{equation}\label{eq:CR-tau-pert}
   (du-X_K\otimes\beta_\tau)^{0,1}=(\wh\eta_\tau\otimes\alpha_\tau)^{0,1}. 
\end{equation}
With these choices, it follows from the discussion in the previous
subsection that the only nontrivial terms in the continuation maps at
$\tau=0,1$ are $c^Fq_0=\chi q_0$, where $\chi$ is the Euler
characteristic of $M$.

As shown in Figure~\ref{fig:cont-coproduct}, the contribution at
$\tau=0$ consists of a pair-of-pants with one positive puncture of
weight $1$ and two negative punctures of weights $-1$ and $2$, with a
cylinder of weights $-1$ and $2$ attached at the first negative puncture.
We reinterpret this as a pair-of-pants with two positive punctures of
weights $1$ and one negative puncture of weight $2$, with a cylinder
with two negative punctures of weights $2$ and $1$ attached at the
first positive puncture. The preceding discussion shows that the count
of these configurations corresponds to the composition
$(1\otimes\mu^F)(\boldtau c_0^F\otimes 1)$. A similar discussion at
$\tau=1$ establishes that $\lambda^F$ satisfies condition~(4) in
Definition~\ref{def:A2+}. 

We define the secondary continuation quadratic vector $Q_0^F$ by 
$$
Q_0^F=-\lambda^F(q_\Max).
$$
Condition~(2), i.e. $\boldtau c_0^F-c_0^F=[\p,Q_0^F]$, follows by inserting $q_\Max$ into the relation for $[\p,\lambda^F]$ and using that $q_\Max$ is a strict two-sided unit for $\mu$ on $FC_*^{=0}(K)$. This is an instance of \emph{unital $A_2^+$-structure}~\cite{CO-cones}. Note that $Q_0^F\in FC_*^{=0}(K)\otimes FC_*^{=0}(K)$ for energy reasons. An inspection of the definition shows that $Q_0^F$ coincides with the secondary continuation element defined in~\S\ref{sec:secondary-cont-map} by interpolating between the perturbing $1$-form $\eta$ and its opposite $-\eta$. See also~\cite{CHO-reduced}.

It remains to prove condition (5) in Definition~\ref{def:A2+}. For $n=1$ it holds because $\chi=0$, so that $c_0^F=0$. We therefore assume w.l.o.g. $n\ge 2$ and, as in the Morse case, we prove condition~(5) in two steps. 
\begin{enumerate}
\item We first prove $\lambda^Fc^F=0$.
This follows from $\lambda^F(q_0)=0$, which is seen as follows.
For action reasons, the coefficient $\la\lambda^F(q_0),x\otimes y\ra$
can only be nonzero if $x,y$ are critical points of $K$. Since
$\lambda^F$ has degree $1-n$, we must have
$1-n = \CZ(x)+\CZ(y)-\CZ(q_0) = n-\ind_V(x)-\ind_V(y)+\ind_V(q_0)$,
hence $\ind_V(x)+\ind_V(y)=\ind_V(q_0)+2n-1=3n-1$. Since
$\ind_V(x)+\ind_V(y)\leq 2n$, this is impossible for $n\ge 2$. 
\item We are left to show that $(1\otimes \mu^F\otimes 1)(\boldsymbol{a}\otimes\boldsymbol{b})=0$, where $(\boldsymbol{a},\boldsymbol{b})=(\boldtau Q_0^F,c_0^F), (\boldtau c_0^F,Q_0^F), (\boldtau Q_0^F,\boldtau c_0^F), (c_0^F,\boldtau Q_0^F)$. We identify the Floer subcomplex $FC_*^{=0}(K)$ generated by orbits of small action with the Morse cochain complex $MC^{n-*}(V)$. The cohomological index of $Q_0^F$ is $2n-1$, so its components must have degrees $n-1$ and $n$. If $n\ge 2$ these degrees are both positive, and therefore any component of $Q_0$ is killed by multiplication with $q_0$ because the latter has cohomological index $n$.    
\end{enumerate}

In summary we have shown

\begin{proposition}\label{prop:A2+symphom}
The operations $c_0^F,Q_0^F,\mu^F,\lambda^F$ on the Floer chain complexes
$FC_*(K)$ resp.~$FC_*(2K)$ satisfy the relations of a special $A_2^+$-structure. 
\end{proposition}

The operations $c_0^F,Q_0^F\mu^F,\lambda^F$ are compatible with Floer
continuation maps between different Hamiltonians $H\leq K$. We will
refer to this structure as being the {\em special $A_2^+$-structure for
  symplectic homology $SH_*(D^*M)$}.  
 
\begin{remark} \label{rmk:about-eta0=eta1}
In the previous construction we imposed the condition $\eta_0=\eta_1$ at the endpoints of the family of $1$-forms $\eta_\tau$ for the same reason why we imposed $v^1=-v^0$ in the Morse case: this ensures that the product on Rabinowitz Floer homology obtained from the $A_2^+$-structure via the cone construction coincides with the product from~\cite{CHO-PD} and fits into a graded Frobenius algebra structure on $S\H_*(S^*M)$. 

The construction of an $A_2^+$-structure would have worked with any choice of interpolating family $\eta_\tau$ such that $\eta_0$ and $\eta_1$ satisfy~\eqref{eq:c-eta}. The necessity of the condition $\eta_1=\eta_0$ for this fine behavior of the product would become visible at chain level within a theory of $A_3^+$-structures. 
\end{remark}

\section{Reduced loop homology}\label{sec:reduced} 

This section expands material from~\cite[\S4]{CHO-reduced} in the particular case of cotangent bundles. 
We assume that $M$ is connected and orientable, and we work either with constant coefficients on the loop space, or with local coefficients $\eta$ obtained by transgressing the 2nd Stiefel-Whitney class. In each of these two cases we have a commutative diagram
$$
\xymatrix{
H^*(\Lambda) \ar[r]^-\eps \ar[d]& H_*(\Lambda) \\
H^0(M)\ar[r]^{\eps_0}& H_0(M), \ar[u]
}
\qquad \qquad 
\xymatrix{
H^*(\Lambda;\eta) \ar[r]^-\eps \ar[d]& H_*(\Lambda;\eta) \\
H^0(M)\ar[r]^{\eps_0}& H_0(M),\ar[u]
}
$$
where the vertical maps are restriction to, resp. inclusion of constant loops, and $\eps_0$ is induced by multiplication with the Euler characteristic $\chi$. From now on we omit from the notation the local system $\eta$. 

\begin{definition}
We define \emph{reduced loop homology, resp. cohomology},  
$$
\ol{H}_*(\Lambda)=\coker\, \eps,\qquad \ol{H}^*(\Lambda)=\ker \eps. 
$$
\end{definition}

In the sequel we restrict the discussion to reduced homology. Reduced cohomology features similar properties, with the roles of the product and coproduct being exchanged (as yet another instantiation of Poincar\'e duality for loop spaces~\cite{CHO-PD}). 

The behavior of reduced homology with respect to the product is very robust. The image of $\eps$ is an ideal in $H_*(\Lambda)$ (see for example~\cite{Tamanoi} or~\cite{CHO-PD}), and therefore the loop product canonically descends to reduced homology $\ol{H}_*(\Lambda)$. 

In contrast, the behavior of reduced homology with respect to the coproduct is very subtle. To describe it, the following variant of reduced loop homology arises naturally. 

\begin{definition}We define \emph{loop homology relative to $\chi\cdot$point} as 
$$
H_*(\Lambda,\chi\cdot\mathrm{point}) = H_*(C_*(\Lambda)/\chi C_*(\mathrm{point})).
$$
\end{definition}

A straightforward calculation shows that we have a canonical isomorphism 
$$
\ol{H}_*(\Lambda)\simeq H_*(\Lambda,\chi\cdot\textrm{point})
$$
whenever the map $\chi H_0(\textrm{point})\to H_0(\Lambda)$ is injective, see Appendix~\ref{sec:reduced-relpoint}. This is the case if $M$ is orientable and if we use a local system that is constant on the component of contractible loops, or if $\chi=0$, or if $R$ is 2-torsion. We place ourselves from now on in this setup, so that we do not need to distinguish between 
$\ol{H}_*(\Lambda)$ and $H_*(\Lambda,\chi\cdot\textrm{point})$. 

The loop coproduct is canonically defined on $H_*(\Lambda,\Lambda_0)$. We now explain that it always extends to $H_*(\Lambda,\chi\cdot\textrm{point})$ (and hence to $\ol{H}_*(\Lambda)$ under our assumptions). However, this extension is \emph{not} canonical. The extension depends on a choice of vector field with nondegenerate zeroes \emph{and} on the choice of a Morse function on $M$. We will completely describe the dependence of the extension on the choice of vector field, and give sufficient conditions for independence of the extension on the choice of Morse function. 

\subsection{Reduced symplectic homology} 

We work with symplectic homology of $D^*M$, our favorite model for loop space homology. Recalling notation from~\S\ref{sec:cont-map}, we fix the following \emph{continuation data}:  
\begin{itemize}
\item a Morse function $V:M\to\R$ with a unique maximum $q_0$.
\item a $1$-form $\eta$ on $M$ which satisfies condition~\eqref{eq:c-eta}, i.e., the zeroes of $\eta$ are nondegenerate and lie outside of $\textrm{Skel}_{n-1}(V)$ (this is equivalent to a vector field $v$ on $M$ whose zeroes have the same property). 
\end{itemize} 
 
We consider Hamiltonians $K:T^*M\to \R$ of the form $K(q,p)=k(|p|)+V(q)$, where $k(0)=0$ and $k=k(r)$ is a linear function of $r$ outside a compact set, of positive slope not belonging to the length spectrum. This data determines via equation~\eqref{eq:c-Floer-pert} the Floer continuation map 
$$
c^F:FC_*(-K)\to FC_*(K),
$$ 
which has the property that the only generator on which it may be nonzero is $q_0$. Moreover, we have computed in ~\S\ref{sec:cont-map} that 
$$
c^F(q_0)=\chi q_0.
$$
The continuation map can be equivalently interpreted as a quadratic vector 
$$
c_0^F(1)=\chi q_0\otimes q_0\in FC_*(K)\otimes FC_*(K).
$$
We emphasize that the chain level expression of the continuation map is the same for \emph{any} choice of continuation data $(V,\eta)$. 

\begin{definition}
The \emph{reduced Floer complex of $K$} is 
$$
\ol{FC}_*(K)=FC_*(K)/\im c^F.
$$ 
Its homology is the \emph{reduced Floer homology of $K$}, denoted $\ol{FH}_*(K)$.
\end{definition}

The \emph{reduced symplectic homology}  $\ol{SH}_*(D^*M)$ is the direct limit of reduced Floer homologies $\ol{FH}_*(K)$ over Hamiltonians $K$ which vanish on $D^*M$ and are linear outside a compact set, perturbed to have the form $k(|p|)+V(q)$ near the zero section as above. \footnote{This group is strictly speaking the analogue of $H_*(\Lambda;\chi\cdot \textrm{point})$. In~\cite{CHO-reduced} we use the more precise notation $FC_*(K;\im c^F)$ and $SH_*(D^*M; \im c^F)$.}

The relation 
$$
[\p,\lambda^F] = (\mu^F\otimes 1)(1\otimes c_0^F) - (1\otimes \mu^F)(\boldtau c_0^F\otimes 1)
$$
proved in~\S\ref{sec:cont-coproduct}, together with $\boldtau c_0^F=c_0^F$, shows that $\lambda^F$ descends to a chain map $\ol{FC}_*(K)\to \ol{FC}_*(2K)\otimes \ol{FC}_*(2K)$. These maps are compatible with the continuation maps obtained by increasing the slope of $K$, giving rise in the limit (with field coefficients) to a well defined coproduct of degree $-n+1$, denoted  
$$
\lambda^F:\ol{SH}_*(D^*M)\to \ol{SH}_*(D^*M)\otimes \ol{SH}_*(D^*M). 
$$

A straightforward enhancement of the Viterbo-Abbondandolo-Schwarz isomorphism shows that the map $\Psi$ induces an isomorphism between reduced homologies 
$$
\Psi_*:\ol{SH}_*(D^*M)\stackrel{\simeq}{\longrightarrow}\ol{H}_*(\Lambda). 
$$

In particular, associated to a choice of continuation data $(V,\eta)$ is a coproduct on $\ol{H}_*(\Lambda)$.  
The key to understanding the dependence of the coproduct on the choice of continuation data $(V,\eta)$ is the secondary continuation map, which we describe next.

\subsection{The secondary continuation map} \label{sec:secondary-cont-map}

Homotopies between different choices of pairs $(V,\eta)$ give rise to
secondary operations which we describe in this subsection. 

Consider two pairs $(V_i,\eta_i)$, $i=0,1$, satisfying the conditions
of the previous subsection, i.e., $V_i:M\to\R$ is a Morse function
with a unique maximum $q_i$ and $\eta_i$ a $1$-form on $M$ such that
condition~\eqref{eq:c-eta} holds. For $i=0,1$ let $K_i:T^*M\to\R$ be
associated Hamiltonians as in the previous subsection. After shifting
$V_0,V_1$ by constants we may assume without loss of generality that
$-K_0\leq -K_1\leq K_1\leq K_0$.

As in the previous subsection we pick a function $\phi:\R\to[-1,1]$
which equals $1$ for $s\leq-1$ and $-1$ for $s\geq 1$. 
Let $H_\sigma:\R\times T^*M\to\R$, $\sigma\in[0,\infty)$, be a smooth
family of $s$-dependent Hamiltonians with the following properties:
\begin{itemize}
\item $\p_sH_\sigma(s,x)\leq 0$ for all $\sigma,s,x$; 
\item $H_0(s,x)=\phi(s)K_0(x)$;
\item $H_\sigma(s,x)$ equals $K_0(x)$ for $s\leq-\sigma-1$ and
  $-K_0(x)$ for $s\geq\sigma+1$;
\item $H_\sigma(s,x)=\phi(s)K_1(x)$ for $|s|\leq\sigma$ and
  $\sigma\geq 1$. 
\end{itemize}
Let $\eta_\sigma$, $\sigma\in[0,\infty)$, be a smooth family of
$1$-forms with $\eta_\sigma=\eta_1$ for all $\sigma\geq 1$. 
We consider pairs $(\sigma,u)$ with $\sigma\in[0,\infty)$
and $u:\R\times S^1\to T^*M$ solving the Floer equation
\begin{equation*}
   \p_su+J(u)\bigl(\p_tu-X_{H_\sigma}(s,u)\bigr)=0
\end{equation*}
and converging to $1$-periodic orbits of $\mp K_0$ as $s\to\pm\infty$.
Their algebraic count gives rise to a degree $1$ map
$$
  \vec Q:FC_*(-K_0) \to FC_{*+1}(K_0)
$$
satisfying
\begin{equation}\label{eq:vecQ}
  \p_{K_0}\vec Q + \vec Q\p_{-K_0} = c_{10}c_1c_{01} - c_0,
\end{equation}
with the Floer continuation maps $c_i:FC_*(-K_i)\to FC_*(K_i)$ for $i=0,1$, 
$c_{01}:FC_*(-K_0)\to FC_*(-K_1)$, and $c_{10}:FC_*(K_1)\to FC_*(K_0)$.  
The map $\vec Q$ factors through the action zero part which we will 
denote by $\vec Q^{=0}$. Since the $V_i$ have unique maxima $q_i$, 
it follows from the previous subsection that the only nontrivial
contribution to $c_0$ is $c_0(q_0)=\chi q_0$. Similarly, the only
nontrivial contribution to the composition $c_{10}c_1c_{01}$ sends
$q_0\mapsto q_1\mapsto \chi q_1\mapsto \chi q_0$. 
This shows that the right hand side of equation~\eqref{eq:vecQ}
vanishes, and therefore $\vec Q$ descends to a map on homology
$$
  \vec Q:FH_*(-K_0) \to FH_{*+1}(K_0)
$$
which factors through the action zero part
$$
  \vec Q^{=0}:FH_*^{=0}(-K_0)\cong H_{n+*}(M) \to FH_{*+1}^{=0}(K_0)\cong H^{n-*-1}(M).
$$
For degree reasons nontrivial contributions can only occur for $*=0$
and $*=-1$ and give maps 
$$
   H_n(M)\to H^{n-1}(M)\quad\text{resp.}\quad H_{n-1}(M)\to H^n(M).
$$
In particular, we have shown

\begin{proposition}\label{prop:cont-map-choices}
If $H_1(M)=0$, then the secondary continuation map
$\vec Q:FH_*(-K_0) \to FH_{*+1}(K_0)$ associated to any interpolation
between pairs $(V_i,\eta_i)$, $i=0,1$, of continuation data vanishes. 
\qed
\end{proposition}

\subsection{Dependence of the continuation coproduct on choices}\label{sec:cont-coproduct-choices}

In this subsection we discuss the dependence of the continuation
coproduct $\lambda^F$ on the data $(V,\eta)$ of a Morse function and a
$1$-form on $M$. 

We consider the setup of the previous subsection and retain the
terminology from there. Thus we are given two pairs $(V_i,\eta_i)$,
$i=0,1$, with associated Hamiltonians $K_i:T^*M\to\R$ satisfying
$-K_0\leq -K_1\leq K_1\leq K_0$. They give rise to Floer continuation
maps $c_i:FC_*(-K_i)\to FC_*(K_i)$, $c_{01}:FC_*(-K_0)\to FC_*(-K_1)$,
and $c_{10}:FC_*(K_1)\to FC_*(K_0)$, and to a degree $1$ map
$\vec Q:FC_*(-K_0) \to FC_{*+1}(K_0)$ satisfying equation~\eqref{eq:vecQ}.

We denote by $\lambda^F_i:FC_*(K_i)\to FC_*(2K_i)\otimes FC_*(2K_i)$ the
continuation coproducts (of degree $1-n$) defined with the data $(V_i,\eta_i)$ 
and families $\eta_{i,\tau}$, $\tau\in[0,1]$ such that $\eta_{i,0}=\eta_{i,1}=\eta_i$ as in~\S\ref{sec:cont-coproduct}. 
Let
$$
  P:FC_*(K_1)\to FC_*(2K_0)\otimes FC_*(2K_0)
$$ 
be the degree $2-n$ map defined by the $2$-parametric family of Floer
problems depicted in Figure~\ref{fig:coproduct-choices}. 
\begin{figure}
\begin{center}
\includegraphics[width=.7\textwidth]{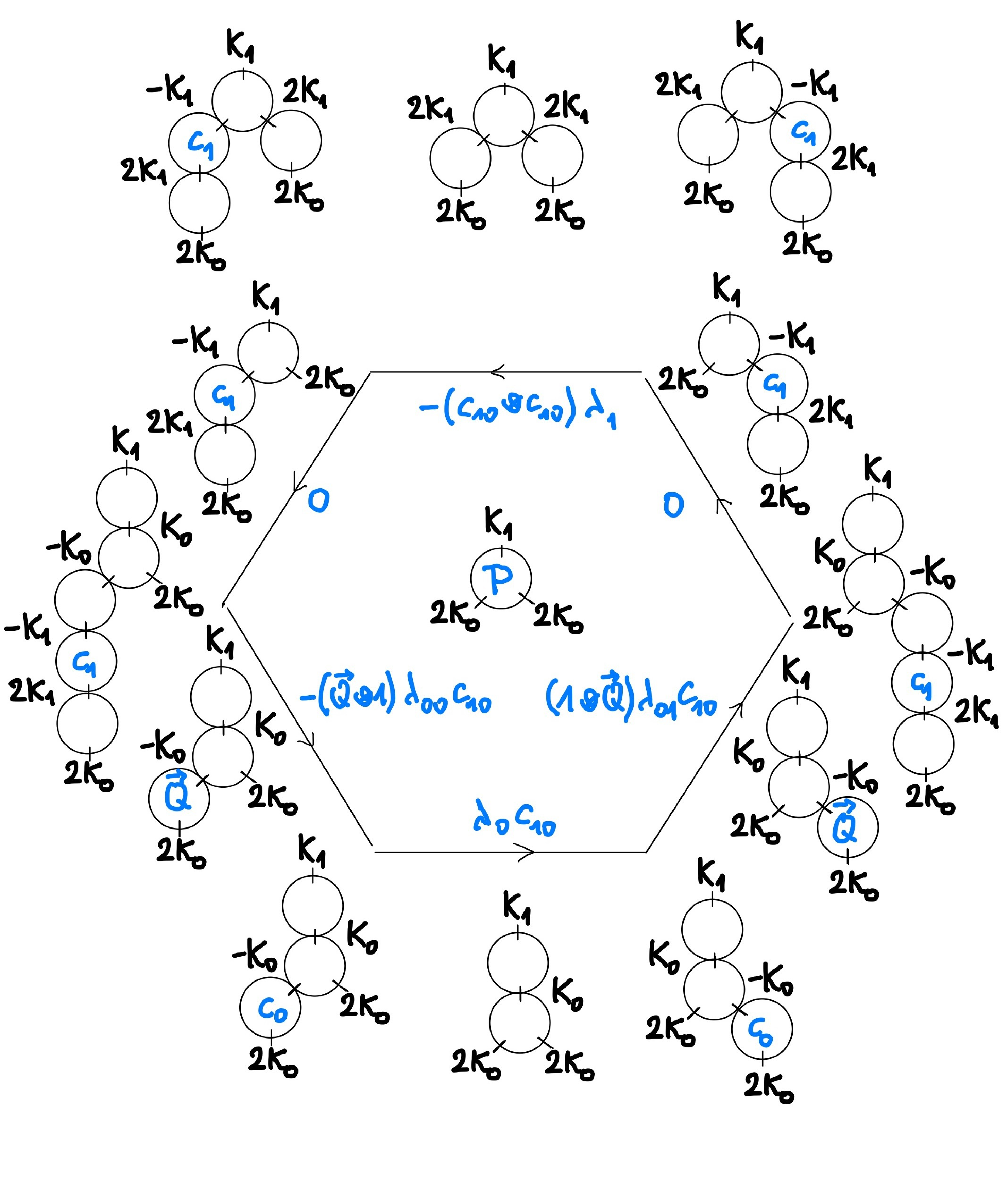}
\caption{The operation $P$}
\label{fig:coproduct-choices}
\end{center}
\end{figure}
These Floer problems are defined in terms of a $2$-parametric family
of Hamiltonian valued $1$-forms on the $3$-punctured sphere with
asymptotics and degenerations as in the figure, and a $2$-parametric
family of $1$-forms on $M$ which 
agree with $\eta_{1,\tau}$, $\tau\in[0,1]$ on the top side, with $\eta_1$ on the two top slanted sides, and with $\eta_{0,\tau}$, $\tau\in[0,1]$ on the bottom side of the hexagon. 

On the reduced Floer chain complex the compositions along the top
vertical sides vanish because they factor through the continuation map
$c_1:FC_*(-K_1)\to FC_*(K_1)$, so we obtain the relation  
\begin{equation*}
   [\p,P] = \lambda^F_0 c_{10} - (c_{10}\otimes c_{10})\lambda^F_1 +
   (1\otimes \vec Q)\lambda^F_{01}c_{10} - (\vec Q\otimes 1)\lambda^F_{00}c_{10}. 
\end{equation*}
Here $\lambda^F_{00}:FC_*(K_0)\to FC_*(-K_0)\otimes FC_*(2K_0)$ and 
$\lambda^F_{01}:FC_*(K_0)\to FC_*(2K_0)\otimes FC_*(-K_0)$ are the
degree $-n$ operations appearing at the ends of the continuation
coproduct $\lambda^F_0$ as in Figure~\ref{fig:cont-coproduct}. 

All the maps appearing on the right hand side of the last displayed
equation are chain maps, so they descend to maps on reduced Floer
homology (denoted by the same letters) satisfying
\begin{equation*}
   (c_{10}\otimes c_{10})\lambda^F_1 - \lambda^F_0 c_{10} =
   (1\otimes\vec Q)\lambda^F_{01}c_{10} - (\vec Q\otimes 1)\lambda^F_{00}c_{10}. 
\end{equation*}
Passing to the direct limit over Hamiltonians $K_0,K_1$ as above, we
have therefore shown

\begin{proposition}\label{prop:dependence-coproduct}
The continuation coproducts $\lambda_i^F$ on $\ol{SH}_*(D^*M)$
defined with continuation data $(V_i,\eta_i)$, $i=0,1$,
satisfy the relation 
\begin{equation*}
   \lambda^F_1 - \lambda^F_0 =
   (1\otimes\vec Q)\lambda^F_{01} - (\vec Q\otimes 1)\lambda^F_{00},
\end{equation*}
where $\vec Q$ is the secondary continuation map of the previous
subsection and $\lambda^F_{00}$, $\lambda^F_{01}$ are induced by the
maps defined above.
\qed
\end{proposition} 

\begin{remark}
The right hand side of the previous equation can be rephrased in
terms of the secondary continuation map and the product $\mu^F$.
We refer to~\cite[\S4.3]{CHO-reduced} for further details. 
\end{remark}

Proposition~\ref{prop:dependence-coproduct} shows that in general the continuation coproduct
may depend on the data $(V,\eta)$. 
If $H_1(M)=0$, however, the secondary coproduct $\vec Q$ vanishes by
Proposition~\ref{prop:cont-map-choices} and we obtain

\begin{proposition}\label{prop:coproduct-choices}
If $H_1(M)=0$, then the
continuation coproduct on reduced symplectic homology $\ol{SH}(D^*M)$ 
is independent of the choice of continuation data $(V,\eta)$. 
\qed
\end{proposition}

\begin{remark} \label{rmk:arbitrary-endpoints}
In~\S\ref{sec:cont-coproduct} we defined the coproduct $\lambda_F$ using a family $\eta_\tau$, $\tau\in[0,1]$ with equal endpoints $\eta_0=\eta_1=\eta$. The proof of Proposition~\ref{prop:dependence-coproduct} shows that, under the assumption $H_1(M)=0$, the coproduct can be defined using families $\eta_\tau$ with arbitrary endpoints satisfying condition~\eqref{eq:c-eta} (in particular we can take $\eta_1=-\eta_0$). This observation simplifies the computations in~\S\ref{sec:spheres} for spheres of dimension $>1$ by allowing the use of constant families of vector fields $v^\tau\equiv v$ for the topological definition of the coproduct.
\end{remark}

\begin{corollary} \label{cor:lambda1=0}
If $H_1(M)=0$ then, denoting $1\in \ol{SH}_*(D^*M)$ the unit and $\lambda_F$ the canonical coproduct, we have $\lambda_F(1)=0$. 
\end{corollary}

A proof of this result in a more general setting is given in~\cite[\S4]{CHO-reduced}, based on the vanishing of the secondary continuation map. We give here a topological proof, see also~\S\ref{sec:coproduct-S3} for the case of spheres of odd dimension $\ge 3$.  

\begin{proof}
We work on the topological side $\ol{H}_*\Lambda$ and compute, as in~\S\ref{sec:coproduct-S3}, the image of the fundamental class $1$ by representing it by constant loops and using a constant vector field $v^\tau\equiv v$ with isolated nondegenerate zeroes. If $v$ has no zeroes then its image under the coproduct is zero because it is represented by the empty chain. In the general case the image is a degenerate $1$-chain, hence vanishes in homology. 
\end{proof}

\section{Viterbo's isomorphism revisited}\label{sec:symp-loop-iso} 

As before, in this section $M$ is a closed oriented manifold, $T^*M$ its cotangent bundle with
the Liouville form $\lambda=p\,dq$, and $D^*M\subset T^*M$ its unit disc
cotangent bundle viewed as a Liouville domain. The symplectic homology
$SH_*(D^*M;\sigma)$ is isomorphic to the Floer homology $FH_*(H;\sigma)$ of a
fibrewise quadratic Hamiltonian $H:S^1\times T^*M\to\R$. On the other
hand, $FH_*(H;\sigma)$ is isomorphic to the loop homology $H_*(\Lambda)$
(Viterbo~\cite{Viterbo-cotangent}, Abbondandolo-Schwarz~\cite{AS,AS-corrigendum},
Salamon-Weber~\cite{SW}, Abouzaid~\cite{Abouzaid-cotangent}). Here we use coefficients twisted by the local system $\sigma$ defined by transgressing the second Stiefel-Whitney class, cf. Appendix~\ref{sec:local-systems}. We drop the local system $\sigma$ from the notation in the rest of this section.

The construction most relevant for our purposes is the chain map
$$
   \Psi:FC_*(H)\to MC_*(S)
$$
from the Floer complex of a Hamiltonian $H:S^1\times T^*M\to\R$ to the
Morse complex of an action functional $S:\Lambda\to\R$ on the loop space
defined in~\cite{AS-product-structures}. When applied to a fibrewise
quadratic Hamiltonian $H$ and the action functional $S_L$ associated to
its Legendre transform $L$, it induces an isomorphism on homology
$$
      \Psi_*:SH_*(D^*M)\cong FH_*(H)\to MH_*(S_L)\cong H_*\Lambda
$$
intertwining the pair-of-pants product with the loop product~\cite{AS-product-structures}. 

One annoying feature of the map $\Psi$ has been that, in contrast to its
chain homotopy inverse $\Phi:MC_*(S_L)\to FC_*(H)$, it does {\em not} preserve
the action filtrations. This would make it unsuitable for some of our
applications in~\cite{CHO-PD} such as those concerned with critical values. Using an
estimate inspired by~\cite{Cieliebak-Latschev}, we show in this section that
$\Psi$ {\em does} preserve suitable action filtrations when applied to
fibrewise linear Hamiltonians rather than fibrewise quadratic ones.

\subsection{Floer homology}\label{sec:Floer}

Consider a smooth time-periodic Hamiltonian $H:S^1\times T^*M\to\R$
which outside a compact set is either fibrewise quadratic, or linear
with slope not in the action spectrum. 
It induces a smooth Hamiltonian action functional  
$$
   A_H:C^\infty(S^1,T^*M)\to\R,\qquad x\mapsto\int_0^1\bigl(x^*\lambda-H(t,x)dt\bigr).
$$
Its critical points are $1$-periodic orbits $x$, which we can assume to
be nondegenerate with Conley--Zehnder index $\CZ(x)$. Let $J$ be a compatible
almost complex structure on $T^*M$
and denote the Cauchy--Riemann operator with Hamiltonian perturbation on $u:\R\times S^1\to
T^*M$ by
$$
   \pb_Hu := \p_su+J(u)\bigl(\p_tu-X_H(t,u)\bigr). 
$$
Let $FC_*(H)$ be the free $R$-module generated by $\Crit(A_H)$ and
graded by the Conley--Zehnder index. The Floer differential is given by
$$
   \p^F:FC_*(H)\to FC_{*-1}(H),\qquad x\mapsto\sum_{\CZ(y)=\CZ(x)-1}\#\MM(x;y)\,y,
$$
where $\#\MM(x;y)$ denotes the signed count of points in the oriented
$0$-dimensional manifold
\begin{align*}
   \MM(x;y) := \{ u:\R\times S^1\to T^*M\mid \, \, & \pb_Hu=0, \\
   & u(+\infty)=x,\ u(-\infty)=y\}/\R.
\end{align*}
Then $\p^F\circ\p^F=0$ and its homology $FH_*(H)$ is isomorphic to the
symplectic homology $SH_*(T^*M)$ if $H$ is quadratic. If $H$ is linear, we obtain an isomorphism to $SH_*(T^*M)$ in the direct limit as the slope goes to infinity.

\subsection{The isomorphism $\Phi$}\label{sec:Phi}

Suppose now that $H$ is fibrewise convex with fibrewise Legendre
transform $L:S^1\times TM\to\R$. As in~\S\ref{ss:A2+loop} we
consider the Morse complex $(MC_*,\p)$ of the action functional
$$
   S_L:\Lambda\to\R,\qquad S_L(q)=\int_0^1L(t,q,\dot q)dt. 
$$
Following~\cite{AS}, for $a\in\Crit(S_L)$ and
$x\in\Crit(A_H)$ we consider the space
\begin{align*}
   \MM(a;x) := \{u:(-\infty,0]\times S^1\to T^*M\mid
     &\pb_H(u)=0,\ u(-\infty)=x, \cr
     &\pi\circ u(0,\cdot)\in W^-(a)\},
\end{align*}
where $W^-(a)$ denotes the stable manifold for the pseudo-gradient flow of
$S_L$ and $\pi:T^*M\to M$ is the projection. (It is sometimes useful to view $W^-(a)$ as the unstable manifold for the negative pseudo-gradient flow of $S_L$.) For generic $H$ this is a manifold of dimension
$$
   \dim\MM(a;x) = \ind(a)-\CZ(x). 
$$
The signed count of $0$-dimensional spaces $\MM(a;x)$ defines a chain map
\begin{equation}
   \Phi:MC_*(S_L)\to FC_*(H),\qquad a\mapsto\sum_{\CZ(x)=\ind(a)}\#\MM(a;x)\,x.
\end{equation}
It was shown in~\cite{AS} that the induced map on homology is an isomorphism 
$$
   \Phi_*:MH_*(S_L)\stackrel{\cong}\longrightarrow FH_*(H). 
$$
For $u\in\MM(a;x)$ consider the loop $(q,p)=u(0,\cdot):S^1\to
T^*M$ at $s=0$. The definition of the Legendre transform yields the estimate
$$
   A_H(q,p) = \int_0^1\Bigl(\la p,\dot q\ra-H(t,q,p)\Bigr)dt \  
   \leq \int_0^1L(t,q,\dot q)dt = S_L(q).
$$
It follows that
$$
   A_H(x)\leq A_H\bigl(u(0,\cdot)\bigr)\leq S_L(q)\leq S_L(a)
$$
whenever $\MM(a;x)$ is nonempty, so $\Phi$ decreases action. 
  
\subsection{The isomorphism $\Psi$}\label{sec:Psi}

Consider once again a fibrewise quadratic Hamiltonian $H:S^1\times
T^*M\to\R$ as in~\S\ref{sec:Phi} with Legendre transform $L$. 
Following~\cite{AS-product-structures,Cieliebak-Latschev},
for $x\in\Crit(A_H)$ and $a\in\Crit(S_L)$ we define  
\begin{align*}
   \MM(x) := \{u:[0,\infty)\times S^1\to T^*M \mid \,\,  & 
     \pb_Hu=0,\\ 
     & u(+\infty,\cdot)=x,\ u(0,\cdot)\subset M\}
\end{align*}
and
\begin{equation*} 
   \MM(x;a) := \{u\in\MM(x)\mid u(0,\cdot)\in W^+(a)\},
\end{equation*}
where $W^+(a)$ is the stable manifold of $a$ for the negative pseudo-gradient flow of $S_L$,
see Figure~\ref{fig:Psi}. 

\begin{figure} [ht]
\centering
\input{Psi.pstex_t}
\caption{Moduli spaces for the map $\Psi$.}
\label{fig:Psi}
\end{figure} 

For generic $H$ these are manifolds of dimensions
$$
   \dim\MM(x) = \CZ(x),\qquad \dim\MM(x;a) = \CZ(x)-\ind(a). 
$$
The signed count of $0$-dimensional spaces $\MM(x;a)$ defines a chain map
\begin{equation*} 
   \Psi:FC_*(H)\to MC_*(S_L),\qquad x\mapsto\sum_{\ind(a)=\CZ(x)}\#\MM(x;a)\,a.
\end{equation*}
The induced map on homology is an isomorphism 
$$
   \Psi_*=\Phi_*^{-1}:FH_*(H)\stackrel{\cong}\longrightarrow MH_*(S_L)\cong H_*\Lambda,
$$
which is the inverse of $\Phi_*$ and intertwines the pair-of-pants product with the loop product.  
This was shown by Abbondandolo and Schwarz~\cite{AS-product-structures} 
with $\Z/2$-coefficients, and by Abouzaid~\cite{Abouzaid-cotangent} (following work of Kragh~\cite{Kragh}, see also Abbondandolo-Schwarz~\cite{AS-corrigendum}) with general coefficients, twisted on $H_*\Lambda$ by a suitable local system, see Appendix~\ref{sec:local-systems}. 
Moreover, Abouzaid proved that $\Psi_*$ is an isomorphism of twisted BV algebras.

Unfortunately, the map $\Psi$ does {\em not} preserve the action
filtrations. This already happens for a classical Hamiltonian
$H(q,p)=\frac12|p|^2+V(q)$: For $u\in\MM(x;a)$ the loop
$q=u(0,\cdot):S^1\to M$ satisfies
\begin{align*}
  A_H(x) &\geq A_H\bigl(u(0,\cdot)\bigr) = -\int_0^1V(q)dt \cr
  &\leq \int_0^1\Bigl(\frac12|\dot q|^2-V(q)\Bigr)dt = S_L(q)\geq S_L(a), 
\end{align*}
so the middle inequality goes in the wrong direction (even if $V=0$).

\subsection{An action estimate for Floer half-cylinders}\label{sec:action-estimate}

Now we will replace the quadratic Hamiltonians from the previous
subsections by Hamiltonians of the shape used in the definition of
symplectic homology. For Floer half-cylinders of such Hamiltonians, we
will estimate the length of their boundary loop on the zero section by 
the Hamiltonian action at $+\infty$. 

We equip $M$ with a Riemannian metric and choose the following data.

The Riemannian metric on $M$ induces a canonical almost complex 
structure $J_\st$ on $T^*M$ compatible with the symplectic form
$\om_\st=dp\wedge dq$ (Nagano~\cite{Nagano}, Tachibana-Okumura~\cite{Tachibana-Okumura}, see also~\cite[Ch.~9]{Blair}). In geodesic normal coordinates $q_i$ at a point $q$ and dual coordinates $p_i$ it is given by
$$
   J_\st:\frac{\p}{\p q_i}\mapsto -\frac{\p}{\p p_i},\qquad \frac{\p}{\p
   p_i}\mapsto\frac{\p}{\p q_i}.  
$$
We pick a nondecreasing smooth function $\rho:[0,\infty)\to(0,\infty)$
with $\rho(r)\equiv 1$ near $r=0$ and $\rho(r)=r$ for large $r$. Then$$
   J:\frac{\p}{\p q_i}\mapsto -\rho(|p|)\frac{\p}{\p p_i},\qquad
   \rho(|p|)\frac{\p}{\p p_i}\mapsto\frac{\p}{\p q_i}
$$
(in geodesic normal coordinates) defines a compatible almost complex
structure on $T^*M$ which agrees with $J_\st$ near the zero section and
is cylindrical outside the unit cotangent bundle. 

We view $r(q,p)=|p|$ as a function on $T^*M$. Then on $T^*M\setminus
M$ we have
$$
   \lambda=r\alpha,\qquad \alpha:=\frac{p\,dq}{|p|}.
$$
Consider a Hamiltonian of the form $H=h\circ r:T^*M\to\R$ for a
smooth function $h:[0,\infty)\to[0,\infty)$ vanishing near $r=0$. 
Then its Hamiltonian vector field equals $X_H=h'(r)R$, where $R$ is
the Reeb vector field of $(S^*M,\alpha)$. The symplectic and
Hamiltonian actions of a nonconstant 1-periodic Hamiltonian orbit
$x:S^1\to T^*M$ are given by 
$$
   \int_x\alpha = h'(r),\qquad A_H(x) = rh'(r)-h(r). 
$$
Given a slope $\mu>0$ which is not in the action spectrum of
$(S^*M,\alpha)$ and any $\eps>0$, 
we can pick $h$ with the following properties:
\begin{itemize}
\item $h(r)\equiv 0$ for $r\leq 1$ and $h'(r)\equiv \mu$ for $r\geq
  1+\delta$, with some $\delta>0$;
\item $h''(r)>0$ and $rh'(r)-h(r)-\eps\leq h'(r)\leq rh'(r)-h(r)$ for $r\in(1,1+\delta)$.
\end{itemize}
Specifically, we choose $0<\delta\le\eps/\mu$, we consider a smooth function $\beta:[0,\infty)\to [0,1]$ such that $\beta=0$ on $[0,1]$, $\beta=1$ on $[1+\delta,\infty)$ and $\beta$ is strictly increasing on $(1,1+\delta)$, and we define $h:[0,\infty)\to[0,\infty)$ by 
$$
h(r)=\mu \int_1^r \beta(\rho) \, d\rho.
$$
We have $rh'-h-h'=\mu\left( (r-1)\beta - \int_1^r\beta \right)$. This expression differentiates to $\mu(r-1)\beta'\ge 0$ and vanishes on $[0,1]$, hence it is nonnegative for $r\ge 0$. On the other hand, we have an upper bound $\mu\left( (r-1)\beta - \int_1^r\beta \right)\le \mu\delta$ for $r\in (1,1+\delta)$, and indeed for $r\ge 0$. Given our choice $\delta \le \eps/\mu$, this establishes the inequalities $rh'(r)-h(r)-\eps\leq h'(r)\leq rh'(r)-h(r)$ for all $r\ge 0$.

These inequalities imply that for each nonconstant $1$-periodic
Hamiltonian orbit $x$ we have
\begin{equation}\label{eq:ham-nonham-action}
   A_H(x)-\eps\leq \int_x\alpha \leq A_H(x). 
\end{equation}
With this choice of $J$ and $H$, consider now as in the previous
subsection a map $u:[0,\infty)\times S^1\to T^*M$ satisfying
$$
  \pb_Hu=0,\ u(+\infty,\cdot)=x,\ u(0,\cdot)\subset M. 
$$
Set $q(t):=u(0,t)$ and denote its length by
$$
   \ell(q) := \int_0^1|\dot q|dt. 
$$
The following proposition is a special case of~\cite[Lemma
  7.2]{Cieliebak-Latschev}. Since the proof was only sketched there,
we give a detailed proof below. 

\begin{proposition}\label{prop:action-estimate}
Let $H,J,u$ be as above with $q=u(0,\cdot)$ and a nonconstant
orbit $x=u(+\infty,\cdot)$. Then 
$$
   \ell(q)\leq \int_x\alpha \leq A_H(x). 
$$
The first inequality is an equality if and only if $u$ is contained in
the half-cylinder over a closed geodesic $q$, in particular $x$ is the
lift of the geodesic $q$. 
\end{proposition}   

The idea of the proof is to show that
\begin{equation*}
   0\leq \int_{(0,\infty)\times S^1}u^*d\alpha = \int_x\alpha-\ell(q). 
\end{equation*}
Since the image $u\bigl((0,\infty)\times S^1\bigr)$ can hit the zero section $M$
where $\alpha$ is undefined, the quantity $\int_{(0,\infty)\times S^1}u^*d\alpha$ has to be interpreted as an improper integral as follows. Given $\eps>0$, let $\tau=\tau_\eps:[0,\infty) \to
[0,\infty)$ be a smooth function with $\tau'(r)\geq 0$ for all $r$, $\tau(r)=0$
near $r=0$, and $\tau(r)=1$ for $r\geq\eps$, and consider the globally defined $1$-form on $T^*M$ given by  
$$
   \alpha_\eps := \frac{\tau(|p|) p\,dq}{|p|}. 
$$
We now \emph{define} 
\begin{equation} \label{eq:intdalpha}
\int_{(0,\infty)\times S^1}u^*d\alpha = \lim \limits_{\sigma\searrow 0}\lim \limits_{\eps\searrow 0}\int_{[\sigma,\infty)\times S^1}u^*d\alpha_\eps.
\end{equation}

The proof of Proposition~\ref{prop:action-estimate} is based on the following lemma. 

\begin{lemma}\label{lem:nonneg}
For any $v\in T_{(q,p)}T^*Q$, we have  
$$
   d\alpha_\eps(v,Jv)\geq 0.
$$
At points where $\tau'(|p|)>0$, equality only holds for $v=0$, whereas at
points where $\tau'=0$ and $\tau\neq 0$ equality holds if and only if 
$v$ is a linear combination of $p\p_p$ and $p\p_q$. 
\end{lemma}

\begin{proof}
In geodesic normal coordinates we compute
\begin{align*}
   d\alpha_\eps &= d\left(\sum_i\frac{\tau(|p|)p_idq_i}{|p|}\right) \\ 
   & =  \sum_i\frac{\tau(|p|)dp_i\wedge dq_i}{|p|} +
   \sum_{i,j}\frac{(\tau'(|p|)|p|-\tau(|p|))p_ip_jdp_i\wedge dq_j}{|p|^3}.
\end{align*}
For a vector of the form $v=\sum_ia_i\rho(|p|)\p_{p_i}$ we obtain
$J_\rho v=\sum_ia_i\p_{q_i}$ and hence by the Cauchy-Schwarz inequality
\begin{align*}
   d\alpha_\eps(v,Jv) &= \sum_i\frac{\tau(|p|)\rho(|p|)a_i^2}{|p|} +
   \sum_{i,j}\frac{(\tau'(|p|)|p|-\tau(|p|))\rho(|p|)p_ip_ja_ia_j}{|p|^3} \cr
   &= \frac{\tau(|p|)\rho(|p|)}{|p|^3}(|a|^2|p|^2-\la a,p\ra^2) +
   \frac{\tau'(|p|)\rho(|p|)}{|p|^2}\la a,p\ra^2 \cr 
   &\geq 0.
\end{align*}
At points where $\tau'>0$, equality only holds for $a=0$, and at points where
$\tau'=0$ and $\tau>0$ equality holds iff $a$ is a multiple of $p$. Similarly, for a general
vector $v=\sum_ia_i\rho(|p|)\p_{p_i} - \sum_ib_i\p_{q_i}$ we get
$d\alpha_\eps(v,Jv)\geq 0$, with equality iff either $a=b=0$ or $\tau'=0$ and both
$a$ and $b$ are multiples of $p$. 
\end{proof}

\begin{proof}[Proof of Proposition~\ref{prop:action-estimate}]
The proof consists in 3 steps.

\smallskip 

{\it Step 1. We prove that $\int_{(0,\infty)\times S^1} u^*d\alpha\ge 0$.}

\smallskip 

In view of Definition~\eqref{eq:intdalpha}, it is enough to show that $u^*d\alpha_\eps\geq 0$ on all of $(0,\infty)\times S^1$.  
To see this, recall that $u$ satisfies the equation
$\p_su+J(u)\bigl(\p_tu-X_H(u)\bigr)=0$, so that 
$$
   u^*d\alpha_\eps = d\alpha_\eps(\p_su,\p_tu)ds\wedge dt =
   d\alpha_\eps\bigl(\p_su,J(u)\p_su+X_H(u)\bigr)ds\wedge dt.
$$
Now at points in $D^*M$ the Hamiltonian vector field $X_H$ vanishes.
At points outside $D^*M$ we have $X_H=h'(r)R$ and
$\alpha_\eps=\alpha$ (we can assume w.l.o.g. $\eps\le 1$), so that $i_{X_H}d\alpha_\eps=h'(r)i_Rd\alpha=0$.
In either case we have 
$$
u^*d\alpha_\eps =
d\alpha_\eps(\p_su,J(u)\p_su)ds\wedge dt,
$$ 
which is nonnegative by Lemma~\ref{lem:nonneg}. 

\smallskip 

{\it Step 2. Denote $u_\sigma=u(\sigma,\cdot)$ for $\sigma>0$. We have 
$$
\lim \limits_{\sigma\searrow 0}\lim \limits_{\eps\searrow 0}\int_{S^1}u_\sigma^*\alpha_\eps = \ell(q).
$$
}

\smallskip 

To see this we consider the map 
$$
   \tilde q:[0,\infty)\times S^1\to T^*M,\qquad \tilde q(s,t):=\bigl(q(t),s\dot q(t)\bigr),
$$
and denote as above $\tilde q_\sigma=\tilde q(\sigma,\cdot)$ for $\sigma>0$. 
Since $J=J_\st$ near the zero section, the maps $u$ and $\tilde q$ 
agree with their first derivatives along the boundary loop $q$ at $s=0$, hence $u_\sigma$ and $\tilde q_\sigma$ are $C^1$-close for $\sigma$ close to $0$. On the other hand $\alpha_\eps$ is $C^0$-bounded near the zero section uniformly with respect to $\eps\to 0$. These two facts imply that the integrals $\int_{S^1}u_\sigma^*\alpha_\eps$ and $\int_{S^1} \tilde q_\sigma^*\alpha_\eps$ are $C^0$-close for $\sigma$ close to $0$, uniformly with respect to $\eps\to 0$, and therefore
$$
\lim \limits_{\sigma\searrow 0}\lim \limits_{\eps\searrow 0}\int_{S^1}u_\sigma^*\alpha_\eps = 
\lim \limits_{\sigma\searrow 0}\lim \limits_{\eps\searrow 0}\int_{S^1}\tilde q_\sigma^*\alpha_\eps.
$$

We now prove that  
\begin{equation}\label{eq:limepsqsigmaalphaeps}
\lim \limits_{\eps\searrow 0}\int_{S^1}\tilde q_\sigma^* \alpha_\eps = \ell(q)
\end{equation}
for all $\sigma>0$, which implies the desired conclusion. Fix therefore $\sigma>0$. Let $I_\eps=\{t\in S^1\, : \, |\sigma\dot q(t)|\le \eps\}$, so that $I_\eps\subset I_{\eps'}$ for $\eps\le \eps'$ and $\int_{\eps>0} I_\eps = I_0 = \{t\, : \, \dot q(t)=0\}$. On the one hand we have 
\begin{align*}
\int_{S^1\setminus I_\eps}\tilde q_\sigma^*\alpha_\eps & = \int_{S^1\setminus I_\eps}\tilde q_\sigma^*\alpha 
= \int_{S^1\setminus I_\eps}\alpha_{(q(t),\sigma\dot q(t))}\cdot \dot{\tilde q}(t) \\
& = \int_{S^1\setminus I_\eps} \frac{\sigma |\dot q(t)|^2}{|\sigma\dot q(t)|}dt 
 = \int_{S^1\setminus I_\eps} |\dot q(t)|dt 
 = \ell(q|_{S^1\setminus I_\eps}).
\end{align*}

We can therefore estimate 
\begin{align*}
\left| \int_{S^1} \tilde q_\sigma^*\alpha_\eps - \ell(q|_{S^1\setminus I_\eps})\right| 
& = \left|\int_{I_\eps} \tilde q_\sigma^*\alpha_\eps\right|
= \left|\int_{I_\eps}\alpha_\eps(\tilde q_\sigma(t))\cdot \dot{\tilde q}_\sigma(t) dt\right| \\
& \le C \cdot \frac{\eps}{\sigma}\cdot m(I_\eps) \to0 \quad \textrm{for } \eps\to 0.
\end{align*}
Here $m(I_\eps)$ is the measure of $I_\eps$, uniformly bounded by the length of the circle, $C>0$ is a $C^0$-bound on $\alpha_\eps$ near the $0$-section, uniform with respect to $\eps\to 0$, and $\eps/\sigma$ is by definition the bound on $|\dot q(t)|$ on $I_\eps$. The estimate follows from $\dot{\tilde q}_\sigma = (\dot q,\sigma \ddot q)$ and the fact that the $1$-form $\alpha_\eps$ only acts on the first component of the vector $\dot{\tilde q}_\sigma$. 

Since $\lim\limits_{\eps\searrow 0}\ell(q|_{S^1\setminus I_\eps}) = \ell(q|_{S^1\setminus I_0}) =\ell(q)$, equality~\eqref{eq:limepsqsigmaalphaeps} follows.

\smallskip 

{\it Step 3. We prove 
$$
\int_{(0,\infty)\times S^1}u^*d\alpha = \int_x \alpha -\ell(q). 
$$
}

\smallskip 

Indeed, for $\sigma,\eps>0$ Stokes' theorem gives
$$
\int_{[\sigma,\infty)\times S^1}u^*d\alpha_\eps =
     \int_x\alpha - \int_{u_\sigma}\alpha_\eps.  
$$
(The $1$-form $\alpha_\eps$ is equal to $\alpha$ near the orbit $x$.) The desired equality follows from the definition of $\int_{(0,\infty)\times S^1}u^*d\alpha$ and Step~2. 

{\it Conclusion.} Combining Step~3 with Step~1 we obtain the first inequality $\ell(q)\leq \int_x\alpha$ in
Proposition~\ref{prop:action-estimate}. Moreover, Lemma~\ref{lem:nonneg} (in the limit $\eps\to 0$) shows that
this inequality is an equality if and only if $u$ is contained in
the half-cylinder over a closed geodesic. 

The second inequality
$\int_x\alpha\leq A_H(x)$ follows from~\eqref{eq:ham-nonham-action}.
\end{proof}

\subsection{The isomorphism $\Psi$ from symplectic to loop homology}\label{sec:Psi-symp}

Now we adjust the definition of $\Psi$ to symplectic homology.
For $J,H$ as in the previous subsection and $x\in\Crit(A_H)$ we define
as before
\begin{align*}
   \MM(x) := \{u:[0,\infty)\times S^1\to T^*M \mid \, \, & 
     \pb_Hu=0,\\ 
     & u(+\infty,\cdot)=x,\ u(0,\cdot)\subset M\}.
\end{align*}
By Proposition~\ref{prop:action-estimate} the loop $q=u(0,\cdot)$
satisfies $\ell(q)\leq A_H(x)$. Moreover, the loop $q$ is smooth and in particular has Sobolev class $H^1$, hence following Anosov~\cite{Anosov} it has a unique $H^1$-reparametrization $\ol q:S^1\to M$, with $|\dot{\ol q}|\equiv{\rm const}$ and $\ol q(0)=q(0)$ (we say that $\ol q$ is \emph{parametrized proportionally to arclength, or PPAL}).
We have
$$
   \ell(q) = \ell(\ol q) = \int_0^1|\dot{\ol q}|dt =
   \Bigl(\int_0^1|\dot{\ol q}|^2dt\Bigr)^{1/2} = E(\ol q)^{1/2}
$$
with the {\em energy}
$$
   E:\Lambda\to\R,\qquad E(q) := \int_0^1|\dot q|^2dt.
$$
The energy defines a smooth Morse-Bott function on the loop space
whose critical points are constant loops and geodesics parametrized
proportionally to arclength. We denote by $W^\pm(a)$ the
unstable/stable manifolds of $a\in\Crit(E)$ with respect to $\nabla E$.
Now for $x\in\Crit(A_H)$ and $a\in\Crit(E)$ we define  
\begin{equation*} 
   \MM(x;a) := \{u\in\MM(x)\mid \ol{u(0,\cdot)}\in W^+(a)\}.
\end{equation*}
An element $u$ in this moduli space still looks as in Figure~\ref{fig:Psi},
where now the loop $q=u(0,\cdot)$ is reparametrized proportionally to
arclength and then flown into $a$ using the flow of $-\nabla E$. 
By Proposition~\ref{prop:action-estimate}, for $u\in \MM(x;a)$ we have
the estimate
\begin{equation}\label{eq:Psi-action-estimate}
   A_H(x) \geq \ell(q) = E(\ol q)^{1/2} \geq E(a)^{1/2}.
\end{equation}
To define the map $\Psi$, we now perturb $H$ and $E$ by small Morse
functions near the constant loops on $M$ and the closed geodesics, and
we generically perturb the almost complex structure $J$ from the
previous subsection. For generic such perturbations, each $\MM(x;a)$
is a manifold of dimension 
$$
   \dim\MM(x;a) = \CZ(x)-\ind(a). 
$$
The signed count of $0$-dimensional spaces $\MM(x;a)$ defines a chain map
\begin{equation*} 
   \Psi:FC_*(H)\to MC_*(E^{1/2}),\qquad x\mapsto\sum_{\ind(a)=\CZ(x)}\#\MM(x;a)\,a.
\end{equation*}
Here $MC_*(E^{1/2})$ denotes the Morse chain complex of $E:\Lambda\to\R$,
graded by the Morse indices of $E$, but filtered by the {\em square
  root} $E^{1/2}$ (which is decreasing under the negative gradient flow of $E$).  
The action estimate~\eqref{eq:Psi-action-estimate} continues to hold
for the perturbed data up to an arbitrarily small error, which we can
make smaller than the smallest difference between lengths of geodesics
below a given length $\mu$. Thus $\Psi$ preserves the filtrations
$$
   \Psi:FC_*^{<b}(H)\to MC^{<b}_*(E^{1/2}).
$$
The induced maps on filtered Floer homology 
$$
   \Psi_*:FH_*^{(a,b)}(H)\to MH_*^{(a,b)}(E^{1/2})\cong H_*^{(a,b)}\Lambda
$$
have upper triangular form with respect to the filtrations with $\pm1$
on the diagonal (given by the half-cylinders over closed geodesics in
Proposition~\ref{prop:action-estimate}), so they are isomorphisms.  
It follows from~\cite{AS-product-structures,Abouzaid-cotangent} that
$\Psi_*$ intertwines the pair-of-pants product with the loop product,
as well as the corresponding BV operators. Passing to the direct limit
over Hamiltonians $H$, we have thus proved

\begin{theorem}\label{thm:Psi-symp}
The map $\Psi$ induces isomorphisms on filtered symplectic homology 
$$
   \Psi_*:SH_*^{(a,b)}(D^*M)\stackrel{\cong}\longrightarrow
   MH_*^{(a,b)}(E^{1/2})\cong H_*^{(a,b)}\Lambda, 
$$
where the left hand side is filtered by non-Hamiltonian action and
the right hand side by the square root of the energy. 
These isomorphisms intertwine the pair-of-pants product with the
Chas-Sullivan loop product, as well as the corresponding BV operators.
\qed
\end{theorem}

\section{Viterbo's isomorphism intertwines $A_2^+$-structures}\label{sec:pop-GH-iso}

We keep the setup from the previous section, so $M$ is a closed
oriented Riemannian manifold and $D^*M\subset T^*M$ its unit disc cotangent bundle.
In this section we prove Theorem~\ref{thm:main2}, which will be an
immediate consequence of earlier results and the following theorem.  

\begin{theorem}\label{thm:cont-loop-iso} 
The chain maps underlying the isomorphism 
$$
\Psi_*:SH_*(D^*M)\stackrel{\cong}\longrightarrow H_*\Lambda
$$
from Theorem~\ref{thm:Psi-symp} are morphisms of special
$A_2^+$-algebras. For $n\neq 2$ these morphisms are special.  
\end{theorem}

\begin{proof}[Proof of Theorem~\ref{thm:main2}]
The first assertion follows from Proposition~\ref{prop:A2+loop} and
Proposition~\ref{prop:TQFT+}. The second assertion follows from 
Theorem~\ref{thm:cont-loop-iso} and Proposition~\ref{prop:A2+toA2-mor}.
\end{proof}

\begin{remark}[open questions about the various identifications between Morse and Floer complexes for cotangent bundles] 
\label{rmk:PhiPsiAS}
We have already discussed in~\S\ref{sec:Phi} the action preserving chain-level isomorphism 
$\Phi:MC_*(S_L)\to FC_*(H)$ of Abbondandolo-Schwarz~\cite{AS}, defined for an asymptotically quadratic Hamiltonian $H$. It would be interesting to clarify whether $\Phi$ also defines a morphism of special $A_2^+$-algebras. 

Abbondandolo-Schwarz constructed in~\cite{AS-Legendre} an action-preserving chain level isomorphism $\Psi_{AS}:FC_*(H)\to MC_*(S_L)$ which is a chain homotopy inverse of $\Phi$. They also argued that, from the perspective of the Legendre transform, the moduli spaces that define $\Psi_{AS}$ arise naturally from the moduli spaces for $\Phi$. We expect that $\Psi_{AS}$ and our morphism $\Psi$ can be connected by a suitable chain homotopy (we know that they induce the same map $\Phi^{-1}_*$ in homology). It would also be interesting to clarify whether $\Psi_{AS}$ is a morphism of special $A_2^+$-algebras. We expect this to hold or fail for both $\Psi_{AS}$ and $\Phi$ simultaneously. 

One can further ask whether $\Psi_{AS}$ and $\Psi$ are homotopic as morphisms of $A_2^+$-structures. This would require in particular to develop the discussion of $A_2^+$-structures from~\cite{CO-cones} by defining such a notion of homotopy.
\end{remark}

To prove Theorem~\ref{thm:cont-loop-iso}, we need to verify the conditions in Definition~\ref{def:A2+mor}
for each chain map $\Psi:FC_*(K)\to MC_*(E^{1/2})$ associated to a
Hamiltonian $K=K_\mu$ as in the previous subsection. 
The first part of condition (i) holds because $\Psi q_0=q_0$, which follows directly
from the definition of $\Psi$. Moreover, seen through the canonical identifications $FC_*^{=0}(K)\equiv MC_*^{=0}(E^{1/2})\equiv MC^{n-*}(V)$, the restriction of $\Psi$ to the energy zero Floer subcomplex acts as the identity. This shows that the second part of condition~(i) is also satisfied. 

The map $\Gamma:FC_*(K)\otimes FC_*(K)\to MC_*(E^{1/2})$ in condition (ii) is 
defined by the count of elements in $0$-dimensional moduli spaces of solutions to a $1$-parametric mixed Floer-Morse problem which we describe below. Inspection of the boundary of the $1$-dimensional moduli spaces of solutions shows that $\Gamma$ satisfies condition (ii). This fact was previously proved in~\cite{AS-product-structures}, which contains the description of an essentially equivalent map $\Gamma$. 

The $1$-parametric Floer-Morse problem is a count of Floer discs in $T^*M$ with two positive
punctures and boundary on the zero section, followed by a Morse pseudo-gradient line in the loop space of $M$. It is obtained as a concatenation of 3 distinct $1$-parametric Floer-Morse problems described by Figure~\ref{fig:Gamma}. {\it On the first interval of parametrization} the underlying moduli space of curves is that of discs with 2 interior punctures and one boundary marked point, where the punctures and the marked point are requested to be aligned. At the negative end of the interval the 2 interior punctures collide and form a sphere bubble (this gives rise to the term $-\Psi\mu^F$ in the expression of $[\p,\Gamma]$), whereas at the positive end of the interval the second puncture collides with the marked point and gives rise to a disc bubble containing the marked point. In this configuration the interior punctures, the node and the marked point are all aligned. {\it On the second interval of parametrization} we allow the marked point to move clockwise towards the node. At the positive end of this interval the marked point collides with the node and forms a disc bubble. However, this disc bubble is constant because the 0-section is an exact Lagrangian, so that we directly replace the configuration by one where the marked point lies at the node. {\it On the third and last interval of parametrization} we insert length $T>0$ pseudo-gradient lines before imposing the incidence condition at the marked point. The positive end of this interval of parametrization corresponds to $T=\infty$ and gives rise to the term $\mu (\Psi\otimes\Psi)$ in the expression of $[\p,\Gamma]$. 

In Figure~\ref{fig:Gamma} the dashed lines represent pseudo-gradient flow lines for the energy functional on loop space. We only represent them in the last two configurations depicted in  Figure~\ref{fig:Gamma} in order not to burden excessively the drawing. However, the reader should be aware that such pseudo-gradient lines are also present in the first five configurations from Figure~\ref{fig:Gamma}. 

For further reference it is convenient to write 
\begin{equation}\label{eq:Gamma123}
\Gamma=\Gamma_1+\Gamma_2+\Gamma_3,
\end{equation}
where $\Gamma_i$, $i=1,2,3$ corresponds to the count of elements in the $0$-dimensional moduli spaces of solutions to the Floer-Morse problem restricted to the $i$-th interval of parametrization for $\Gamma$.

\begin{figure}
\begin{center}
\includegraphics[width=.8\textwidth]{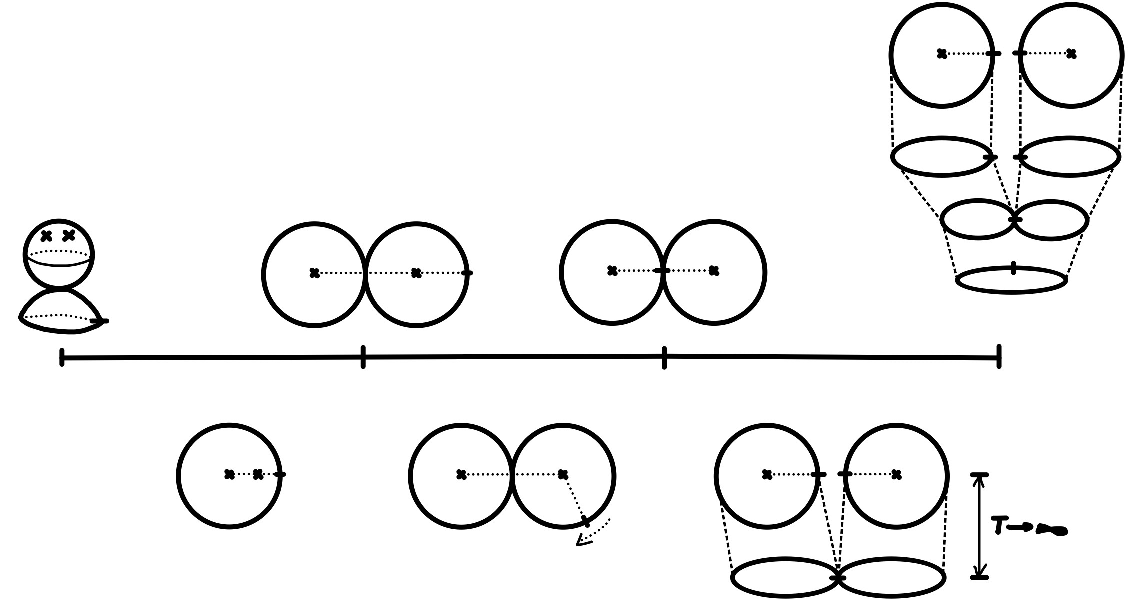}
\caption{The map $\Gamma$}
\label{fig:Gamma}
\end{center}
\end{figure}

The remainder of this section is devoted to the proof of condition (iii). 
For this we need to construct a chain homotopy 
$$
   \Theta: FC_*(K) \to MC_*(E^{1/2})\otimes MC_*(E^{1/2}) 
$$
satisfying $\Theta c^F=0$ and
\begin{equation}\label{eq:Theta-Floer}
   [\p,\Theta] = \lambda\Psi - (\Psi\otimes\Psi)\lambda^F
  - (\Gamma\otimes\Psi)(1\otimes c_0^F) + (\Psi\otimes\Gamma)(\boldtau c_0^F\otimes 1).
\end{equation}
The map $\Theta$ will be defined by a count of Floer maps to $T^*M$
defined over a $2$-parametric family of punctured annuli. In the first
subsection we describe the underlying moduli space of conformal annuli.

\subsection{Conformal annuli}\label{sec:conformal-annuli}

A {\em (conformal) annulus} is a compact genus zero Riemann surface
with two boundary components. By the uniformization theorem (see for example~\cite{Buser}), 
each annulus is biholomorphic to $[0,R]\times\R/\Z$ with its standard complex
structure for a unique $R>0$ called its {\em (conformal) modulus}.  
The exponential map $s+it\mapsto e^{2\pi(s+it)}$ sends the standard annulus
onto the annulus
$$
   A_R=\{z\in\C\mid 1\leq|z|\leq e^{2\pi R}\}\subset\C.  
$$
It will be useful to consider slightly more general annuli in the Riemann
sphere $S^2=\C\cup\{\infty\}$. A {\em circle in $S^2$} is the
transverse intersection of $S^2\subset\R^3$ with a plane. We will
call a {\em disc in $S^2$} an open domain $D\subset S^2$ bounded by a circle,
and an {\em annulus in $S^2$} a set $\ol{D}\setminus D'$ for two discs
$D,D'\subset S^2$ satisfying $\ol{D'}\subset D$ (with the induced complex structure). 

\begin{lemma}\label{lem:conf-annuli}
Every annulus $A$ in $S^2$ of conformal modulus $R$ can be mapped by a
M\"obius transformation onto the standard annulus $A_R\subset\C\subset
S^2$ above.  
\end{lemma}

\begin{proof}
Write $A=\ol{D}\setminus D'$ for discs $D,D'\subset S^2$. After
applying a M\"obius transformation, we may assume that $D$ is the 
disc $\{z\in\C\mid |z|<e^{2\pi R}\}$. Let $D_1\subset D$ be the unit disc.
There exists a M\"obius transformation $\phi$ of $D$ sending a point $z'\in\p
D'$ to a point $z_1\in\p D_1$ and the positive tangent direction to $\p
D'$ at $z'$ to the positive tangent direction to $\p D_1$ at
$z_1$. Thus $\phi$ sends $\p D'$ to a circle tangent to $\p D_1$ at
$z_1$, and since the annuli $\ol{D}\setminus D'$ and $\ol{D}\setminus D_1$ both
have modulus $R$ we must have $\phi(\p D')=\p D_1$, hence $\phi(D')=D_1$.
\end{proof}

For each $R$, the standard annulus $[0,R]\times\R/\Z$ carries two
canonical foliations: one by the line segments $[0,R]\times\pt$
and one by the circles $\pt\times\R/\Z$. Moreover, these two foliations are invariant under the automorphism group of the annulus. Hence by Lemma~\ref{lem:conf-annuli}
each annulus in $S^2$ also carries two canonical foliations, one by
circle segments connecting the two boundary components and one by circles,
such that the foliations are orthogonal and the second one contains the two
boundary loops. These two foliations can be intrinsically described as follows: the automorphism group of an annulus $A$ is $\Aut(A)\simeq S^1$. The second foliation consists of the orbits of the $S^1$-action. The first foliation is the unique foliation orthogonal to the first one. Its leaves connect the two boundary components because this is the case for a standard annulus.

Figure~\ref{fig:conf-annuli}
shows a $1$-parametric family of annuli in $\C$ whose conformal moduli
tend to $0$ together with their canonical foliations. The domain at
modulus $0$ is the difference of two discs touching at one point, the node.
Putting the node at the origin, the inversion $z\mapsto 1/z$ maps this
domain onto a horizontal strip in $\C$ (with the node at $\infty$)
with its standard foliations by straight line segments and lines.
Opening up the node, we can conformally map it onto the standard disc
with two boundary points corresponding to the node (since the map is
not a M\"obius transformation, the two foliations will not be by
circle segments).  

\begin{figure}
\begin{center}
\includegraphics[width=1\textwidth]{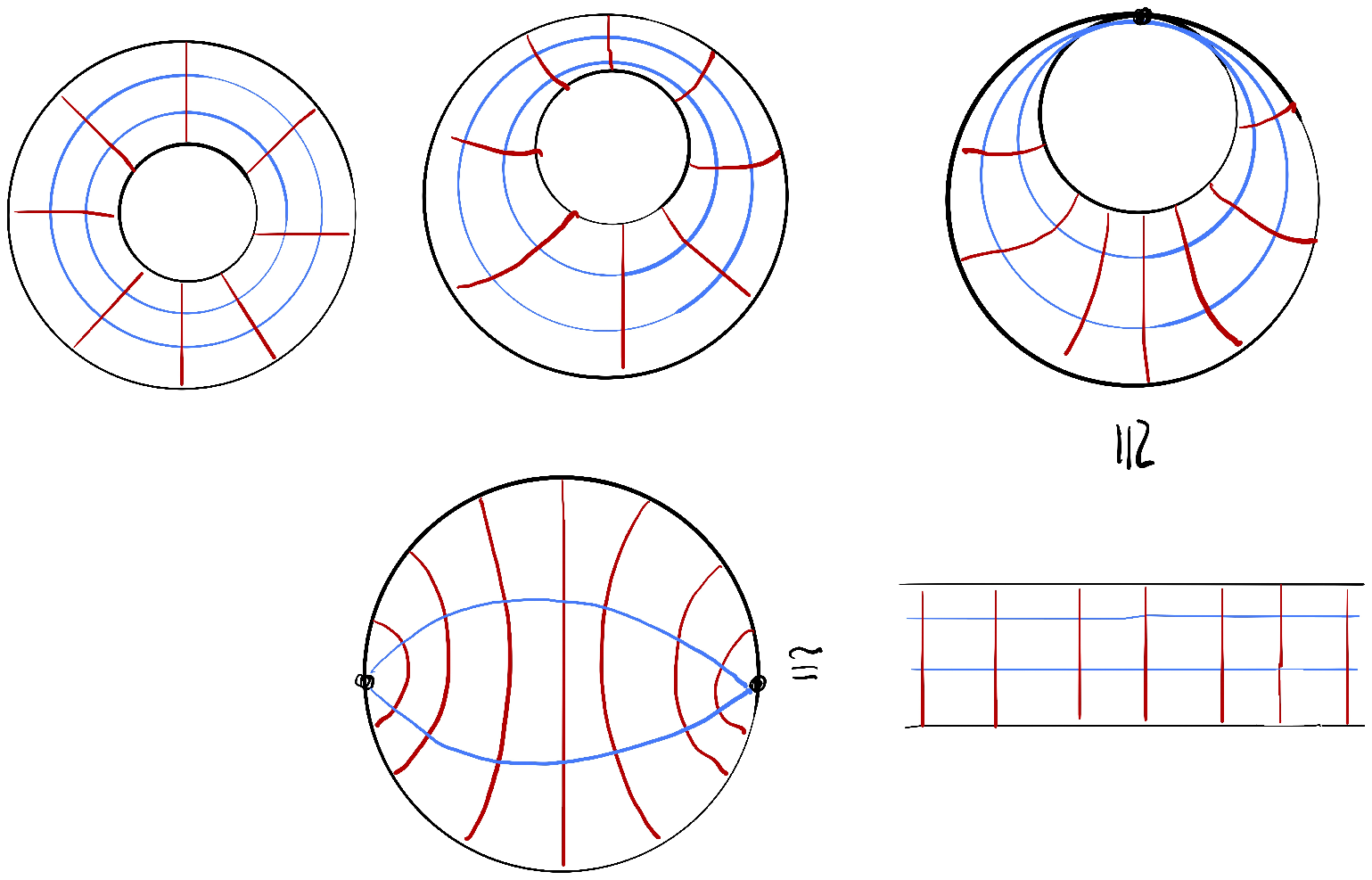}
\caption{Conformal annuli and their canonical foliations}
\label{fig:conf-annuli}
\end{center}
\end{figure}

{\bf Annuli with aligned marked points. }
The relevant domains for our purposes are annuli with $3$ marked
points, one interior and one on each boundary component. We require
that the $3$ points are {\em aligned}, by which we mean that they lie
on the same leaf of the canonical foliation connecting the two
boundary components. (In the next subsection the interior marked point
will correspond to the input from the Floer complex and the boundary
marked points will be the initial points of the boundary loops on the
zero section.) 

Figure~\ref{fig:annuli-fixed-modulus}
shows the moduli space of such annuli with {\em fixed} finite conformal modulus. 
It is an interval over which the interior marked point moves from one
boundary component to the other. Each end of the interval corresponds
to a rigid nodal curve consisting of an annulus with one boundary marked point
and a disc with an interior and a boundary marked point, where the
marked point and the node are aligned in the annulus, and the two
marked points and the node are aligned in the disc (i.e., they lie on
a circle segment perpendicular to the boundary). 

\begin{figure}
\begin{center}
\includegraphics[width=.7\textwidth]{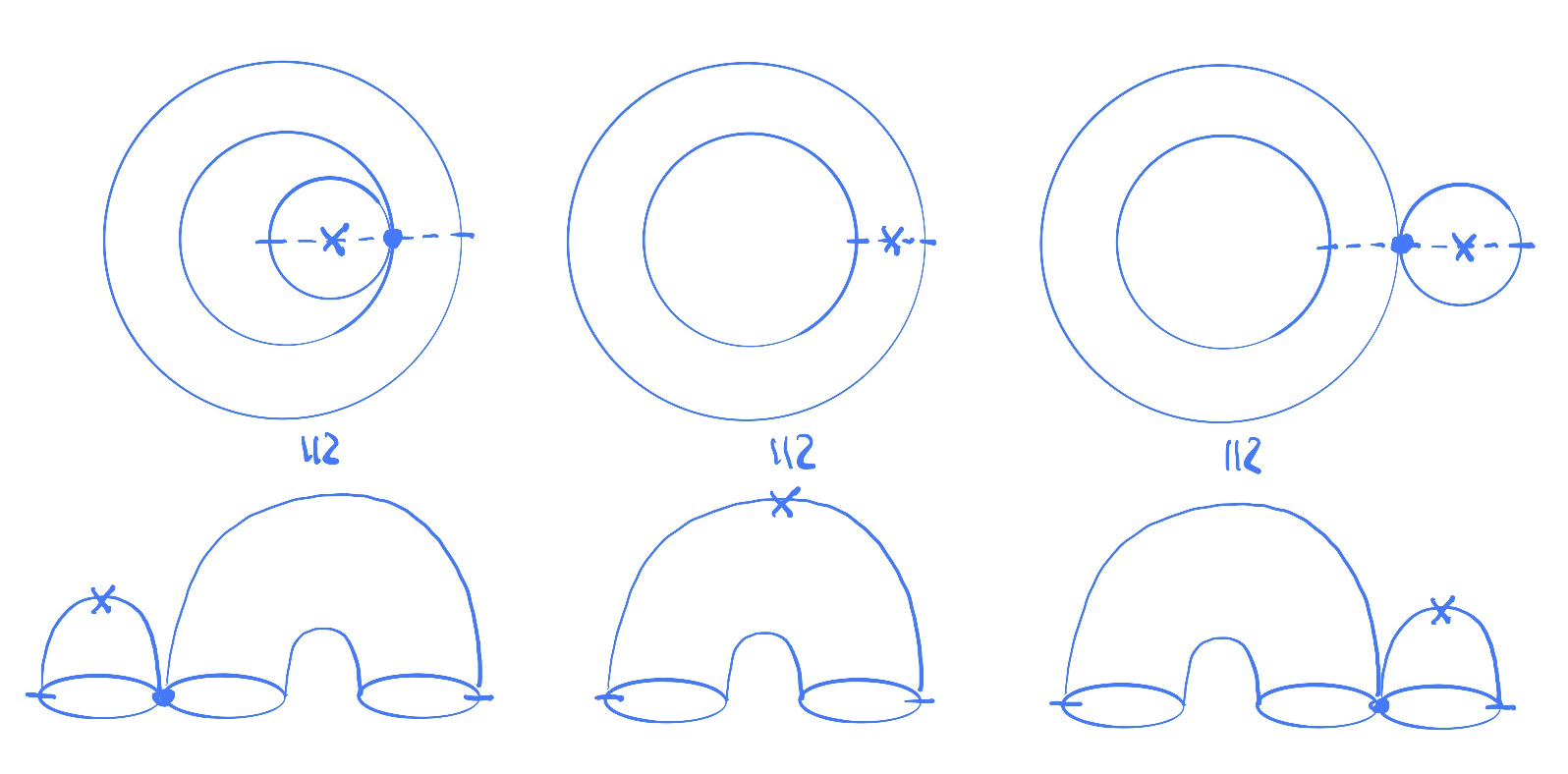}
\caption{Annuli with aligned marked points and fixed modulus}
\label{fig:annuli-fixed-modulus}
\end{center}
\end{figure}

Figure~\ref{fig:annuli-varying-modulus}
shows the moduli space of such annuli with {\em varying} conformal modulus. 
It is a pentagon in which we will view the two lower sides as being
``horizontal'' (although they meet at an actual corner).
Then in the vertical direction the conformal modulus increases from
$0$ (on the top side) to $\infty$ (on the two lower sides), while in
the horizontal direction the interior marked point moves from one   
boundary component to the other. In all configurations the marked
points and nodes are aligned.
The interior nodes occuring along the bottom sides carry
asymptotic markers (depicted as arrows) that are aligned 
with the boundary marked points. In particular, each interior node comes with
an orientation reversing isomorphism between the tangent circles
matching the asymptotic markers (this is the ``decorated compactification'').
\begin{figure}
\begin{center}
\includegraphics[width=1\textwidth]{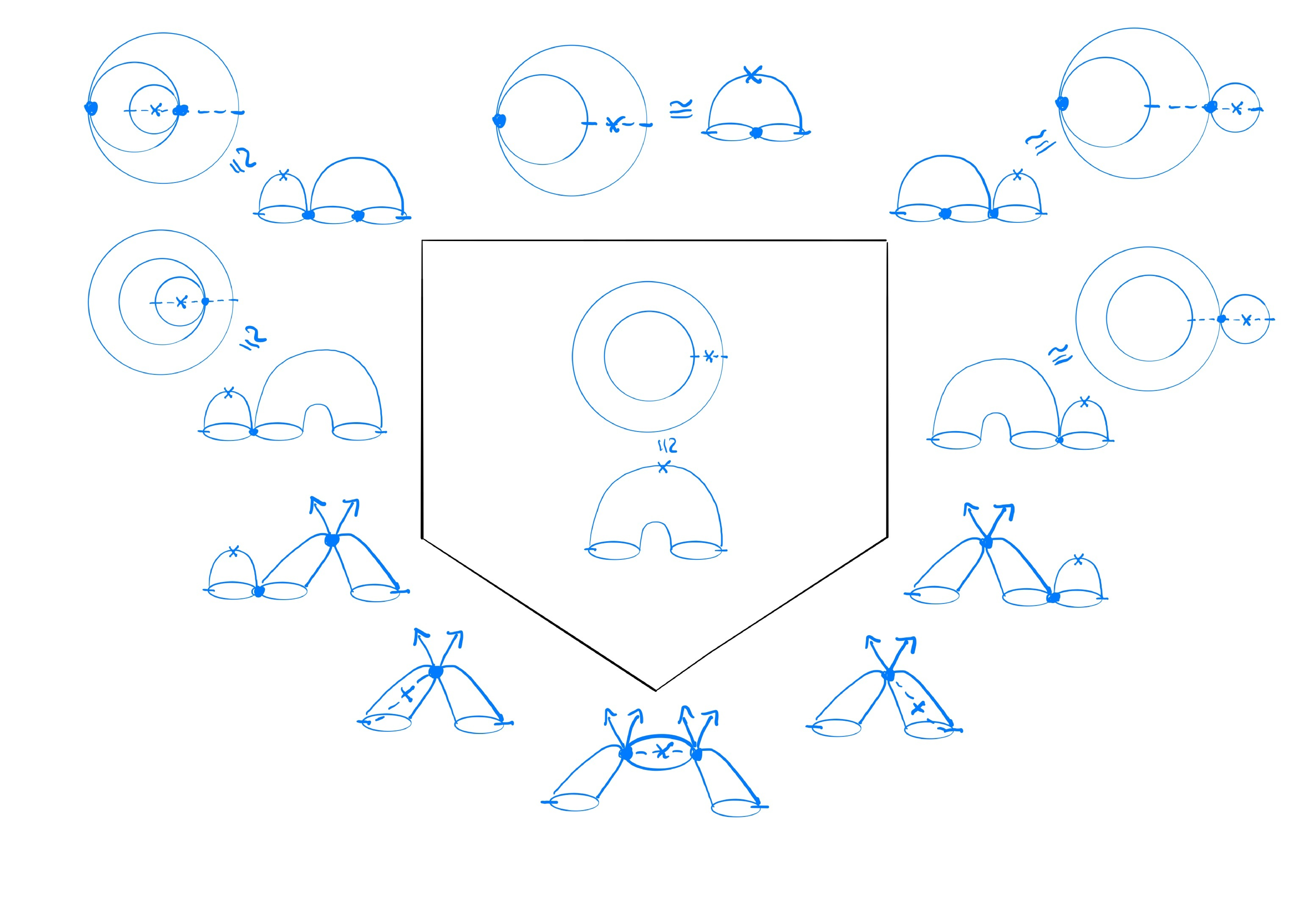}
\caption{Annuli with aligned marked points and varying modulus}
\label{fig:annuli-varying-modulus}
\end{center}
\end{figure}

\subsection{Floer annuli}\label{sec:Floer-annuli}

Now we define a moduli space of Floer maps into $T^*M$ over the moduli
space $\PP$ of annuli in Figure~\ref{fig:annuli-varying-modulus}.
For this, we choose a family of $1$-forms $\beta_\tau$, $\tau\in\PP$, with
the following properties (see Figure~\ref{fig:Floer-annuli}):
\begin{itemize}
\item $d\beta_\tau\leq 0$ for all $\tau$;
\item $\beta_\tau$ equals $dt$ near the (positive) interior puncture,
  and $2\,dt$ in coordinates $(s,t)\in[0,\eps)\times\R/\Z$ near each
    (negative) boundary component, i.e., it has weights $1,2,2$;
\item on annuli of infinite modulus, $\beta_\tau$ has weights at the
  punctures (positive or negative) as shown in the figure. 
\end{itemize}

\begin{figure}
\begin{center}
\includegraphics[width=1\textwidth]{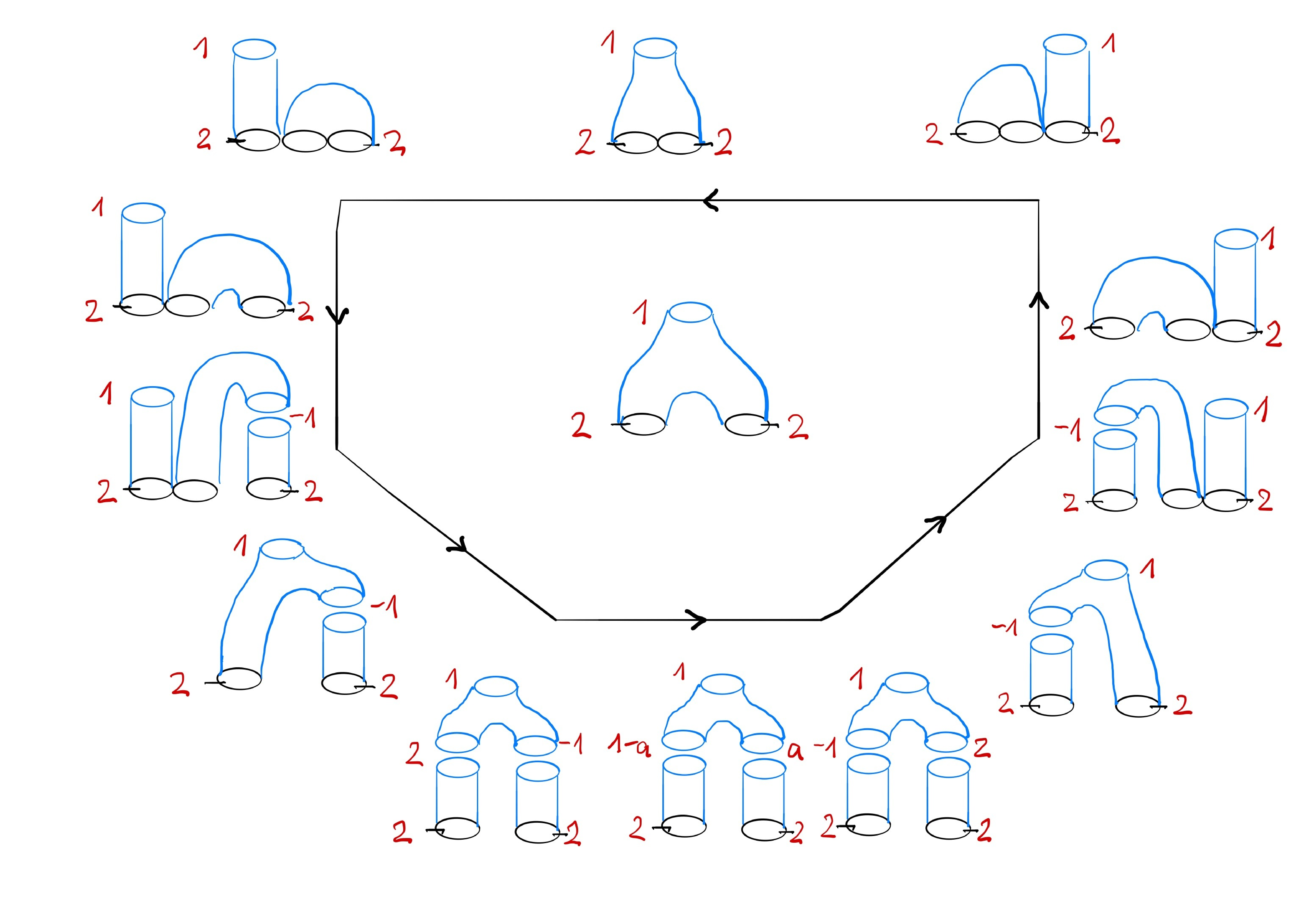}
\caption{The hexagon of Floer annuli}
\label{fig:Floer-annuli}
\end{center}
\end{figure}

In the figure the (black) bottom circles are boundary components, (blue) intermediate circles are
interior punctures (viewed as positive or negative  when going upwards or downwards),
and (red) numbers denote the weights. Such a family $\beta_\tau$ exists
because on each component of each broken curve the sum of negative
weights is greater or equal to the sum of positive weights. 

The annuli carry two marked points on their boundary circles (depicted
as black dashes) which are aligned with the interior puncture. Again, all 
interior punctures carry asymptotic markers (not drawn) that are aligned 
with the boundary marked points, also over broken curves and are
matching across each pair of positive/negative punctures.

Note that the bottom corner of the pentagon in Figure~\ref{fig:annuli-varying-modulus}
has been replaced by a new side over which the underlying stable
domain is fixed, but the weights at the positive/negative puncture vary
as depicted with $a\in[-1,2]$. 
Thus the conformal modulus is $0$
along the top side, and $\infty$ along the three bottom sides.  

We fix a nonnegative Hamiltonian $K:T^*M\to\R$ as in~\S\ref{sec:cont-coproduct}.
For $\tau\in\PP$ we denote by $\Sigma_\tau$ the corresponding
(possibly broken) annulus with one positive interior puncture $z_+$ and two
numbered boundary marked points $z_1,z_2$ on the boundary components
$C_1,C_2$, equipped with the $1$-form $\beta_\tau$. Given $x\in
FC_*(K)$ we define the moduli space 
\begin{align*}
   \PP(x) := \{&(\tau,u)\mid \tau\in\PP,\; u:\Sigma_\tau\to
   T^*M,\;(du-X_K\otimes\beta_\tau)^{0,1}=0,\cr
   & u(z_+)=x,\;u(C_i)\subset M \text{ for }i=1,2\}, 
\end{align*}
where the condition $u(z_+)=x$ is understood as being $C^\infty$-convergence
$u(s,\cdot)\to x$ as $s\to\infty$ in cylindrical coordinates
$(s,t)\in[0,\infty)\times S^1$ near the positive puncture $z_+$.
By Anosov~\cite{Anosov}, 
the restriction $u|_{C_i}$ 
can be uniquely parametrized over $[0,1]$ as an $H^1$-curve proportionally to arclength
such that time $0$ corresponds to the marked point $z_i$,
$i=1,2$. Viewing $u|_{C_i}$ with these parametrizations thus yields a
boundary evaluation map
$$
   \ev_\p:\PP(x)\to\Lambda\times\Lambda,\qquad (\tau,u)\mapsto(u|_{C_1},u|_{C_2}).
$$
Note that this map is also canonically defined over the boundary of
$\PP$. Indeed, this is clear everywhere except possibly over the two
vertical sides where one boundary loop is split into two.
There one component of $\Sigma_\tau$ is an annulus without interior
puncture, on which the map $u$ is therefore constant (see the next subsection).
Hence in the split boundary loop one component is constant, and we map
it simply to the other component parametrized proportionally to arclength. 

The expected dimension of $\PP(x)$ is
$$
   \dim\PP(x) = n\chi(\Sigma_\tau) + \CZ(x) + \dim\PP = \CZ(x) + 2 - n,
$$
where $\chi(\Sigma_\tau)=-1$ is the Euler characteristic of the
punctured annulus. However, the moduli space $\PP(x)$ is not
transversely cut out over the vertical sides of $\PP$. Indeed, the
moduli space of non-punctured annuli appearing there has Fredholm
index $n\chi(A) + 1=1$, where $\chi(A)=0$ is the Euler characteristic
of the annulus $A$ and the $+1$ corresponds to the varying conformal modulus.
But the actual dimension of this space is $n+1$, where $n$ is the
dimension of the space of constant maps $A\to M$. In the following
subsections we explain how to achieve transversality by perturbing
the Floer equation by a section in the obstruction bundle. 

\subsection{Moduli problems and obstruction bundles}\label{sec:moduli problems}

To facilitate the discussion in the next subsection, we introduce in
this subsection a general setup for moduli problems and obstruction
bundles. Our notion of a moduli problem will be a slight
generalization of that of a $G$-moduli problem in~\cite{Cieliebak-Mundet-Salamon}
for the case of the trivial group $G$, which allows us to work with
integer rather than rational coefficients. 

A {\em moduli problem} is a quadruple $(\cB,\cF,\cS,\cZ)$ with the following properties:
\begin{itemize}
\item $p:\cF\to\cB$ is a Banach fibre bundle over a Banach manifold;
\item $\cZ\subset\cF$ is a Banach submanifold transverse to the fibres;\footnote{In particular, $T_z\cZ\subset T_z\cF$ is a closed subspace which has a closed complement for all $z\in\cZ$.}
\item $\cS:\cB\to\cF$ is a smooth section such that the solution set
$$
    \cM := \cS^{-1}(\cZ)\subset\cB
$$
is compact and for each $b\in\cM$ the composed operator
$$
   D_b\cS:T_b\cB\stackrel{T_b\cS}\longrightarrow
   T_{\cS(b)}\FF\longrightarrow T_{\cS(b)}\cF/T_{\cS(b)}\cZ
$$
is Fredholm with constant index $\ind(\cS)=\ind(D_b\cS)$, and its determinant bundle 
$$
   \det(\cS) = \Lambda^{\rm top}\ker(D\cS)\otimes\Lambda^{\rm
     top}\coker(D\cS)^* \to \MM
$$
is oriented.
\end{itemize}

A {\em morphism} between moduli problems $(\cB,\cF,\cS,\cZ)$ and
$(\cB',\cF',\cS',\cZ')$ is a pair $(\psi,\Psi)$ with the following properties:
\begin{itemize}
\item $\psi:\cB\into\cB'$ is a smooth embedding;
\item $\Psi:\cF\to\cF'$ is a smooth injective bundle map covering
$\psi$ such that
$$
     \cS'\circ\psi=\Psi\circ\cS,\qquad 
     \cM'=\psi(\cM),\qquad \cZ'=\Psi(\cZ).
$$
Moreover, the linear operators $T_b\psi:T_b\cB\to T_{\psi(b)}\cB'$ 
and $D_z\Psi:T_z\cF/T_z\cZ\to T_{\Psi(z)}\cF/T_{\Psi(z)}\cZ$ induce
for each $b\in\cM$ isomorphisms
$$
        T_b\psi:\ker D_b\cS\to\ker D_{\psi(b)}\cS',\quad 
        D_{\cS(b)}\Psi:\coker D_b\cS\to\coker D_{\psi(b)}\cS'
$$
such that the resulting isomorphism from $\det(\cS)$ to $\det(\cS')$ is orientation preserving. 
\end{itemize}

\begin{proposition}
Each moduli problem $(\cB,\cF,\cS,\cZ)$ has a canonical {\em Euler class}
$$
   \chi(\cB,\cF,\cS,\cZ)\in H_{\ind(\cS)}(\cB;\Z).  
$$
Moreover, if $(\psi,\Psi)$ is a morphism between moduli problems $(\cB,\cF,\cS,\cZ)$ and
$(\cB',\cF',\cS',\cZ')$, then $\ind(\cS)=\ind(\cS')$ and
$$
   \psi_*\bigl(\chi(\cB,\cF,\cS,\cZ)\bigr)=\chi(\cB',\cF',\cS',\cZ')\in H_{\ind(\cS)}(\cB';\Z).
$$
\end{proposition}

\begin{proof}
This follows directly from the corresponding results in~\cite{Cieliebak-Mundet-Salamon}.
To construct the Euler class, we compactly perturb $\cS$ to a section $\wt\cS$
which is transverse to $\cZ$; then $\wt\cM=\wt\cS^{-1}(\cZ)$ is a
compact manifold of dimension $d=\ind(\cS)$ which inherits a canonical
orientation and thus represents a class in $H_d(\cB;\Z)$, and it is
easy to see that this class is independent of the choice of perturbation.
The assertion about morphisms is obvious. 
\end{proof}

A special case of a moduli problem arises if $\cF=\cE\to\cB$ is a
Banach {\em vector} bundle and $\cZ=\cZ_\cE$ is the zero section in $\cE$.
In this case $D_b\cS$ is the vertical differential of $\cS$ at
$b\in\cM=\cS^{-1}(0)$ and we arrive at the usual notion of a Fredholm section.
This is the setup considered in~\cite{Cieliebak-Mundet-Salamon};
the general case can be reduced to this one (via a morphism of moduli
problems) by passing to the normal bundle of $\cZ$. 

Consider now a moduli problem $(\cB,\cF,\cS,\cZ)$ such that
\begin{enum}
\item $\cM=\cS^{-1}(\cZ)\subset\cB$ is a smooth submanifold, and
\item $\ker(D_b\cS)=T_b\cM$ for each $b\in\cM$. 
\end{enum}
Then the cokernels $\coker(D_b\cS)$ fit together into the
smooth {\em obstruction bundle}
$$
   \cO := \coker(D\cS)\to\cM
$$
whose rank is related to the Fredholm index of $\cS$ by
$$
   \dim\cM = \ind(\cS) + \rk\,\cO.
$$
We thus obtain a finite dimensional moduli problem
$(\cM,\cO,0,\cZ_\cO)$, where $0:\cM\to\cO$ denotes the zero section
and $\cZ_\cO\subset\cO$ its graph. 

\begin{lemma}\label{lem:obstruction-bundle}
In the preceding situation there exists a canonical morphism of
moduli problems
$$
   (\iota,\exp):(\cM,\cO,0,\cZ_\cO)\to (\cB,\cF,\cS,\cZ),
$$
where $\psi=\iota:\cM\into\cB$ is the inclusion and $\exp:\cO\into\cF$
is a fibrewise exponential map. 
\end{lemma}

\begin{proof}
Choose $\cN\to\cZ$ a smooth Banach vector bundle
such that for each $z\in\cZ$, 
$$
   \cN_z\subset T_z\cF_{p(z)}\quad\text{and}\quad T_z\cF=T_z\cZ\oplus \cN_z.
$$
Since $\cN$ represents the normal bundle to $\cZ$ in $\cF$, we can assume that
$D\cS$ takes values in $\cN$ and $\cO$ is a subbundle of $\cN$
complementary to $\im D\cS$. Pick a fibrewise Riemannian metric on
$\cF$ whose exponential map restricts to a fibre preserving embedding
$$
   \exp:\cO\into\cF,\qquad \cO_z\into\cF_{p(z)}.
$$
Now it is easy to check that $(\iota,\exp)$ with the inclusion $\iota:\cM\into\cB$
defines a morphism $(\cM,\cO,0,\cZ_\cO)\to (\cB,\cF,\cS,\cZ)$. 
\end{proof}

In the situation of Lemma~\ref{lem:obstruction-bundle}, the Euler
class of $(\cB,\cF,\cS,\cZ)$ is therefore represented by the zero
set $\eta^{-1}(0)$ of a section $\eta:\cM\to\cO$ in the obstruction
bundle which is transverse to the zero section. Concretely, keeping
the notation from the proof, $\exp\circ\eta$ defines a section
of the fibre bundle $\cF|_\cM\to\cM$. We extend the bundle $\cO\to\cM$
to a bundle $\wt\cO\to\wt\cB$ on a neighbourhood $\wt\cB\subset\cB$ of
$\cM$ and $\eta$ to a section $\wt\eta$ of the bundle $\wt\cO\to\wt\cB$
vanishing near the boundary of $\wt\cB$. Then the perturbed section
$\wt\cS=\cS+\exp\circ\wt\eta$ of $\cF\to\cB$ is transverse to $\cZ$ 
and its solution set $\wt\cS^{-1}(\cZ)$ represents the Euler class of
$(\cB,\cF,\cS,\cZ)$. 

\begin{remark}[orientations]\label{rem:orientations}
In the situation of Lemma~\ref{lem:obstruction-bundle} we are given an
orientation of 
\begin{equation}\label{eq:detS}
   \det(\cS) = \Lambda^{\rm top}T\MM\otimes\Lambda^{\rm
     top}\OO^*.
\end{equation}
Let now $\eta:\MM\to\OO$ be a section transverse to the zero section.
Its zero set $\cA:=\eta^{-1}(0)\subset\MM$ is a submanifold and at each
$b\in\cA$ the linearization $D_b\eta:T_b\MM\to\OO_b$ is
surjective with kernel $\ker D_b\eta=T_b\cA$, so we get a
canonical isomorphism of line bundles
$$
   \Lambda^{\rm top}T\MM|_\cA \cong \Lambda^{\rm top}T\cA\otimes\Lambda^{\rm
     top}\OO|_\cA\to\cA.
$$
Combined with~\eqref{eq:detS} this yields a canonical isomorphism
$$
   \Lambda^{\rm top}T\cA\cong\det(\cS)|_\cA,
$$
so an orientation of $\det(\cS)$ induces an orientation of $\cA$.  
In the case $\ind(\cS)=0$ this can be made more explicit as follows. 
Then $\rk\,\OO=\dim\MM$ and an orientation of $\det(\cS)$ induces
an isomorphism
$$
   \Lambda^{\rm top}T\MM \cong \Lambda^{\rm top}\OO. 
$$
For $b\in\eta^{-1}(0)$ we define the sign $\sigma(b)$ to be $+1$ if the
isomorphism $D_b\eta:T_b\MM\stackrel{\cong}\longrightarrow\OO_b$ 
preserves orientations, and $-1$ otherwise. Then the signed count
$$
   \chi(\OO) = \sum_{b\in\eta^{-1}(0)}\sigma(b) 
$$
is the {\em Euler number} of the obstruction bundle $\OO\to\MM$. 
\end{remark}

Finally, consider a moduli problem $(\cB,\cF,\cS,\cZ)$ which splits as follows:
\begin{itemize}
\item $p=(p_0,p_1):\cF=\cF_0\times_{\cB}\cF_1\to\cB$;
\item $\cZ=\cZ_0\times_{\cB}\cZ_1$;
\item $\cS=\cS_0\times\cS_1$ for sections $\cS_i:\cB\to\cF_i$ such
  that $\cS_1$ is transverse to $\cZ_1$. 
\end{itemize}

\begin{lemma}\label{lem:reduces-moduli-problem}
In the situation above there exists a reduced moduli problem
$$
  (\ol\cB,\ol\cF,\ol\cS,\ol\cZ)=\bigl(\cS_1^{-1}(\cZ_1),\cF_0|_{\ol\cB},\cS_0|_{\ol\cB},
  \cZ_0|_{\ol\cB}\bigr)  
$$
and a morphism $(\psi,\Psi)$ of moduli problems from
$(\ol\cB,\ol\cF,\ol\cS,\ol\cZ)$ to\break $(\cB,\cF,\cS,\cZ)$, with
$\psi:\ol\cB\into\cB$ the inclusion and $\Psi(f_0)=\bigl(f_0,S_1\circ p_0(f_0)\bigr)$.
\end{lemma}

\begin{proof}
Since $\cS_1$ is transverse to $\cZ_1$, it follows that
$\ol\cB\subset\cB$ is a submanifold and
$(\ol\cB,\ol\cF,\ol\cS,\ol\cZ)$ defines a moduli problem. 
Now it follows directly from the definitions that $(\psi,\Psi)$ as
in the lemma induces for $b\in\ol\cB$ the canonical identities
\begin{gather*}
   T_b\psi:\ker D_b\ol\cS = \ker D_b\cS_0\cap\ker D_b\cS_1 = \ker D_{\psi(b)}\cS,\cr
   D_{\ol\cS(b)}\Psi:\coker D_b\ol\cS = \coker(D_b\cS_0|_{\ker D_b\cS_1}) = \coker D_{\psi(b)}\cS,
\end{gather*}
hence it defines a morphism of moduli problems. 
\end{proof}

\subsection{Constant Floer annuli}\label{sec:constant-annuli}

In this subsection we apply the results of the previous subsection to moduli
spaces of annuli. We begin with a rather general setup. Let $(\Sigma,j)$
be a compact Riemann surface with boundary, and $(V,J)$ be an almost
complex manifold with a half-dimensional totally real submanifold
$L\subset V$. For $m\in\N$ and $p\in\R$ with $mp>2$ we consider the
Banach manifold
$$
   \BB = W^{m,p}\bigl((\Sigma,\p\Sigma),(V,L)\bigr)
$$
and the Banach space bundle $\EE\to\BB$ whose fibre over $u\in\BB$ is
$$
   \EE_u = W^{m-1,p}\bigl(\Sigma,\Hom^{0,1}(T\Sigma,u^*TV)\bigr). 
$$
Denote $\cZ_\cE$ the zero section. The Cauchy-Riemann operator
$$
   \pb u = (du)^{0,1} = \frac12\bigl(du+J(u)\circ du\circ j\bigr)
$$
defines a Fredholm section $\pb:\BB\to\EE$. Assuming a setup in which the space of solutions $\pb^{-1}(\cZ_\cE)$ is compact (e.g. if $J$ is tamed by an exact symplectic structure on $V$, the totally real submanifold $L$ is exact Lagrangian, and $\Sigma$ has a compact group of automorphisms), we obtain a moduli problem $(\cB,\cE,\dbar,\cZ_\cE)$.

{\bf Constant annuli of positive modulus. }
Now we apply the preceding discussion to the moduli space of constant
annuli appearing in the previous subsection. Consider a fixed annulus
$(\Sigma,j)$ of finite conformal modulus $R>0$, equipped with a $1$-form
$\beta$ as above satisfying $d\beta\leq 0$ and $\beta=2dt$ in
cylindrical coordinates near the two (negative) boundary loops. Let $K$ be the nonnegative Hamiltonian from~\S\ref{sec:cont-coproduct}. Then the Floer operator
$\pb_Ku:=(du-X_K\otimes\beta)^{0,1}$ defines a Fredholm section in the
appropriate bundle $\EE\to\BB$ over the Banach manifold 
$$
   \BB = W^{m,p}\bigl((\Sigma,\p\Sigma),(T^*M,M)\bigr).
$$
We denote its zero set by $\MM:=\pb_K^{-1}(0)$. For $u\in\MM$ the
usual energy estimate (see e.g.~\cite{Ritter}) gives 
$$
   E(u) = \frac12\int_\Sigma|du-X_K(u)\otimes\beta|^2{\rm vol}_\Sigma \leq
   - A_{2K}(u|_{\p\Sigma}) = 0,
$$
where the Hamiltonian action of $u|_{\p\Sigma}$ vanishes because both 
the Liouville form and the Hamiltonian $K$ vanish on the zero section $M$.
This implies that $du-X_K(u)\otimes\beta\equiv 0$. Since $X_K$
vanishes near the zero section, it follows that $du\equiv 0$ near $\p\Sigma$
and therefore, by unique continuation, $u$ is constant equal to a
point in $M$. Hence the moduli space
$$
   \MM = M
$$
consists of points in $M$, viewed as constant maps $\Sigma\to
M$. Since $X_K$ vanishes near the zero section, the Floer operator
$\pb_K$ agrees with the Cauchy-Riemann operator $\pb$ near $\MM$, so
we can and will replace $\pb_K$ by $\pb$ in the following discussion of
obstruction bundles. 

We identify $\Sigma$ with the standard annulus $[0,R]\times\R/\Z$ and
its trivial tangent bundle $T\Sigma = \Sigma\times\C$. 
Consider a point $u\in M$, viewed as a constant map $u:\Sigma\to M$. 
We identify 
$$
   T_u^*M= \R^n,\qquad T_uM=i\R^n,\qquad T_u(T^*M)=\C^n. 
$$
Then we have
\begin{align*}
  T_u\BB &= W^{m,p}\bigl((\Sigma,\p\Sigma),(\C^n,i\R^n)\bigr), \cr
  \EE_u &= W^{m-1,p}\bigl(\Sigma,\Hom^{0,1}(\C,\C^n)\bigr) = W^{m-1,p}(\Sigma,\C^n),
\end{align*}
where for the last equality we use the canonical isomorphism
$$
   \Hom^{0,1}(\C,\C^n)\stackrel{\cong}\longrightarrow \C^n,\qquad \eta\mapsto\eta(\p_s).
$$
With these identifications, the linearized Cauchy-Riemann operator reads
$$
   D_u\pb:W^{m,p}\bigl((\Sigma,\p\Sigma),(\C^n,i\R^n)\bigr) \to
   W^{m-1,p}(\Sigma,\C^n),\qquad \xi\mapsto\p_s\xi+i\p_t\xi. 
$$
An easy computation using Fourier series (see~\cite{Ci95}) shows that
$$
   \ker(D_u\pb)=i\R^n=T_uM,\qquad \coker(D_u\pb)=\R^n=T_u^*M.
$$
So the Cauchy-Riemann operator, and thus the Floer operator,
satisfies conditions (i) and (ii) in the previous subsection with the obstruction bundle 
$$
   \cO=\coker(D\pb_K)\cong T^*M\to M=\MM,
$$
and Lemma~\ref{lem:obstruction-bundle} implies

\begin{corollary}\label{cor:constant-annuli}
In the preceding situation there exists a canonical morphism of
moduli problems
$$
   (\iota,I):(M,T^*M,0,\cZ_{T^*M})\to (\cB,\cE,\pb,\cZ_{\cE}),
$$
where $\iota:M\into\cB$ is the inclusion as constant maps and $I$
converts a cotangent vector into a constant $(0,1)$-form. 
\qed
\end{corollary}

Note in particular that $\pb_K$ has index zero. 
A section in the obstruction bundle transverse to the zero section
corresponds under the isomorphism $\cO\cong T^*M$ to a $1$-form $\eta$ on $M$ with nondegenerate
zeroes $p_1,\dots,p_k$, and the zero set of the perturbed Floer
operator $\pb_K+\wt\eta$ consists of $p_1,\dots p_k$ viewed as
constant maps $\Sigma\to M$. 
Having chosen the orientation of $\det(\dbar)$ to be induced by the canonical isomorphism $TM\cong T^*M$, we obtain that the signed count 
$$
  \sum_{i=1}^k\sigma(p_i) = \chi(T^*M)
$$
agrees  
with the Euler number of $T^*M$. Note that
the Euler number of $T^*M$ equals the Euler characteristic of $M$ 
(this follows from the canonical isomorphism $T^*M\cong TM$ and the Poincar\'e-Hopf theorem). 

{\bf Constant annuli of modulus zero. } 
Annuli of conformal modulus zero can be viewed as moduli problems in
two equivalent ways. For the first view, we take as domain the compact
region $A\subset\C$ bounded by two circles touching at one point, the node.
Given $(V,J)$ and $L\subset V$ as above, we therefore obtain a moduli
problem $(\cB^A,\cE^A,\cS^A,\cZ_{\cE^A})$ with 
$$
   \BB^A = W^{m,p}\bigl((A,\p A),(V,L)\bigr),\quad \EE^A_u =
   W^{m-1,p}\bigl(A,\Hom^{0,1}(TA,u^*TV)\bigr),
$$
the Cauchy-Riemann operator $\cS^A=\pb^A$, and the zero section $\cZ_{\cE^A}\subset\cE^A$.  

For the second view, we take as domain the closed unit disk
$D\subset\C$ with $\pm i$ viewed as nodal points which are identified. This gives rise to a 
moduli problem $(\cB^D,\cF^D,\cS^D,\cZ^D)$ with 
\begin{align*}
  \BB^D &= W^{m,p}\bigl((D,\p D),(V,L)\bigr),\cr
  \cF^D &= \cE^D\times(L\times L),\qquad
  \EE^D_u = W^{m-1,p}\bigl(D,\Hom^{0,1}(TD,u^*TV)\bigr),\cr
  \cS^D &= \pb^D\times\ev:\cB^D\to\cE^D\times(L\times L),\qquad \ev(u)=\bigl(u(i),u(-i)\bigr),\cr 
  \cZ^D &= \cZ_{\cE^D}\times\Delta,\qquad \Delta=\{(q,q)\mid q\in L\}\subset L\times L. 
\end{align*}
Note that the indices of the two moduli problems agree, 
$$
   \ind(\cS^D) = \ind(\pb^D)-n = \ind(\cS^A). 
$$
Let $\phi:D\to A$ be a continuous map which maps $\pm i$ onto the
nodal point and is otherwise one-to-one, and which is biholomorphic in
the interior.\footnote{We may construct $\phi$ as a composition $\phi=\vartheta\circ\log\circ\psi$ where $\psi$ is the M\"obius transformation sending $D$ onto the upper halfplane $\H$ with $\psi(-i)=0$ and $\psi(i)=\infty$, $\log$ is the logarithm sending $\H$ onto the strip $S=\{z\in\C\mid 0\leq\Im z\leq\pi\}$, and $\vartheta$ is the M\"obius transformation sending $S$ onto $A$.}
Then composition with $\phi$ defines a diffeomorphism
$$
   \cB^D\supset\ev^{-1}(\Delta)\cong \cB^A
$$
(where we use as area form on $A$ the pullback under $\phi$ of an
area form on $D$). Since $\ev:\cB^D\to L\times L$ is transverse to the
diagonal $\Delta$, we are in the situation of Lemma~\ref{lem:reduces-moduli-problem}.
We conclude that there exists a morphism of moduli problems 
$$
   (\psi,\Psi):(\cB^A,\cE^A,\cS^A,\cZ_{\cE^A})\to (\cB^D,\cF^D,\cS^D,\cZ^D),
$$
where $\psi:\cB^A=\ev^{-1}(\Delta)\into\cB^D$ is the inclusion and
$\Psi(u;\eta)=\bigl(u;\eta,\ev(u)\bigr)$.

Now we specialize to the case $(V,L) = (T^*M,M)$ with its canonical
almost complex structure $J$. Then both solution spaces
$\cM^A=(\pb^A)^{-1}(0)$ and $\cM^D=(\cS^D)^{-1}(\cZ^D)=(\pb^D)^{-1}(0)=M$
consist of constant maps to $M$. Moreover, in view of the preceding
discussion and the fact that the Cauchy-Riemman operator
$\pb^D:\cB^D\to\cE^D$ over the disk is transverse to the zero section,
they both satisfy the hypotheses (i) and (ii) of
Lemma~\ref{lem:obstruction-bundle}, so combined with the
preceding discussion we obtain

\begin{corollary}\label{cor:constant-mod-zero-annuli}
There exists a commuting diagram of morphisms of moduli problems
\begin{equation*}
\xymatrix
@C=30pt
{
   (\cB^A,\cE^A,\cS^A,\cZ_{\cE^A})\ar[r]^{(\psi,\Psi)} & (\cB^D,\cF^D,\cS^D,\cZ^D) \cr
   (M,T^*M,0_{T^*M},\cZ_{T^*M}) \ar[u]^{(\iota^A,\Psi^A)}
  \ar[r]^{(\id,\exp)} & (M,M\times M,\ev,\Delta) \ar[u]^{(\iota^D,\Psi^D)}
}
\end{equation*}
where $\iota^A:M\into\cB^A$ and $\iota^D:M\into\cB^D$ are the
inclusions as constant maps, the bundle $M\times M\to M$ is given
by projection onto the first factor, and $\exp:T^*M\to M\times
M$ is the composition of the isomorphism $T^*M\cong TM$ induced by a
metric on $M$ with the exponential map $TM\to M\times M$.  
Thus the Euler class of each of these moduli problems is represented 
by the nondegenerate zeroes $p_1,\dots,p_k$ of a $1$-form $\eta$ on
$M$ (or equivalently, of a vector field $v$ on $M$), with signs that
add up (up to a global sign) to the Euler characteristic $\chi$ of $M$. 
\qed
\end{corollary}

\subsection{Proof of Theorem~\ref{thm:cont-loop-iso}}\label{sec:PD-loop-iso}

Now we can conclude the proof of Theorem~\ref{thm:cont-loop-iso}.

For $x\in FC_*(K)$ consider the moduli space $\PP(x)$ of Floer
annuli described in~\S\ref{sec:Floer-annuli} with its boundary
evaluation map $\ev_\p:\PP(x)\to\Lambda\times\Lambda$. Pick a $1$-form
$\eta$ on $M$ with nondegenerate zeroes $p_1,\dots,p_k$. As in~\S\ref{sec:constant-annuli}, we view $\eta$ as a section of the
obstruction bundle over the vertical sides of the hexagon in
Figure~\ref{fig:Floer-annuli}. We extend this section by a cutoff
function to a section $\wt\eta$ over the whole hexagon and add it as a right hand side to the
Floer equation. We choose the data such that the moduli space $\PP(x)$
is transversely cut out, and thus defines a compact manifold with
corners of dimension $\CZ(x)+2-n$. 

We may assume without loss of generality that $M$ is connected. We
pick a $C^2$-small Morse function $V:M\to\R$ with a unique maximum at
$q_0\in M$ such that $p_1,\dots,p_k$ flow to $q_0$ under the positive
gradient flow of $V$. Let $MC_*(S)$ denote the Morse complex of the
perturbed energy functional
$$
   S:\Lambda\to\R,\qquad S(q):=\int_0^1\bigl(|\dot q|^2-V(q)\bigr)dt
$$
(note that there is no factor $1/2$ in front of $|\dot q|^2$). 
For $x\in FC_*(K)$ and $a,b\in MC_*(S)$ we define
$$
   \PP(x;a,b) := \{(\tau,u)\in\PP(x)\mid \ev_\p(u)\in W^+(a)\times W^+(b)\},
$$
where $W^+(a)$ is the stable manifold of $a$ with respect to the
negative pseudo-gradient flow of $S$. Recall that the boundary evaluation map
involves reparametrization of the boundary loops proportionally to arclength.
For generic choices, these are manifolds of dimension
$$
   \dim\PP(x;a,b) = \CZ(x)-\ind(a)-\ind(b)+2-n.
$$
If the dimension is $0$ these spaces are compact and their signed counts
$$
   \Theta_1(x) := \sum_{a,b}\#\PP_{\dim=0}(x;a,b)\,a\otimes b
$$
define a degree $2-n$ map
\begin{align*}
   \Theta_1:  FC_*(K)\to MC_*(S)\otimes MC_*(S).
\end{align*}
Next we consider a $1$-dimensional moduli space $\PP_{\dim=1}(x;a,b)$
and compute its boundary. Besides splitting off index $1$ Floer cylinders 
and negative pseudo-gradient flow lines, which give rise to the term
$[\p,\Theta_1]$, there are contributions from the sides of the
hexagon in Figure~\ref{fig:Floer-annuli} which we analyze separately.
Note that the indices now satisfy
$$
   \CZ(x)-\ind(a)-\ind(b)=n-1. 
$$

\underline{Vertical left side}: Here the broken curves consist of a
half-cylinder attached at a boundary node to an annulus without
interior puncture, where the two boundary loops flow into $a,b$ under
the negative pseudo-gradient flow of $S$. By the discussion in~\S\ref{sec:constant-annuli} the moduli space of annuli is
$[0,\infty]\times\eta^{-1}(0)$, where $[0,\infty]$ encodes the
conformal modulus and $\eta^{-1}(0)$ consists of the points
$p_1,\dots,p_k$ (with signs $\sigma(p_i)$). In particular, we must
have $b=q_0$ and therefore $\ind(b)=\ind(q_0)=0$.
The half-cylinders belong to the moduli space 
\begin{align*}
   \MM(x;a) = \{ & u:[0,\infty)\times S^1\to T^*M\mid
     (du-X_K\otimes\beta)^{0,1}=\wt\eta,\cr
     & u(\infty,\cdot)=x,\ u(0,\cdot)\in W^+(a)\}.
\end{align*}
They carry a boundary nodal point which is aligned with the boundary marked
point $(0,0)$ and the puncture at $\infty$, and is therefore given by
$(0,1/2)$. The evaluation at the nodal point defines an evaluation map
$$
   \ev_{1/2}:\MM(x;a)\to M,\qquad u\mapsto u(0,1/2). 
$$
For the broken curve to exist this evaluation map must meet one of the
constant annuli, i.e.~one of the points $p_1,\dots,p_k\in M$, which
generically does not happen because
$$
   \dim\MM(x;a) = \CZ(x)-\ind(a) = n-1. 
$$
Hence the vertical left side gives no contribution to the boundary. 
   
\underline{Vertical right side}: 
Similarly, the vertical right side gives no contribution to the boundary.

\underline{Lower left side}: 
Here the broken curves consist of a disc with two interior punctures, one
positive and one negative, attached at its negative puncture to the
positive puncture of a half-cylinder along an orbit in $FC_*(-K)$,
where the two boundary loops flow into $a,b$ under the negative pseudo-gradient flow of $S$. 
By choosing the $1$-form $\beta$ equal to $dt$ on a long cylindrical
piece of the half-cylinder, we can achieve that these half-cylinders
are in one-to-one correspondence with broken curves consisting of a
cylinder with weights $(-1,1)$ and a half-cylinder with weights $(1,2)$,
as shown in the middle of Figure~\ref{fig:deg-half-cyl}. Reinterpreting
these curves as on the right of that figure, we see that their
count corresponds to the composition $-(\Gamma_1\otimes\Psi)(1\otimes c_0^F)$, where we recall that $\Gamma_1$ denotes the first term in the expression $\Gamma=\Gamma_1+\Gamma_2+\Gamma_3$ from~\eqref{eq:Gamma123}. 

\begin{figure}
\begin{center}
\includegraphics[width=.7\textwidth]{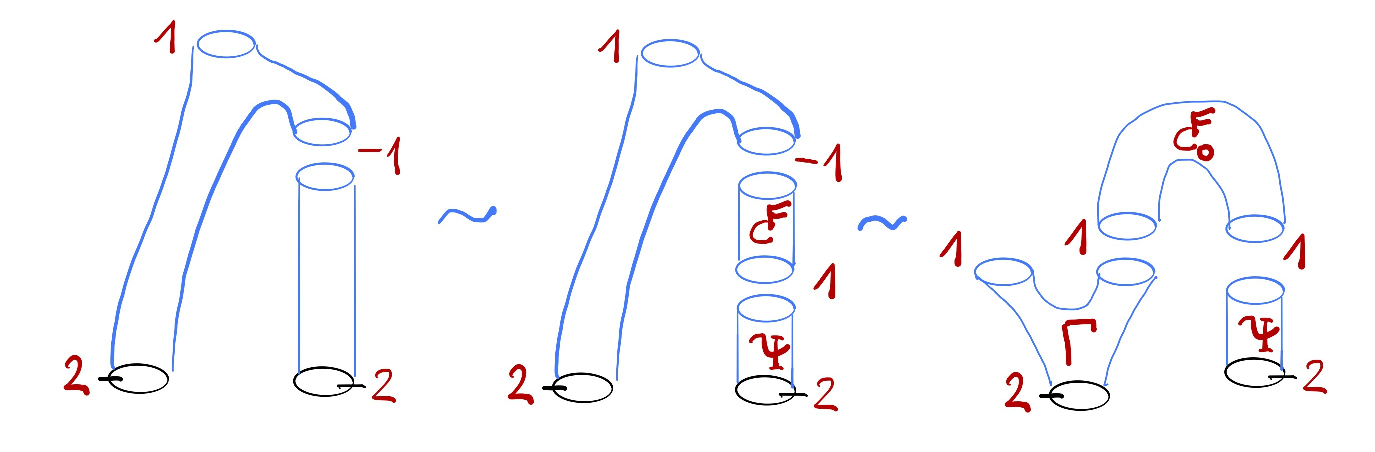}
\caption{Degenerating the half-cylinders}
\label{fig:deg-half-cyl}
\end{center}
\end{figure}

\underline{Lower right side}: 
Similarly, the contribution from the lower right side
corresponds to the composition $(\Psi\otimes\Gamma_1)(c_0^F\otimes 1)$.

The discussion so far shows that 
\begin{equation}\label{eq:Theta1} 
   [\p,\Theta_1] = (\Psi\otimes\Gamma_1)(\boldtau c_0^F\otimes 1)
   -(\Gamma_1\otimes\Psi)(1\otimes c_0^F) + \Theta^{\rm top} + \Theta^{\rm bottom},
\end{equation}
where $\Theta^{\rm top}$ and $\Theta^{\rm bottom}$ are the degree $1-n$ maps
arising from the contributions of the top and bottom sides of the
hexagon to the boundary of $\PP_{\dim=1}(x;a,b)$ which we discuss next.

\underline{Bottom side}: 
The family of broken curves on the bottom side can be deformed in an
obvious way to the family of broken curves shown in Figure~\ref{fig:Theta-bottom}.
Since the half-cylinders with weights $(2,2)$ define the map $\Psi$
and the family of 3-punctured spheres above them defines the
continuation coproduct $\lambda^F$ from~\S\ref{sec:cont-coproduct},
this shows that $\Theta^{\rm bottom}$ is equal to $-(\Psi\otimes\Psi)\lambda^F$.

\begin{figure}
\begin{center}
\includegraphics[width=.7\textwidth]{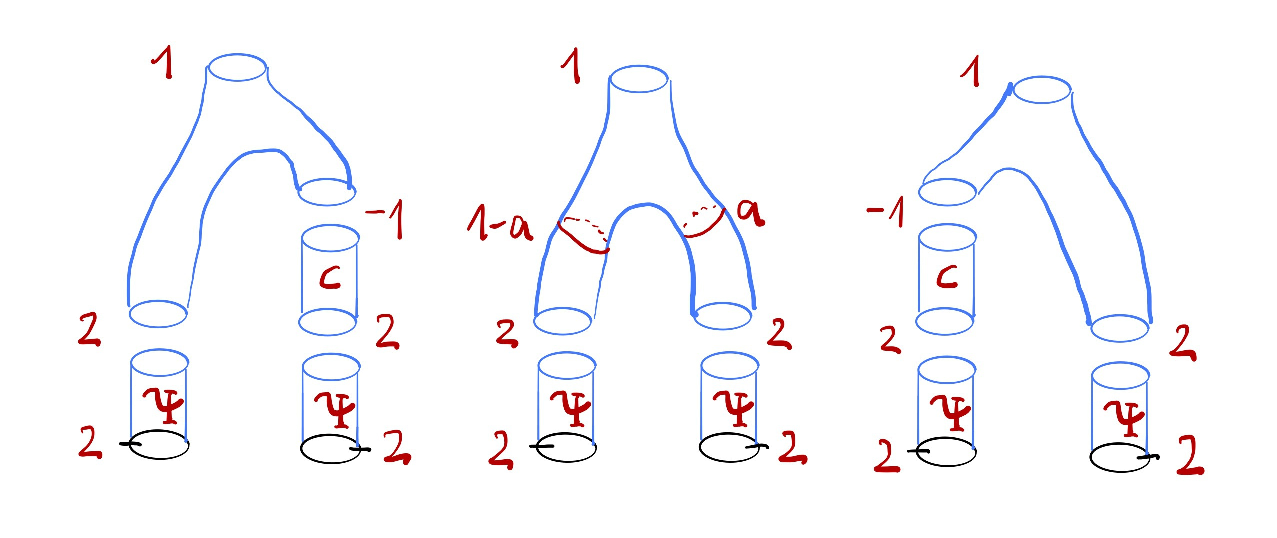}
\caption{Degenerating the curves on the bottom side}
\label{fig:Theta-bottom}
\end{center}
\end{figure}

\underline{Top side}: 
The family on the top side of the hexagon consists of punctured annuli
of modulus $0$, i.e., punctured discs with two nodal points on the
boundary that are identified to a node. Moreover, the boundary carries
two marked points that are separated by the nodal points and aligned
with the interior puncture. We wish to relate this family to the loop coproduct, 
but for this we face two problems: First, the
boundary loops carry two marked points whereas the loops for the
loop coproduct carry only one (the initial time $t=0$); and second,
the self-intersection of the boundary loop occurs at the nodal points
and not at one of the marked points. 

Both problems are resolved simultaneously as follows. We enlarge this $1$-parametric family to a
$2$-parametric family in which we keep the two boundary marked points
aligned, but drop the condition that the interior puncture is aligned
with them. The $2$-parametric family forms the hexagon shown in
Figure~\ref{fig:mod-zero-annuli}. Here the interior puncture is
depicted as a cross, the aligned boundary marked points as endpoints
of a dashed line, and the nodal points as thick dots. The bottom side
of the hexagon (drawn in black) corresponds to the $1$-parametric
family on the top side of Figure~\ref{fig:Floer-annuli}. 
Note that here we made a choice by letting the interior puncture move
freely {\em above} the dashed line connecting the two boundary marked points;
we could equally well have taken the mirror hexagon where the interior
puncture moves below the dashed line. 

\begin{figure}
\begin{center}
\includegraphics[width=1\textwidth]{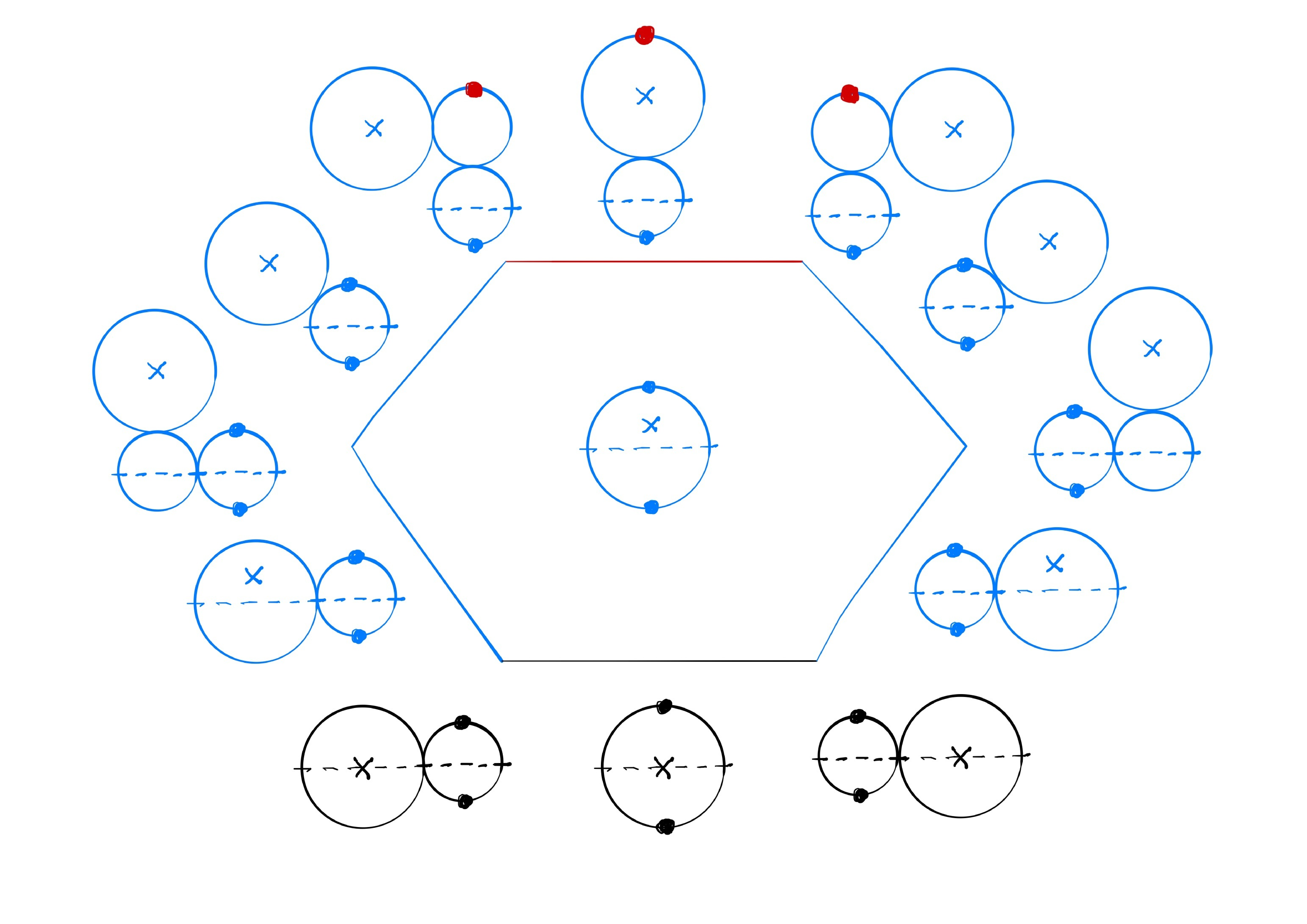}
\caption{Floer annuli of modulus zero}
\label{fig:mod-zero-annuli}
\end{center}
\end{figure}

The hexagon in Figure~\ref{fig:mod-zero-annuli} defines a deformation
from the bottom (black) side to the top side (drawn in red). The
configurations in this figure are to be interpreted as follows. 

$\bullet$ Each configuration has two boundary loops obtained by going
around in the counterclockwise direction: the first loop from the
bottom to the top nodal point, and the second one from the top to the
bottom nodal point. Each boundary loop carries a marked point. As
before, each boundary loop of the zero section is reparametrized
proportionally to arclength and then flown into a critical point on
$\Lambda$ under the negative pseudo-gradient flow of the functional $S:\Lambda\to\R$.

$\bullet$ In each configuration the unique component carrying the interior
puncture (which may be nonconstant) is drawn as a large disc, so the
small discs are all constant. In particular, each small disc carrying
the two nodal points is a constant annulus of modulus zero. Under the
perturbation of the Cauchy-Riemann equation described in
Corollary~\ref{cor:constant-mod-zero-annuli}, such a component lands on
the transverse zeroes $p_1,\dots,p_k$ of a $1$-form $\eta$ and further
flows into the basepoint $q_0$. 
In particular, since the signs add up to $\chi$ we have that all configurations on the upper and lower left sides
land in $R\chi q_0\otimes MC_*(S)$, while those on the upper and lower
right sides land in $MC_*(S)\otimes R\chi q_0$. 
The upper and lower left sides therefore compute $-(\Gamma_2\otimes\Psi)(1\otimes c_0^F)$, whereas the lower and upper right sides compute $(\Psi\otimes\Gamma_2)(\boldtau c_0^F\otimes 1)$, with $\Gamma_2$ being the second term in~\eqref{eq:Gamma123}.

Thus the hexagon in
Figure~\ref{fig:mod-zero-annuli} provides a chain homotopy $\Theta_2$ from $\Theta^{\rm top}$ (defined by the bottom side) to the
operation $\wt\Theta^{\rm top} + (\Psi\otimes\Gamma_2)(\boldtau c_0^F\otimes 1) -(\Gamma_2\otimes\Psi)(1\otimes c_0^F)$, where $\wt\Theta^{\rm top}$ is defined by the top side, i.e.
\begin{equation}\label{eq:Theta2}
[\p,\Theta_2]= \wt\Theta^{\rm top} + (\Psi\otimes\Gamma_2)(\boldtau c_0^F\otimes 1) -(\Gamma_2\otimes\Psi)(1\otimes c_0^F) - \Theta^{\rm top}.
\end{equation} 

$\bullet$ Consider now the top side. Since both marked points and the
black nodal point lie on the same constant component, we can remove
this component and replace the three points by one nodal/marked point
as shown in Figure~\ref{fig:Theta-top}. 
\begin{figure}
\begin{center}
\includegraphics[width=.7\textwidth]{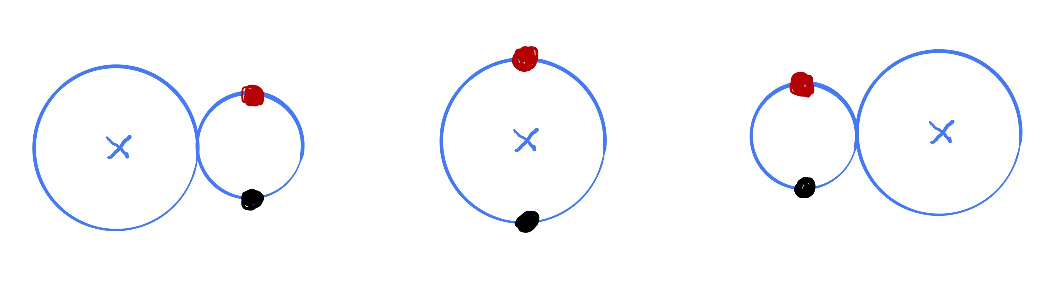}
\caption{Interpreting the curves on the top side}
\label{fig:Theta-top}
\end{center}
\end{figure}
The boundary of these configurations consists of loops $q:[0,1]\to M$
with one (black) marked/nodal point at time $0$ and an additional (red)
nodal point at time $s$ which moves from $0$ to $1$ as we traverse the
side from left to right.
In view of Corollary~\ref{cor:constant-mod-zero-annuli} and Remark~\ref{rem:orientations}, the map 
$\wt\Theta^{\rm top}:FC_*(K)\to MC_*(S)\otimes MC_*(S)$ is defined by counting isolated configurations consisting of punctured discs as in the definition of the moduli spaces $\cM(x)$ from~\S\ref{sec:Psi-symp}, additionally decorated with two marked points, with an incidence condition at the marked points, followed by semi-infinite negative pseudo-gradient lines of $S$ starting at the de-concatenated loops. Now we deform $\wt\Theta^{\rm top}$ once more by 
inserting a negative pseudo-gradient trajectory of $S$ of finite length $T\ge 0$ between the boundary loop of the disc and the loop on which we impose the incidence condition at the marked points. 
As $T\to\infty$ this becomes the chain map $\Psi:FC_*(K)\to MC_*(S)$
followed by the Morse theoretic coproduct $\lambda$, whereas on the boundary of the top side we see appear the terms $(\Psi\otimes\Gamma_3)(\boldtau c_0^F\otimes 1) -(\Gamma_3\otimes\Psi)(1\otimes c_0^F)$, with $\Gamma_3$ being the third term from~\eqref{eq:Gamma123}. 
We obtain therefore a homotopy $\Theta_3$ between the operation $\wt\Theta^{\rm top}$ defined by the top side and $\lambda\Psi +(\Psi\otimes\Gamma_3)(\boldtau c_0^F\otimes 1) -(\Gamma_3\otimes\Psi)(1\otimes c_0^F)$, i.e.
\begin{equation}\label{eq:Theta3}
[\p,\Theta_3]= \lambda\Psi +(\Psi\otimes\Gamma_3)(\boldtau c_0^F\otimes 1) -(\Gamma_3\otimes\Psi)(1\otimes c_0^F) - \wt\Theta^{\rm top}.
\end{equation}

Summing together equations~\eqref{eq:Theta1}, \eqref{eq:Theta2} and~\eqref{eq:Theta3}, and recalling that $\Gamma=\Gamma_1+\Gamma_2+\Gamma_3$, we obtain the desired relation~\eqref{eq:Theta-Floer}. 

For $n\neq 2$, the condition $\Theta c^F=0$ follows by an index
argument analogous to the proof of the relation $\lambda^Fc^F=0$ in
Proposition~\ref{prop:A2+symphom}. 
Together with the discussion at the beginning of this section, this
concludes the proof of Theorem~\ref{thm:cont-loop-iso}. 
\qed

\begin{remark}[Perturbation by $1$-form/vector field]\label{rem:perturbation}
Let us analyze how the perturbation by a $1$-form $\eta$ with transverse
zeroes propagates to the diagrams in the preceding proof. 
Along the left hand sides of the hexagon in Figure~\ref{fig:Floer-annuli},
we perturb the Floer operator by $\wt\eta$ at the second output. 
This continues along the left hand sides of the hexagon in
Figure~\ref{fig:mod-zero-annuli} as the perturbation of the constant
modulus zero annuli by the vector field $v$ corresponding to $\eta$
at the second output. 
As a result, the left hand configuration in Figure~\ref{fig:Theta-top}
is perturbed by applying the time-one-map $f$ of $v$ as we go
{\em counterclockwise} from the black to the red dot.
This means for $s\geq 0$ close to $0$ the evaluations of the corresponding loops $q:[0,1]\to M$
at time $0$ (the black dot) and $s$ (the red dot) are related by $q(s)=f(q(0))$.\\ 
Along the right hand sides of the hexagon in Figure~\ref{fig:Floer-annuli},
we perturb the Floer operator by $\wt\eta$ at the first output. 
As a result, the right hand configuration in Figure~\ref{fig:Theta-top}
is perturbed by applying the time-one-map $f$ of $v$ as we go
{\em clockwise} from the black to the red dot.
This means that for $s\leq 1$ close to $1$ the evaluations of the corresponding loops $q:[0,1]\to M$
at time $0$ (the black dot) and $s$ (the red dot) are related by
$q(0)=f(q(s))$, or equivalently, $q(s)=f^{-1}(q(0))$.\\ 
Therefore, the perturbation by the $1$-form $\eta$ on the Floer side
translates on the loop side into the perturbation by an $s$-dependent
vector field which agrees with $v$ near $s=0$ and with $-v$ near $s=1$. 
\end{remark}

\section{Relation to other Floer-type coproducts}\label{sec:other-coproducts} 

The continuation coproduct $\lambda^F$ discussed in the previous
sections descends to {\em positive action} symplectic homology
$SH_*^{>0}(D^*M)$ (since the action inequality implies that if the
input orbit is constant, then so must be the output orbits).  
In this section we relate $\lambda^F$ to other coproducts on
$SH_*^{>0}(D^*M)$ that have appeared in the literature, thus proving
Theorem~\ref{thm:main1} from the Introduction. 

In particular, we will prove that $\lambda^F$ agrees with the {\em
  Abbondandolo--Schwarz coproduct} $\lambda^{AS}$ defined in~\cite{AS-product-structures}. 
Abbondandolo and Schwarz defined in~\cite{AS-product-structures} the ring isomorphism
$\Psi_*:SH_*(D^*M)\stackrel{\cong}\longrightarrow H_*\Lambda$ 
and they asserted~\cite[Theorem 1.4]{AS-product-structures} that its
reduction modulo constant loops
$\Psi_*^{>0}:SH_*^{>0}(D^*M)\stackrel\cong\longrightarrow
H_*(\Lambda,\Lambda_0)$ intertwines the coproduct $\lambda^{AS}$ with
the loop coproduct $\lambda$. However, to our knowledge no proof of
this result has appeared. We will actually give two proofs in this
section: the first one uses Theorem~\ref{thm:cont-loop-iso} and the
identification $\lambda^F=\lambda^{AS}$, the second one uses a direct
argument and suitable interpolating moduli spaces.  

This section is structured as follows. In~\S\ref{sec:varying-weights-coproduct} we recall from~\cite{CO-cones} the definition of the varying weights coproduct $\lambda^w$, which coincides with $\lambda^F$ by~\cite[Lemma~7.2]{CO-cones} and which can be more easily related to $\lambda^{AS}$. In~\S\ref{sec:AS-coproduct} we recall from~\cite{AS-product-structures}
the definition of the Abbondandolo--Schwarz coproduct $\lambda^{AS}$. 
In~\S\ref{sec:varying-weights-AS-coproduct} we show that $\lambda^{AS}$ is equal to $\lambda^w$. 
In~\S\ref{sec:symplectic-Morse-coproducts} we prove directly that $\lambda^{AS}$
corresponds to the loop coproduct $\lambda$ under the
isomorphism $SH_*^{>0}(D^*M)\cong H_*(\Lambda,\Lambda_0)$. 

The situation is summarized in the following diagram. 
\begin{equation}\label{eq:isos-coproducts}
\xymatrix
@C=30pt
{
      \lambda \,\ar@{=}[r]^-{\S{\tiny \ref{sec:symplectic-Morse-coproducts}}}
      \ar@/_20pt/@{=}[rrr]_-{\text{Thm.}{\tiny \ref{thm:cont-loop-iso}}}
      &\lambda^{AS} \ar@{=}[r]^-{\S{\tiny \ref{sec:varying-weights-AS-coproduct}}}
      &\lambda^w \ar@{=}[r]^-{\mbox{\tiny \cite{CO-cones}}}
      &\lambda^F.
 }
\end{equation}

The whole discussion concerns the free loop space, but
it carries over verbatim to the based loop space. 

For simplicity, we assume throughout this section that $M$ is oriented
and we use untwisted coefficients in a commutative ring $R$; the
necessary adjustments in the nonorientable case and with 
twisted coefficients are explained in Appendix~\ref{sec:local-systems}.
We denote 
$$
   S^1:=\R/\Z\quad\text{and}\quad\Lambda:=W^{1,2}(S^1,M). 
$$

\subsection{Varying weights coproduct}\label{sec:varying-weights-coproduct}

We recall the definition of the varying weights coproduct $\lambda^w$ on $SH_*^{>0}(V)$ from~\cite[\S7.1]{CO-cones}. Since there we actually describe the algebraically dual product on $SH_*^{<0}(V,\p V)$, we will recap in some detail the necessary notation and arguments.  
The construction goes back to Seidel, see also~\cite{Ekholm-Oancea}.  We work with a Liouville domain $V$ of dimension $2n$, the symplectic completion is denoted $\wh V=V\cup [1,\infty)\times \p V$ and the radial coordinate in the positive symplectization $[1,\infty)\times \p V$ is denoted $r$. 

Let $\Sigma$ be the genus zero Riemann surface with three punctures, one of them labeled as positive $\chi_+$ and the other two labeled as negative $\upsilon_-$, $\zeta_-$, endowed with cylindrical ends $[0,\infty)\times S^1$ at the positive puncture and $(-\infty, 0]\times S^1$ at the negative punctures. Denote $(s,t)$, $t\in S^1$ the induced cylindrical coordinates at each of the punctures. Consider a smooth family of $1$-forms $\beta_\tau\in \Omega^1(\Sigma)$, $\tau\in (0,1)$ satisfying the following conditions: 
\begin{itemize}
\item {\sc (nonpositive)} $d\beta_\tau\le 0$; 
\item {\sc (weights)} $\beta_\tau=dt$ near each of the punctures;
\item {\sc (interpolation)} we have $\beta_\tau=\tau dt$ on $[-R(\tau),0]\times S^1$ in the cylindrical end near $\upsilon_-$, and 
$\beta_\tau=(1-\tau) dt$ on $[-R(1-\tau),0]\times S^1$ in the cylindrical end near $\zeta_-$, for some smooth function $R:(0,1)\to \R_{>0}$. In other words, the family $\{\beta_\tau\}$ interpolates between a $1$-form which varies a lot near $\upsilon_-$ and very little near $\zeta_-$, and a $1$-form which varies a lot near $\zeta_-$ and very little near $\upsilon_-$;
\item {\sc (neck stretching)} we have $R(\tau)\to +\infty$ as $\tau\to 0$. 
\end{itemize} 
We can assume without loss of generality that for $\tau$ close to
$0$ we have $\beta_\tau=f_\tau(s)dt$ in the cylindrical end at
the negative puncture $\upsilon_-$, with $f'_\tau\le 0$, $f_\tau=1$
near $-\infty$, and $f_\tau=\tau$ on $[-R(\tau),0]$, and
similarly for $\tau$ close to $1$ in the cylindrical end at the negative puncture $\zeta_-$.

Let $H:\wh V\to\R$ be a convex smoothing localized near $\p V$ of a Hamiltonian which is zero on $V$ and linear with respect to $r$ with positive slope on $[1,\infty)\times \p V$. The Hamiltonian $H$ further includes a small time-dependent perturbation localized near $\p V$, so that all $1$-periodic orbits are nondegenerate. Assume the slope is not equal to the period of a closed Reeb orbit. Denote $\cP(H)$ the set of $1$-periodic orbits of $H$. The elements of $\cP(H)$ are contained in a compact set close to $V$. 

Let $J=(J_\tau^\zeta)$, $\zeta\in \Sigma$, $\tau\in(0,1)$ be a generic family of compatible almost complex structures, independent of $\tau$ and $s$ near the punctures, cylindrical and independent of $\tau$ and $\zeta$ in the symplectization $[1,\infty)\times \p V$.
For $x,y,z\in\cP(H)$ denote 
\begin{align*}
   \MM^1(x;y,z) := \bigl\{&(\tau,u)\;\bigl|\; 
   \tau\in(0,1),\ u:\Sigma\to \wh V, \cr
   & (du-X_H\otimes \beta_\tau)^{0,1}=0, \cr
   & \lim_{\stackrel{s\to+\infty}{\zeta=(s,t)\to \chi_+}}u(\zeta)=x(t),\cr 
   & \lim_{\stackrel{s\to-\infty}{\zeta=(s,t)\to \upsilon_-}} u(\zeta)=y(t),\quad \lim_{\stackrel{s\to-\infty}{\zeta=(s,t)\to \zeta_-}} u(\zeta)=z(t) \bigr\}.
\end{align*}
In the symplectization $[1,\infty)\times \p V$ we have $H\ge 0$ and therefore $d(H\beta)\le 0$, so that elements of the above moduli space are contained in a compact set. The dimension of the moduli space is 
$$
\dim\, \MM^1(x;y,z) = \CZ(x)-\CZ(y)-\CZ(z)-n+1. 
$$
When it has dimension zero the moduli space $\MM^1_{\dim=0}(x;y,z)$ is compact. When it has dimension $1$ the moduli space $\MM^1_{\dim=1}(x;y,z)$ admits a natural compactification into a manifold with boundary 
\begin{align*}
   \p\MM^1_{\dim=1}(x;y,z)
   &= \coprod_{\CZ(x')=\CZ(x)-1}\MM(x;x')\times\MM^1_{\dim=0}(x';y,z) \cr
   &\amalg\coprod_{\CZ(y')=\CZ(y)+1}\MM^1_{\dim=0}(x;y',z)\times\MM(y';y) \cr
   &\amalg\coprod_{\CZ(z')=\CZ(z)+1}\MM^1_{\dim=0}(x;y,z')\times\MM(z';z) \cr
   &\amalg\MM^1_{\tau=1}(x;y,z) \amalg\MM^1_{\tau=0}(x;y,z).
\end{align*}
Here $\MM^1_{\tau=1}(x;y,z)$ and $\MM^1_{\tau=0}(x;y,z)$ denote the fibers of the first projection $\MM^1_{\dim=1}(x;y,z)\to (0,1)$, $(\tau,u)\mapsto \tau$ near $1$, respectively near $0$. (By a standard gluing argument the projection is a trivial fibration with finite fiber near the endpoints of the interval $(0,1)$.) 

Consider the degree $-n+1$ operation 
$$
\lambda^w:FC_*(H)\to FC_*(H)\otimes FC_*(H)
$$
defined on generators by 
$$
\lambda^w(x)=\sum_{\CZ(y)+\CZ(z)=\CZ(x)-n+1} \# \MM^1_{\dim=0}(x;y,z) y\otimes z,
$$
where $\# \MM^1_{\dim=0}(x;y,z)$ denotes the signed count of elements in the $0$-dimen\-sional moduli space $\MM^1_{\dim=0}(x;y,z)$. 
Consider also the degree $-n$ operations 
$$
\lambda^w_i:FC_*(H)\to FC_*(H)\otimes FC_*(H), \qquad i=0,1
$$
defined on generators by 
$$
\lambda^w_i(x)=\sum_{\CZ(y)+\CZ(z)=\CZ(x)-n} \# \MM^1_{\tau=i}(x;y,z) y\otimes z,
$$
where $\# \MM^1_{\tau=i}(x;y,z)$ denotes the signed count of elements in the $0$-dimen\-sional moduli space $\MM^1_{\tau=i}(x;y,z)$. 

Denote by $\p^F$ the Floer differential on the Floer complex of $H$. The formula for $\p\MM^1_{\dim=1}(x;y,z)$ translates into the algebraic relation  
\begin{equation} \label{eq:product_with_weights}
\p^F \lambda^w + \lambda^w(\p^F\otimes\mathrm{id}+ \mathrm{id}\otimes \p^F)=\lambda^w_1-\lambda^w_0.
\end{equation}
We now claim that  
$$
\mathrm{Im}(\lambda^w_0)\subset FC_*^{=0}(H)\otimes FC_*(H),\quad \mathrm{Im}(\lambda^w_1)\subset FC_*(H)\otimes FC_*^{=0}(H). 
$$
To prove the claim for $\lambda^w_0$, note that this map can be
expressed as a composition $(c\otimes\mathrm{id})\circ \lambda_0$, where
$\lambda_0:FC_*(H)\to FC_*(\tau H)\otimes FC_*(H)$ is a
pair-of-pants coproduct with $\tau>0$ small, and $c:FC_*(\tau H)\to FC_*(H)$ is a
continuation map. Taking into account that $\tau H$ has no nontrivial
$1$-periodic orbits for $\tau$ small, and because the action decreases along continuation maps, we obtain
$c(FC_*(\tau H))\subset FC_*^{=0}(H)$, which proves the claim. The argument for $\lambda^w_1$ is similar.  

It follows that $\lambda^w$ induces a degree $-n+1$ chain map
\begin{equation}\label{eq:sigmaF}
   \lambda^w:FC_*^{>0}(H)\to FC_*^{>0}(H)\otimes FC_*^{>0}(H).
\end{equation}
Passing to the limit as the slope of $H$ goes to $+\infty$ we obtain the degree $-n+1$ \emph{varying weights coproduct} $\lambda^w$ on $SH_*^{>0}(V)$. 

\begin{proposition}[{\cite[Lemma~7.2]{CO-cones}}]
The continuation coproduct and the varying weights coproduct coincide on $SH_*^{>0}(V)$: 
$$
\lambda^F=\lambda^w. 
$$ 
\qed
\end{proposition}

\subsection{Abbondandolo--Schwarz coproduct}\label{sec:AS-coproduct} 

In this subsection we recall from~\cite{AS-product-structures} the
definition of a secondary pair-of-pants product on Floer homology
of a cotangent bundle, which we will
refer to as the {\em Abbondandolo--Schwarz coproduct} $\lambda^{AS}$.
We recall the notation and conventions from~\S\ref{sec:Floer}
regarding the Floer complex. In particular near the zero section $H(q,p)=\varepsilon |p|^2 + V(q)$ for a small $\varepsilon>0$ and a Morse function $V:M\to \R$ such that all nonconstant critical points of $A_H$ have action larger than $\min\, V$. 

For $x,y,z\in\Crit(A_H)$ set (see Figure~\ref{fig:lambdaF}) 
\begin{align*}
   \MM^{1,AS}& (x;y,z) \\ 
   & := \bigl\{ (\tau,u,v,w)\;\bigl|\;
   \tau\in[0,1],\ u:[0,\infty)\times S^1\to T^*M \cr
   & \qquad \, v,w: (-\infty,0]\times S^1\to T^*M,\ \pb_Hu=\pb_Hv=\pb_Hw=0, \cr
   & \qquad \, u(+\infty,\cdot)=x,\ v(-\infty,\cdot)=y,\ w(-\infty,\cdot)=z,\cr
   & \qquad \, v(0,t) = u(0,\tau t),\ 
   w(0,t)=u(0,\tau+(1-\tau)t)\bigr\}.
\end{align*}
Note that the matching conditions imply $u(0,\tau)=u(0,0)$. 

\begin{figure} [ht]
\centering
\input{lambdaF.pstex_t}
\caption{The moduli spaces $\MM^{1,AS}(x;y,z)$.}
\label{fig:lambdaF}
\end{figure}

\begin{lemma}[{\cite[\S5]{AS-product-structures}}]  \label{lem:secondary-coproduct-moduli}
For generic choices of Hamiltonian and almost complex structure the space $\MM^{1,AS}(x;y,z)$ is a transversely cut out
manifold of dimension 
$$
   \dim\MM^{1,AS}(x;y,z) = \CZ(x)-\CZ(y)-\CZ(z)-n+1. 
$$\qed
\end{lemma}

The dimension of $\MM^{1,AS}(x;y,z)$ is calculated in~\cite{AS-product-structures} using an equivalent description of the moduli space as follows. Define $\wt v,\wt w: (-\infty,0]\times [0,1]\to T^*M$ by $\wt v(s,t) = v(s,t)$ and $\wt w(s,t) = w(s,t)$, and also $\wt y, \wt z:[0,1]\to T^*M$ by $\wt y(t) = y(t)$ and $\wt z(t) = z(t)$. Then there is a canonical identification between elements of $\MM^{1,AS}(x;y,z)$ and elements of 
\begin{align*}
   \wt{\MM}^{1,AS} & (x;\wt y,\wt z) \\ 
   & := \bigl\{(\tau,u,\wt v,\wt w)\;\bigl|\;
   \tau\in[0,1],\ u:[0,\infty)\times S^1\to T^*M \cr
   & \qquad \, \wt v,\wt w: (-\infty,0]\times [0,1]\to T^*M,\cr 
   & \qquad \, \pb_Hu=\pb_H\wt
   v=\pb_H\wt w=0, \cr
   & \qquad \, u(+\infty,\cdot)=x,\ \wt v(-\infty,\cdot)=\wt y,\ \wt
   w(-\infty,\cdot)=\wt z,\cr
   & \qquad \, \bigl(\wt v(s,0),{\bf C}\wt v(s,1)\bigr)\in N^*\Delta,\ 
   \bigl(\wt w(s,0),{\bf C}\wt w(s,1)\bigr)\in N^*\Delta, \cr 
   & \qquad \, \wt v(0,t) = u(0,\tau t),\ 
   \wt w(0,t)=u(0,\tau+(1-\tau)t)\bigr\}.
\end{align*}
Here ${\bf C}:T^*M\to T^*M$ is the antisymplectic involution $(q,p)\mapsto(q,-p)$, $\Delta\subset M\times M$ is the diagonal, and $N^*\Delta\subset T^*(M\times M)$ its conormal bundle. The space $\wt\MM^{1,AS}(x;\wt y,\wt z)$ is a moduli space with jumping Lagrangian boundary conditions as in~\cite{AS2}, so for generic $H$ and $J$ it is a
transversely cut out manifold. Its dimension is given by the Fredholm 
index of the linearized problem~\cite[(37)]{AS-product-structures}.

If $\MM^{1,AS}(x;y,z)$ has dimension zero it is compact and defines a map
$$
   \lambda^{AS}:FC_*\to(FC\otimes FC)_{*-n+1},\qquad
   x\mapsto\sum_{y,z}\#\MM^{1,AS}_{\dim=0}(x;y,z)\,y\otimes z.
$$
If it has dimension $1$ it can be compactified to a compact
$1$-dimensional manifold with boundary 
\begin{align*}
   \p\MM^{1,AS}_{\dim=1}(x;y,z) 
   &= \coprod_{\CZ(x')=\CZ(x)-1}\MM(x;x')\times\MM^{1,AS}_{\dim=0}(x';y,z) \cr
   &\amalg\coprod_{\CZ(y')=\CZ(y)+1}\MM^{1,AS}_{\dim=0}(x;y',z)\times\MM(y';y) \cr
   &\amalg\coprod_{\CZ(z')=\CZ(z)+1}\MM^{1,AS}_{\dim=0}(x;y,z')\times\MM(z';z) \cr
   &\amalg\MM^{1,AS}_{\tau=1}(x;y,z) \amalg\MM^{1,AS}_{\tau=0}(x;y,z).
\end{align*}
Here the first three terms correspond to broken Floer cylinders and the
last two terms to the intersection of $\MM^{1,AS}(x;y,z)$ with the sets
$\{\tau=1\}$ and $\{\tau=0\}$, respectively. Therefore we have
\begin{equation}\label{eq:lambda-chain-map-Floer}
   (\p^F\otimes\id+\id\otimes\p^F)\lambda^{AS} + \lambda^{AS}\p^F = \lambda^{AS}_1-\lambda^{AS}_0,
\end{equation}
where for $i=0,1$ we set
$$
   \lambda^{AS}_i:FC_*\to(FC\otimes FC)_{*-n},\qquad
   x\mapsto\sum_{y,z}\#\MM^{1,AS}_{\tau=i}(x;y,z)\,y\otimes z.
$$
Let us look more closely at the map $\lambda^{AS}_1$. For $\tau=1$ the
matching conditions in $\MM^{1,AS}(x;y,z)$ imply that $w(0,t)=u(0,0)$ is a constant loop.
For action reasons $z$ must then be a critical point, so that $\mathrm{Im}(\lambda^{AS}_1)\subset FC_*(H)\otimes FC_*^{=0}(H)$. Similarly we have $\mathrm{Im}(\lambda^{AS}_0)\subset FC_*^{=0}(H)\otimes FC_*(H)$, and therefore
$\lambda^{AS}$ descends to a chain map  
\begin{equation} \label{eq:coproduct-positive-homology-matching}
   \lambda^{AS}:FH_*^{>0}\to(FH^{>0}\otimes FH^{>0})_{*-n+1}  
\end{equation}
with $FC_*^{>0}=FC_*(H)/FC_*^{=0}(H)$. Note that we have $FH_*^{>0}(H)\cong SH_*^{>0}(D^*M)$ for the quadratic Hamiltonians considered in this section.

\subsection{The varying weights coproduct equals the Abbondandolo--Schwarz coproduct}\label{sec:varying-weights-AS-coproduct} 

\begin{proposition}\label{prop:products-AS-weights}
Let $M$ be a closed connected oriented manifold. The secondary coproducts $\lambda^w$ defined via~\eqref{eq:sigmaF} and $\lambda^{AS}$ defined via~\eqref{eq:coproduct-positive-homology-matching} agree on $SH_*^{>0}(D^*M)$.
\end{proposition}

\begin{proof} We assume without loss of generality that the Hamiltonian used in the definition of the coproduct $\lambda^{AS}$ is the same as the one used in the definition of the coproduct $\lambda^w$, i.e. a convex smoothing of a Hamiltonian which vanishes on $D^*M$ and is linear with respect to the radial coordinate $r=|p|$ outside of $D^*M$. The point of the proof is to exhibit the Floer problem defining the moduli spaces $\MM^{1,AS}(x;y,z)$ for $\lambda^{AS}$ as a limiting case of the Floer problem defining the moduli spaces $\MM^1(x;y,z)$ for $\lambda^w$. 

Note first that for $0$-dimensional moduli spaces $\MM^{1,AS}_{\dim=0}(x;y,z)$ we can restrict $\tau$ to $(0,1)$. Given $\tau\in(0,1)$ a triple $(u,v,w)$ as in the definition of $\MM^{1,AS}(x;y,z)$ can be interpreted as a single map $\tilde u:\Sigma\to T^*M$ satisfying $(d\tilde u-X_H\otimes \beta_\tau)^{0,1}=0$,  where $\Sigma$ is a Riemann surface and $\beta_\tau$ is a $1$-form explicitly described as follows. The Riemann surface is 
$$
\Sigma=\R\times [-\tau,0]\, \amalg \, \R\times [0,1-\tau]\, /  \sim
$$ 
with  
$$
(s,-\tau)\sim (s,1-\tau), \quad (s,0^-)\sim (s,0^+) \quad \mbox{ for }s\ge 0,
$$
$$
(s,-\tau)\sim (s,0^-), \quad (s,0^+)\sim (s,1-\tau) \quad \mbox{for }s\le 0.
$$
(We use the notation $(s,0^-)$ for points in $\R\times\{0\}\subset \p(\R\times [-\tau,0])$, and $(s,0^+)$ for points in $\R\times\{0\}\subset \p(\R\times [0,1-\tau])$.) This is a smooth Riemann surface with canonical cylindrical ends $[0,\infty)\times S^1$ at the positive puncture and $(-\infty,0]\times \R/\tau \Z$ and $(-\infty,0]\times \R/(1-\tau)\Z$ at the negative punctures. See Figure~\ref{fig:tubes}. 

\begin{figure} [ht]
\centering
\input{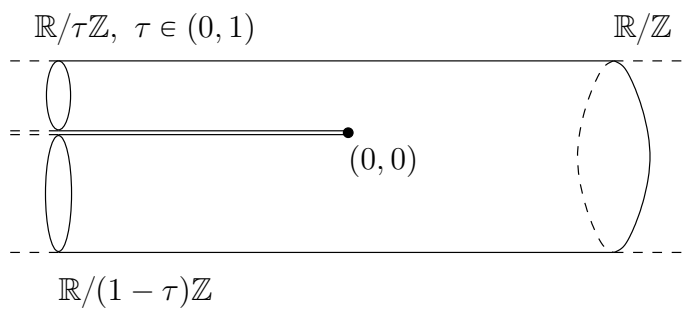}
\caption{A pair-of-pants $\Sigma$ with large cylindrical ends.}
\label{fig:tubes}
\end{figure}

A conformal parametrization of $\Sigma$ near the point $(0,0)$ is induced from the map $\C\to \C$, $z\mapsto z^2$.
\footnote{Consider the half-pair-of-pants $\Sigma_{\frac 1 2}=\R\times [-\tau,0]\, \amalg \, \R\times [0,1-\tau]\, /\sim$, where $(s,0^-)\sim (s,0^+)$ for $s\ge 0$. A conformal parametrization near $(0,0)$ is given by the map $z\mapsto z^2$ defined in a neighborhood of $0\in\{\Re z\ge 0\}$. This map actually establishes a global conformal equivalence between $H= \{z\in\C \, : \, \Re z\ge 0, \, 2(\Re z) (\Im z)\in [-\tau,1-\tau]\}$ and $\Sigma_{\frac 1 2}$. The Riemann surface $\Sigma$ admits a natural presentation as the gluing of two copies of $\Sigma_{\frac 1 2}$. Accordingly, it can be identified to $H\cup-H/\sim$ where the equivalence relation $\sim$ stands for suitable identifications of boundary components. The map $z\mapsto z^2$ defined in a neighborhood of $0\in H\cup-H/\sim$ provides a conformal parametrization of $\Sigma$ near the point $(0,0)$. }
The Riemann surface $\Sigma$ carries a canonical smooth closed $1$-form $dt$.
\footnote{Read through the identification of $\Sigma$ with $H\cup -H/\sim$, this is $2d(xy)$ in a neighborhood of $0$.} 
Upon identifying the cylindrical ends at the negative punctures with
$(-\infty,0]\times S^1$, this canonical $1$-form becomes equal to
$\tau dt$, respectively $(1-\tau)dt$ at those punctures. The
$1$-form $\beta_\tau$ is defined to be the {\it discontinuous}
$1$-form equal to $dt$ on the cylindrical end $[0,\infty)\times S^1$
  at the positive puncture, equal to $\frac 1\tau dt$ on the
  cylindrical end $(-\infty,0]\times \R/\tau \Z$ at the first
negative puncture, and equal to $\frac 1{1-\tau} dt$ on the
cylindrical end $(-\infty,0]\times \R/(1-\tau)\Z$ at the second
negative puncture. Equivalently, upon normalizing the cylindrical ends
at the negative punctures into $(-\infty,0]\times S^1$, the $1$-form
$\beta_\tau$ is simply $dt$. This discontinuous $1$-form
$\beta_\tau$ can be interpreted as a limit of $1$-forms which are
obtained by interpolating  
from $\tau\,dt$ and $(1-\tau)dt$
(near $0$) towards $dt$ (near $-\infty$) in the normalized cylindrical ends at the negative punctures, where the interpolation region shrinks and approaches $s=0$. It was noted in~\S\ref{sec:AS-coproduct} that the limit case defines a Fredholm problem $\wt{\MM}^{1,AS}(x;\wt y,\wt z)$ with jumping Lagrangian boundary conditions. The Fredholm problem before the limit is naturally phrased in terms of the Riemann surface $\Sigma$ without boundary, but it can be reinterpreted as a problem with Lagrangian boundary conditions by cutting $\Sigma$ open along $\{s=0\}$. As such, it converges in the limit to the Fredholm problem with jumping Lagrangian boundary conditions described above. By regularity and compactness, the two Fredholm problems are equivalent near the limit, and the corresponding counts of elements in $0$-dimensional moduli spaces are the same. 
\end{proof}

\subsection{Abbondandolo--Schwarz coproduct equals loop 
coproduct}\label{sec:symplectic-Morse-coproducts} 

Recall the Hamiltonian $H:S^1\times T^*M\to\R$ from~\S\ref{sec:AS-coproduct} and its fibrewise
Legendre transform $L:S^1\times TM\to\R$ from~\S\ref{ss:A2+loop}. 
Also recall from~\S\ref{ss:A2+loop} the notations concerning the Morse complex $MC_*$
of the action functional $S_L$ which we will use freely. In
particular, $\p$ denotes the Morse boundary operator and $\wt\lambda$
the coproduct from Remark~\ref{rem:GH-alternative}. 

We assume that $M$ is oriented and we use the Morse complex twisted by the local system $\sigma$ obtained by transgressing the second Stiefel-Whitney class. 

Following~\cite{AS-product-structures}, for $x\in\Crit(A_H)$ and $a\in\Crit(S_L)$ we
define 
\begin{align*}
   \MM(x) := \{u:[0,\infty)\times S^1\to T^*M\mid  &
     \ \pb_Hu=0,\\
     & \ u(+\infty,\cdot)=x,\ u(0,\cdot)\subset M\}
\end{align*}
and
\begin{equation} \label{eq:Mxa}
   \MM(x;a) := \{u\in\MM(x)\mid u(0,\cdot)\in W^+(a)\},
\end{equation}
where $W^+(a)$ is the stable manifold of $a$ for the negative pseudo-gradient flow of $S_L$. See Figure~\ref{fig:Psi}.

For generic $H$ these are manifolds of dimensions
$$
   \dim\MM(x) = \CZ(x),\qquad \dim\MM(x;a) = \CZ(x)-\ind(a). 
$$
The signed count of $0$-dimensional spaces $\MM(x;a)$ defines a chain map
\begin{equation*} 
   \Psi:FC_*\to MC_*,\qquad a\mapsto\sum_{\ind(a)=\CZ(a)}\#\MM(x;a)\,a.
\end{equation*}
The induced map on homology is an isomorphism 
$$
   \Psi_*:FH_*\stackrel{\cong}\longrightarrow MH_*\cong H_*(\Lambda;\sigma)
$$
intertwining the pair-of-pants product with the loop product.  

\begin{proposition} \label{prop:iso-ol-FH-MH}
The map $\Psi$ descends to an isomorphism on homology modulo the constant loops
$$
   \Psi_*:FH_*^{>0}\stackrel{\cong}\longrightarrow MH_*^{>0}\cong H_*(\Lambda,\Lambda_0;\sigma)
$$
which intertwines the Abbondandolo--Schwarz coproduct $\lambda^{AS}$ with the
loop coproduct $\lambda$. 
\end{proposition}

\begin{proof}
For $x\in\Crit(A_H)$ and $b,c\in\Crit(S_L)$
define
\begin{align*}
   \MM^+(x) := \bigl\{&(\sigma,\tau,u,v,w)\;\bigl|\;
   \sigma\in[0,\infty),\ \tau\in[0,1],\cr 
   & u:[\sigma,\infty)\times S^1\to T^*M, v,w: [0,\sigma]\times S^1\to T^*M, \cr
   & \ \pb_Hu=\pb_Hv=\pb_Hw=0, \cr
   & u(+\infty,\cdot)=x,\ v(0,t)\in M,\ w(0,t)\in M,\cr
   & v(\sigma,t) = u(\sigma,\tau t),\ 
   w(\sigma,t)= u(\sigma,\tau+(1-\tau)t)\bigr\},
\end{align*}
\begin{align*}
   \MM^+(x;b,c) := \{&(\sigma,\tau,u,v,w)\in\MM^+(x)\mid \cr 
   & v(0,\cdot)\in W^+(b),\ w(0,\cdot)\in W^+(c)\}, \cr
   \MM^-(x;b,c) := \bigl\{&(\sigma,\tau,u,\alpha,\beta,\gamma)\;\bigl|\;
   \sigma\in(-\infty,0],\ \tau\in[0,1],\ u\in\MM(x), \cr
   & \alpha=\phi_{-\sigma}(u(0,\cdot)),\ \beta\in W^+(b),\ \gamma\in
     W^+(c), \cr
   & \beta(t) = \alpha(\tau t),\ 
   \gamma(t)=\alpha(\tau+(1-\tau)t) \bigr\},
\end{align*}
where $\MM(x)$ was defined above and $\phi_s:\Lambda\to\Lambda$
for $s\geq 0$ denotes the flow of 
the negative pseudo-gradient of $S_L$. 
Note that
$\alpha,\beta,\gamma$ in the definition of $\MM^-(x;b,c)$ are actually
redundant and just included to make the definition more
transparent. As above it follows that for generic $H$ these spaces are
transversely cut out manifolds of dimensions $\dim\MM^+(x)=\CZ(x)-n+2$ and 
$$
   \dim\MM^+(x;b,c) = \dim\MM^-(x;b,c) = \CZ(x)-\ind(b)-\ind(c)-n+2.
$$
We set
\begin{align*}
   \MM^2(x;b,c) := \MM^+(x;b,c) \amalg \MM^-(x;b,c). 
\end{align*}
If this space has dimension zero it is compact and defines a map
$$
   \Theta:FC_*\to(MC\otimes MC)_{*-n+2},\qquad
   x\mapsto\sum_{b,c}\#\MM^2_{\dim=0}(x;b,c)\,b\otimes c.
$$
If it has dimension $1$ it can be compactified to a compact
$1$-dimensional manifold with boundary 
\begin{align*}
   \p\MM^2_{\dim=1}(x;b,c) 
   &= \coprod_{\CZ(x')=\CZ(x)-1}\MM(x;x')\times\MM^2_{\dim=0}(x';b,c) \cr
   &\amalg\, \coprod_{\ind(b')=\ind(b)+1}\MM^2_{\dim=0}(x;b',c)\times\MM(b';b) \cr
   &\amalg\, \coprod_{\ind(c')=\ind(c)+1}\MM^2_{\dim=0}(x;b,c')\times\MM(c';c) \cr
   &\amalg \, \coprod_{y,z}\MM^1_{\dim=0}(x;y,z)\times\MM(y;b)\times\MM(z;c) \cr
   &\amalg \, \coprod_a\MM(x;a)\times\wt\MM^1_{\dim=0}(a;b,c) \cr
   &\amalg\MM^2_{\tau=1}(x;b,c) \amalg\MM^2_{\tau=0}(x;b,c),
\end{align*}
where $\wt\MM^1(a;b,c)$ are the moduli spaces in
Remark~\ref{rem:GH-alternative} defining the coproduct $\wt\lambda$ with $f_t=\mathrm{id}$.
Here the first term corresponds to splitting off of Floer cylinders,
the second and third ones to splitting off of Morse pseudo-gradient lines, the
fourth one to $\sigma=+\infty$, the fifth one to $\sigma=-\infty$, and
the last two terms to the intersection of $\MM^2(x;b,c)$ with the sets
$\{\tau=1\}$ and $\{\tau=0\}$, respectively. The intersections
of $\MM^\pm(x;b,c)$ with the set $\{\sigma=0\}$ are equal with
opposite orientations and thus cancel out. Therefore we have
\begin{equation}\label{eq:lambda-chain-homotopy}
   (\p\otimes\id+\id\otimes\p)\Theta + \Theta\p^F =
  (\Psi\otimes\Psi)\lambda^{AS} - \wt\lambda\Psi + \Theta_1-\Theta_0,
\end{equation}
where for $i=0,1$ we set
$$
   \Theta_i:FC_*\to(MC\otimes MC)_{*-n+1},\qquad
   x\mapsto\sum_{b,c}\#\MM^2_{\tau=i}(x;b,c)\,b\otimes c.
$$
Arguing as in the previous subsection, we see that the $\Theta_0$ has image in $MC_*^{=0}\otimes MC_*$, and $\Theta_1$ has image in $MC_*\otimes MC_*^{=0}$. 
Together with equation~\eqref{eq:lambda-chain-homotopy} this shows
that $\Theta$ descends to a map 
$$
   \Theta:FC_*^{>0}\to(MC^{>0}\otimes MC^{>0})_{*-n+2}
$$
between the positive chain complexes which is a chain homotopy between
$(\Psi\otimes\Psi)\lambda^{AS}$ and $\wt\lambda\Psi$, which concludes the proof. 
\end{proof}

\section{Loop coproduct for odd-dimensional spheres}\label{sec:spheres}

In this section we compute the loop coproduct on reduced loop homology
$\ol{H}_*(\Lambda S^n)=H_*(\Lambda S^n)$ of odd-dimensional spheres $S^n$.
For its definition we use a Morse function $S^n\to\R$ with only two
critical points, the minimum and the maximum, and a vector field $v$
(or equivalently a $1$-form $\eta$) which is nowhere vanishing. 
By Proposition~\ref{prop:coproduct-choices}, the coproduct
does not depend on these choices if $n\geq 2$.
For the same reason, in the definition of the loop coproduct we can
use a constant family $v^\tau\equiv v$ instead of the family $v^\tau$
interpolating between $v^0=v$ and $v^1=-v$ from~\S\ref{ss:A2+loop}
(see Remark~\ref{rmk:arbitrary-endpoints}). 
In the case $n=1$ we will see that the loop coproduct actually depends
on the choice of $v$.  

For our computation we first give a third definition of the loop coproduct on reduced loop homology in terms of singular homology. 

\subsection{Topological description of the loop coproduct} 

We define the loop coproduct on singular loop homology relative to $\chi\cdot$point. It is induced by a densely defined operation 
$$
   \lambda: C_*(\Lambda)\to C_{*+1-n}(\Lambda\times\Lambda)
$$
on singular chains constructed as follows. 
The beginning of the construction is like in~\S\ref{ss:A2+loop}. We fix a small vector field $v$ on $M$ with nondegenerate
zeroes such that the only periodic orbits of $v$ with period $\leq 1$
are its zeroes, and we consider a generic family of vector fields $v^\tau$, $\tau\in [0,1]$ which interpolates between $v^0=v$ and $v^1=-v$. We denote $f_t^\tau:M\stackrel{\cong}\longrightarrow M$, $t\in\R$ the flow of $v^\tau$, and $f^\tau=f_1^\tau$. For each $q\in M$ we denote as in~\S\ref{ss:A2+loop} the induced path from $q$ to $f^\tau(q)$ by
$\pi_q^\tau:[0,1]\to M$, $\pi_q^\tau(t):=f_t^\tau(q)$, 
and the inverse path by $(\pi_q^\tau)^{-1}$. 

In the spirit of~\cite{CS}, let $a:K_a\to\Lambda$ be a chain such that the map
$$
   \ev_a:K_a\times[0,1]\to M\times M,\qquad (x,\tau)\mapsto\Bigl(f^\tau\bigl(a(x)(0)\bigr),a(x)(\tau)\Bigr)
$$
is transverse to the diagonal $\Delta\subset M\times M$. Then
$$
   K_{\lambda(a)} := \ev_a^{-1}(\Delta) 
   = \{(x,\tau)\in K_a\times[0,1]\mid a(x)(\tau) = f^\tau\bigl(a(x)(0)\bigr)\}
$$
is a compact manifold with corners and we define 
$$
   \lambda(a): K_{\lambda(a)}\to\Lambda\times\Lambda
$$
by
$$
   \lambda(a)(x,\tau) := \Bigl(a(x)|_{[0,\tau]}\#(\pi_{a(x)(0)}^\tau)^{-1}\,,\,
   \pi_{a(x)(0)}^\tau\#a(x)|_{[\tau,1]}\Bigr). 
$$
See Figure~\ref{fig:lambda-new} where $\alpha=a(x)$.
At $\tau=0$ and $\tau=1$ the condition in $K_{\lambda(a)}$ becomes
$a(x)(0)=q\in\Fix(f^0)$, respectively $a(x)(1)=q\in \Fix(f^1)$, and denoting the constant loop at $q$ by the
same letter we find
$$
   \lambda(a)(x,1) = \bigl(a(x)\#q,q\bigr),\qquad
   \lambda(a)(x,0) = \bigl(q,q\#a(x)\bigr).
$$
It follows that
$$
   \p\lambda(a) \pm \lambda(\p a) =
   \sum_{q\in\Fix(f^1)}\ind_{-v}(q)(a\bullet q)\times q -
   \sum_{q\in\Fix(f^0)}\ind_v(q)q\times(q\bullet a),
$$
where $q$ is viewed as a $0$-chain and the loop products with the
constant loop $q$ are given by  
\begin{align*}
   a\bullet q &: K_{a\bullet q}=\{x\in K_a\mid a(x)(0)=q\}\to\Lambda
   ,\qquad x\mapsto a(x)\#q, \cr
   q\bullet a &: K_{q\bullet a}=\{x\in K_a\mid q=a(x)(0)\}\to\Lambda
   ,\qquad x\mapsto q\#a(x).
\end{align*}
Here the signs $\ind_{\pm v}(q)$ arise from the discussion before
Remark~\ref{rem:ft}, noting that the restriction of $\ev_a$ to $\tau=0$
or $\tau=1$ is the composition of the evaluation $K_a\to M$, $x\mapsto
a(x)(0)$ and the map $M\to M\times M$, $q\mapsto\bigl(f^0(q),q\bigr)$, respectively $q\mapsto\bigl(f^1(q),q\bigr)$.
  
Let us now fix a basepoint $q_0\in M$ and consider $a$ such that the
map 
$$
   \ev_{a,0}:K_a\to M,\qquad x\mapsto a(x)(0)
$$ 
is transverse to $q_0$. We choose all zeroes of $v$ (i.e.~fixed points
of $f^0$ and $f^1$) so close to $q_0$ that $\ev_{a,0}$ is transverse to each of them. 
Then after identifying the domains $K_{q\bullet a}$
with $K_{q_0\bullet a}$ and transferring loops at $q$ to loops at $q_0$
we have
\begin{equation}
   \p\lambda(a) \pm \lambda(\p a) = \chi\Bigl((a\bullet q_0)\times
   q_0 - q_0\times (q_0\bullet a)\Bigr),
\end{equation}
where $\chi=\sum_{q\in\Fix(f^1)}\ind_{-v}(q)=\sum_{q\in\Fix(f^0)}\ind_{v}(q)$ is the Euler characteristic of $M$.
Recalling the notation $C_*(\Lambda,\chi\text{pt}):=C_*(\Lambda)/\chi R q_0$ for the chains relative to $\chi\cdot$point, we see that $\lambda$ induces a chain map $C_*\Lambda\to (C(\Lambda,\chi\text{pt})\otimes C(\Lambda,\chi\text{pt}))_{*+1-n}$. Moreover, this factors through $C_*(\Lambda,\chi\text{pt})$: if $n\ge 2$ this holds for degree reasons, and if $n=1$ this holds tautologically because the Euler characteristic is zero.  The outcome is a coproduct $H_*(\Lambda,\chi\text{pt})\to (H(\Lambda,\chi\text{pt})\otimes H(\Lambda,\chi\text{pt}))_{*+1-n}$ on the 
homology relative to $\chi\cdot$point. Under our standing assumption of orientability on $M$, this is the same as a coproduct on reduced loop homology $\ol H_*\Lambda\to (\ol H\Lambda\otimes \ol H\Lambda)_{*+1-n}$. 

\subsection{Loop coproduct for spheres of odd dimension $n\geq 3$}\label{sec:coproduct-S3}

In this subsection we use $\Z$-coefficients and assume $n\geq 3$ is odd.
Recall from~\cite{Cohen-Jones-Yan} that the degree shifted homology
of the free loop space of $S^n$ is the free graded commutative algebra
$$
   \H_*(\Lambda S^n)=H_{*+n}(\Lambda S^n) \cong \Lambda[A,U],\qquad |A|=-n,\quad |U|=n-1,
$$
where the shifted degree $|a|$ is related to the geometric degree by
$|a|=\deg a-n$. Here $A$ is the class of a point (of geometric degree
$0$) and $U$ is represented by the descending manifold of the Bott family of simple great circles tangent at their basepoint to a given non-vanishing vector field on the sphere (of geometric degree $2n-1$). 
Since $\chi(S^n)=0$, the coproduct $\lambda$ is defined on
$\H_*(\Lambda S^n)$ and has shifted degree $1-2n$ (and geometric
degree $1-n$). The unit $1$ is represented by the fundamental chain of
all constant loops (of geometric degree $n$). 

We begin with some explicit computations of coproducts, to be compared to~\cite{Hingston-Wahl}.

\begin{lemma}\label{lem:Sn}
For $n\ge 3$ odd, the loop coproduct on $\H_*(\Lambda S^n)$ satisfies

(a) $\lambda(A)=\lambda(1)=0$,

(b) $\lambda(AU) = A\otimes A$,

(c) $\lambda(AU^2) = A\otimes AU+AU\otimes A$,

(d) $\lambda(U) = A\otimes 1-1\otimes A$.
\end{lemma}

\begin{proof}
We will actually prove these relations in $H_*(\Lambda S^n)$, in which case (a-c) remain unchanged, but (d) becomes $\lambda(U)=A\otimes 1 + 1 \otimes A$ (the sign change comes from the odd degree shift by $n$).

We recall the observation made at the beginning of this section that in the definition of the loop coproduct we can use a constant family of vector fields $v^\tau\equiv v$. We fix such a choice in the sequel, with $v$ small and nowhere vanishing. We denote $f^\tau_t=f_t$ the flow of $v^\tau=v$, and we denote $f^\tau=f$ the time one flow.  

(a) To prove $\lambda(A)=0$ we represent $A$ by the constant loop at $q_0$. Then $f^\tau(q_0)\neq q_0$ for all values of $\tau\in [0,1]$ and therefore $\lambda(A)$ is supported by the empty set. 
To prove $\lambda(1)=0$ we represent $1$ by the $S^n$-family of constant loops and note that $f^\tau(q)\neq q$ for all $q\in S^n$. Thus $\lambda(1)$ is supported by the empty set. 

(b) We fix a unit tangent vector $v_0$ at $q_0$ and we represent $AU$ by the $(n-1)$-chain $a:K^{n-1} \to\Lambda S^n$ of all circles
with fixed initial point $q_0$ and initial direction $v_0$. ($K^{n-1}$ is
the $(n-1)$-disc of all $2$-planes in $\R^{n+1}$ through $q_0$ containing the
vector $v_0$, whose boundary is mapped to $q_0$.) Then $a(k)(0)=q_0$
for all $k\in K^{n-1}$. Since the evaluation map $(k,\tau)\mapsto a(k)(\tau)$
covers $S^n$ once, there exists a unique $(k,\tau)$ for which
$a(k)(\tau)=f^\tau(q_0)=f(q_0)$. Therefore, $\lambda(a)$ is homologous to the
$0$-cycle $A\otimes A$. 

(c) We represent $AU^2$ by the $(2n-2)$-chain $a:K^{2n-2} \to\Lambda S^n$ of all circles
with fixed initial point $q_0$. ($K^{2n-2}$ is a fibre bundle $K^{n-1}\to
K^{2n-2}\to S^{n-1}$, where $S^{n-1}$ is the $(n-1)$-sphere of all initial directions
at $q_0$ and $K^{n-1}$ is the $(n-1)$-disc from (b) of all circles through
$q_0$ in a given initial direction.) Then $a(k)(0)=q_0$ for all $k\in K^{2n-2}$. 
Recall that $f(q_0)\neq q_0$ is a point close to $q_0$. Let us fix
some initial direction $v_0$ at $q_0$. For every
sufficiently large circle (whose diameter is bigger than the distance
from $q_0$ to $f(q_0)$) with initial point $q_0$ and initial
direction $v_0$ there exist precisely two rotations of the initial
direction such that the rotated circles pass through $f(q_0)$. One of
these rotated circles passes though $f(q_0)$ near $\tau=0$ and the other
one near $\tau=1$. As the circle varies over the $(n-1)$-chain $K^{n-1}$ of all 
circles with initial point $q_0$ and initial direction $v_0$ (and we
let $f(q_0)$ move to $q_0$), these two families of rotated circles
give rise to cycles representing the classes $A\otimes AU$ and
$AU\otimes A$, respectively. 

(d) We represent $U$ by the $(2n-1)$-chain $a:K^{2n-1}\to \Lambda S^n$ of all circles
starting at their basepoint $q\in S^n$ in direction $v(q)$. ($K^{2n-1}$ is a fibre bundle $K^{n-1}\to K^{2n-1}\to S^n$, where $S^n$ corresponds to the
initial points and $K^{n-1}$ is the $(n-1)$-disc from (b).)
For every $q$ there exists a unique circle $a(x_q)$ starting at $q$ in direction $v(q)$ and passing through $f(q)$. Since all the circles constituting the chain $a$ are simple, there is a unique $\tau_q$ such that $a(x_q)(\tau_q)=f(q)=f^{\tau_q}(a(x_q)(0))$.  By splitting each $a(x_q)$ at the parameter value $\tau_q$ using the path $\pi_q(t)=f_t(q)$ we obtain a cycle $s:S^n\to \Lambda\times \Lambda$ that represents $\lambda(a)$. This cycle has degree $n$, it sits over the diagonal $\Delta\subset S^n\times S^n$ as an element of the fibration $(\ev,\ev):\Lambda\times\Lambda\to S^n\times S^n$, and denoting $\pi:\Lambda\times_{S^n} \Lambda\to S^n$ the restriction of this fibration to the diagonal we have $\pi\circ s=\mathrm{Id}_{S^n}$. On the other hand, $H_n(\Lambda\times_{S^n} \Lambda)$ has rank $1$, generated by the class of the diagonal: that the rank is at most $1$ follows by inspection of the spectral sequence of the fibration $\Omega S^n\times \Omega S^n\hookrightarrow \Lambda \times_{S^n}\Lambda\to S^n$,
using the fact that $H_*(\Om S^n)$ is a polynomial ring on one generator in degree $n-1$, and that it is at least one follows from the fact that the diagonal is a section. This implies that the cycle $s$ is homologous to the diagonal $\Delta\subset S^n\times S^n$ in $\Lambda \times_{S^n}\Lambda$, hence also in $\Lambda\times \Lambda$, and we conclude $\lambda(U)=[\Delta]=[pt]\otimes [S^n] + [S^n]\otimes [pt] = A\otimes 1 + 1 \otimes A$ in $H_*(\Lambda)\otimes H_*(\Lambda)$. 

The previous proof uses an algebraic argument related to the diagonal. An alternative, entirely geometric proof can be given in case the sphere $S^n$ admits two orthogonal non-vanishing vector fields. (By Adams' theorem~\cite{Adams}, this is the case if and only $n+1$ is divisible by $4$.) We pick $v$ to be one of these and denote $w$ the other one. We represent $U$ by the $(2n-1)$-chain $a:K^{2n-1} \to\Lambda S^n$ of all circles starting at their basepoint $q\in S^n$ in direction $w(q)$.  
Thus for every $q\in S^n$ there exists a unique circle $a(x_q)$ starting at $q$ in direction $w(q)$ and passing through $f(q)$, and since all the circles constituting the chain $a$ are simple there is a unique $\tau_q$ such that $a(x_q)(\tau_q)=f(q)=f^{\tau_q}(a(x_q)(0))$. Since $v$ is orthogonal to $w$, each circle $a(x_q)$ is small and the resulting cycle $\lambda(a)$ can be deformed in $\Lambda\times \Lambda$ to the diagonal $\Delta\subset \Lambda_0\times \Lambda_0$. In turn, this is represented in $H_*(\Lambda_0)\otimes H_*(\Lambda_0)$ by $[q_0]\otimes [\Lambda_0]+[\Lambda_0]\otimes [q_0]$, i.e. $A\otimes 1 + 1 \otimes A$.
\end{proof}

Note that Lemma~\ref{lem:Sn} is compatible with graded cocommutativity of $\lambda$ on $\H_*\Lambda$, i.e. $\boldtau\lambda=-\lambda$. To compute the full expression of the coproduct we use the following structural result from~\cite{CHO-reduced}. 

\begin{theorem}[{\cite[Theorem~6.4]{CHO-reduced}}] \label{thm:uiasccSn}
Let $M$ be a closed manifold of dimension $n\ge 2$. Then the loop homology $\H_*(\Lambda M)$ endowed with the loop product $\mu$ and the loop coproduct $\lambda$ is a commutative and cocommutative unital infinitesimal anti-symmetric bialgebra. In particular, the following ``unital infinitesimal relation" holds: 
$$
\lambda\mu = (\mu\otimes \one)(\one\otimes\lambda) + (\one\otimes \mu)(\lambda\otimes \one) - (\mu\otimes\mu)(\one\otimes \lambda1\otimes \one),
$$
where we denote $\one$ the identity map and $1$ the unit for the product. \qed
\end{theorem}

For $M=S^n$ with $n\geq 3$ odd we proved in Lemma~\ref{lem:Sn} that $\lambda1=0$, so the unital infinitesimal relation reduces to the so-called ``infinitesimal relation", or ``Sullivan relation" 
$$
\lambda\mu = (\mu\otimes \one)(\one\otimes\lambda) + (\one\otimes \mu)(\lambda\otimes \one). 
$$
Such a relation was conjectured in~\cite{Sullivan-open-closed}. Note that Sullivan's relation is not satisfied by the ``extension by $0$" loop coproduct from~\cite{Hingston-Wahl}. 

\begin{proposition}\label{prop:spheres}
For $n\geq 3$ odd, the loop coproduct on $\H_*(\Lambda S^n)$ satisfies for all $k\geq 0$
\begin{align*}
   \lambda(U^k) &= \sum_{i,j\geq 0,\ i+j=k-1}\bigl(AU^i\otimes
   U^j - U^i\otimes AU^j\bigr), \cr
   \lambda(AU^k) &= \sum_{i,j\geq 0,\ i+j=k-1}AU^i\otimes AU^j. 
\end{align*}
\end{proposition}

\begin{proof} The proof is a straightforward induction on $k$ using knowledge of $\mu$, Sullivan's relation, and the values $\lambda(A)=0$
and $\lambda(U)=A\otimes 1-1\otimes A$ from Lemma~\ref{lem:Sn}. As an example, the values of $\lambda(AU)$ and $\lambda(AU^2)$ from Lemma~\ref{lem:Sn} can be recovered as follows. For the computation, recall that the shifted degrees of $A$, $AU$, $\lambda$ are odd, the shifted degree of $U$ is even, and $A^2=0$.   
\begin{align*}
\lambda(AU) & = \lambda\mu(A\otimes U) \\
& = (\mu\otimes \one)(\one\otimes\lambda)(A\otimes U) + (\one\otimes\mu)(\lambda\otimes \one)(A\otimes U) \\
& = - (\mu\otimes \one)(A\otimes\lambda(U)) + (\one\otimes \mu)(\lambda(A)\otimes U) \\
& = - (\mu\otimes \one)(A\otimes (A\otimes 1 - 1 \otimes A)) \\
& = A\otimes A. 
\end{align*}
\begin{align*}
\lambda(AU^2) & = \lambda\mu(AU\otimes U) \\
& = (\mu\otimes \one)(\one\otimes\lambda)(AU\otimes U) + (\one\otimes\mu)(\lambda\otimes \one)(AU\otimes U) \\
& = -(\mu\otimes \one)(AU\otimes \lambda(U)) + (\one\otimes \mu)(\lambda(AU)\otimes U) \\
& = -(\mu\otimes \one)(AU\otimes (A\otimes 1 - 1 \otimes A)) + (\one\otimes\mu)(A\otimes A \otimes U) \\
& = AU\otimes A + A\otimes AU. 
\end{align*}
\end{proof} 

\begin{remark}
We note in particular that this extended coproduct on reduced homology has contributions from the constant loops, unlike the one from~\cite{Hingston-Wahl}. These contributions from the constant loops play an essential role for the unital infinitesimal algebra structure. 
\end{remark}

The previous computation allows us to recover the Sullivan-Goresky-Hingston coproduct on $\H_*(\Lambda S^n,\Lambda_0)=H_{*+n}(\Lambda S^n,\Lambda_0)$~\cite{Hingston-Wahl,Goresky-Hingston}. Our method ultimately relies on the infinitesimal relation and involves a minimal geometric input in the form of Lemma~\ref{lem:Sn} (a) and (d). In comparison, the computation from~\cite{Hingston-Wahl} of the coproduct on $H_*(\Lambda S^n,\Lambda_0)$ relies on geometric input which is quite involved. In a sense, the ``algebra" of the infinitesimal relation replaces the ``geometry" of spaces of circles from~\cite{Hingston-Wahl}. 

\begin{corollary} For $n\ge 3$ odd, the Sullivan-Goresky-Hingston coproduct on $\H_*(\Lambda S^n,\Lambda_0)$ is given by 
\begin{align*}
   \lambda(U^k) &= \sum_{i,j\geq 1,\ i+j=k-1}\bigl(AU^i\otimes
   U^j - U^i\otimes AU^j\bigr), \cr
   \lambda(AU^k) &= \sum_{i,j\geq 1,\ i+j=k-1}AU^i\otimes AU^j. 
\end{align*}
\end{corollary} 

\begin{proof}
It is enough to discard the terms involving constant loops from the formulas of Proposition~\ref{prop:spheres}.
\end{proof}

\subsection{Loop coproduct for $S^1$}\label{sec:S1}

In this section we study the loop coproduct on the loop space
of $S^1=\R/\Z$. The degree shifted loop homology with $R$-coefficients
is as a ring with respect to the loop product given by
$$
   \H_*(\Lambda S^1) = H_{*+1}(\Lambda S^1) =
   \Lambda[A,U,U^{-1}],\qquad |U|=0,\,|A|=-1, 
$$
where the classes $AU^k$ and $U^k$ are represented by the cycles
$$
   AU^k(t) = kt,\qquad U^k(r,t) = r+kt,\qquad r,t\in S^1,\ k\in\Z. 
$$
To define the loop coproduct $\lambda$ (of shifted degree $-1$), we
need to pick a nowhere vanishing vector field on $S^1$. Up to homotopy
there are two choices of nonvanishing vector fields on $S^1$, 
$$
   v_\pm(x) = \pm\eps,
$$
for some fixed small $\eps>0$. We associate to $v_\pm$ the $\tau$-dependent vector fields
$$
   v_\pm^\tau(x) = \pm(1-2\tau)\eps,\qquad \tau\in [0,1]
$$
which agree with $v_\pm$ at $\tau=0$ and with $-v_\pm$ at $\tau=1$. 
Their time-one maps are
$$
   f_\pm^\tau(x) = x\pm(1-2\tau)\eps. 
$$
In the next Proposition we compute the coproducts $\lambda_\pm$ associated
to this choice of $\tau$-dependent vector fields. 

\begin{proposition}\label{prop:S1-reduced}
The loop coproducts $\lambda_\pm$ on $\H_*(\Lambda S^1)$ defined with the $\tau$-dependent vector fields
$v_\pm^\tau$ are given for $k\in\Z$ by
\begin{align*}
  \lambda_+(AU^k) &= \begin{cases}
     \quad \sum_{i=0}^{k}AU^i\otimes AU^{k-i}, & k\geq 0, \\
     -\sum_{i=k+1}^{-1}AU^i\otimes AU^{k-i}, & k<0,
  \end{cases} \cr 
  \lambda_+(U^k) &= \begin{cases}
     \quad \sum_{i=0}^{k}(AU^i\otimes U^{k-i} - U^i\otimes AU^{k-i}), & k\geq 0, \\
     -\sum_{i=k+1}^{-1}(AU^i\otimes U^{k-i} - U^i\otimes AU^{k-i}), & k<0,
  \end{cases}  
\end{align*}
\begin{align*}
  \lambda_-(AU^k) &= \begin{cases}
     \quad \sum_{i=1}^{k-1}AU^i\otimes AU^{k-i}, & k > 0, \\
     -\sum_{i=k}^{0}AU^i\otimes AU^{k-i}, & k\leq 0,
  \end{cases} \cr 
  \lambda_-(U^k) &= \begin{cases}
     \quad \sum_{i=1}^{k-1}(AU^i\otimes U^{k-i} - U^i\otimes AU^{k-i}), & k> 0, \\
     -\sum_{i=k}^{0}(AU^i\otimes U^{k-i} - U^i\otimes AU^{k-i}), & k\leq 0.
  \end{cases}  
\end{align*}
\end{proposition}
   
\begin{proof}
Let us compute $\lambda_\pm(AU^k)$. By definition, we need to
determine the times $\tau\in (0,1)$ such that
\begin{align*}
   AU^k(\tau) = k\tau = f_\pm^\tau(AU^k(0)) = \pm(1-2\tau)\eps \mod\Z, 
\end{align*}
i.e.~$k\tau = i\pm(1-2\tau)\eps$ with $i\in\Z$. In other words, we are
looking for the $i\in\Z$ such that
$$
   \tau = \frac{i\pm\eps}{k\pm 2\eps} \in (0,1).
$$
For $\lambda_+$ we obtain
\begin{align*}
  \tau = \frac{i+\eps}{k+2\eps} \in (0,1)
  \Longleftrightarrow \begin{cases}
     i=0,\dots,k, & k\geq 0, \\
     i=k+1,\dots,-1, & k<0,
  \end{cases}
\end{align*}
while for $\lambda_-$ we get
\begin{align*}
  \tau = \frac{i-\eps}{k-2\eps} \in (0,1)
  \Longleftrightarrow \begin{cases}
     i=1,\dots,k-1, & k > 0, \\
     i=k,\dots,0, & k\leq 0.
  \end{cases}
\end{align*}
This yields the expressions for $\lambda_\pm(AU^k)$, and
$\lambda_\pm(U^k)$ is computed similarly. 
\end{proof}

Proposition~\ref{prop:S1-reduced} shows that for $M=S^1$ the coproduct on
reduced loop homology does depend on the choice of a nowhere vanishing
vector field. One can verify that both coproducts $\lambda_\pm$ define
together with the loop product a commutative cocommutative infinitesimal anti-symmetric bialgebra in the sense of~\cite{CHO-reduced} with 
$$
   \lambda_\pm(1) = \pm(A\otimes 1-1\otimes A).
$$
The unital infinitesimal relation reads now 
$$
\lambda_\pm\mu = (\mu\otimes \one)(\one\otimes \lambda_\pm) + (\one\otimes\mu)(\lambda_\pm\otimes\one) - (\mu\otimes\mu)(\one\otimes\lambda_\pm(1)\otimes\one).
$$

\begin{remark} Just like in the case of higher-dimensional spheres, the expressions of the coproducts $\lambda_\pm$ on $\H_*(\Lambda S^1)$ can be derived from the unital infinitesimal relation combined with knowledge of the product $\mu$ and of the values 
\begin{align*} 
\lambda_\pm(1)&=\pm(A\otimes 1-1\otimes A),\\
\lambda_\pm(A)&=\pm A\otimes A,\\
\lambda_+(U)&=(A\otimes U - U\otimes A) + (AU\otimes 1 - 1\otimes AU),\\
\lambda_-(U)&=0.
\end{align*}
For example, to compute $\lambda_\pm(U^{-1})$ one applies the unital infinitesimal relation to $U\otimes U^{-1}$, to compute $\lambda_\pm(AU^{-1})$ one applies the unital infinitesimal relation to $A\otimes U^{-1}$ (or to $AU^{-1}\otimes U$) etc. 
\end{remark}

\begin{remark} The example of the circle is very rich in that it also shows that the condition $v^1=-v^0$ for the family of vector fields $v^\tau$ is necessary in order for the coproducts to have a good algebraic behaviour. For example, with a constant family $v^\tau\equiv v_+$ we find an operation $\lambda_{v_+,v_+}$ given by 
\begin{align*}
  \lambda_{v_+,v_+}(AU^k) &= \begin{cases}
     \quad \sum_{i=0}^{k-1}AU^i\otimes AU^{k-i}, & k\geq 0, \\
     -\sum_{i=k}^{-1}AU^i\otimes AU^{k-i}, & k<0,
  \end{cases} \cr 
  \lambda_{v_+,v_+}(U^k) &= \begin{cases}
     \quad \sum_{i=0}^{k-1}(AU^i\otimes U^{k-i} - U^i\otimes AU^{k-i}), & k\geq 0, \\
     -\sum_{i=k}^{-1}(AU^i\otimes U^{k-i} - U^i\otimes AU^{k-i}), & k<0.
  \end{cases}  
\end{align*}
A direct check shows that this operation is neither coassociative, nor cocommutative, though it satisfies the unital infinitesimal relation with $\lambda_{v_+,v_+}(1)=0$, i.e. Sullivan's relation. Similarly, with the constant family $v^\tau\equiv v_-$ we find an operation $\lambda_{v_-,v_-}$ given by 
\begin{align*}
  \lambda_{v_-,v_-}(AU^k) &= \begin{cases}
     \quad \sum_{i=1}^{k}AU^i\otimes AU^{k-i}, & k> 0, \\
     -\sum_{i=k+1}^{0}AU^i\otimes AU^{k-i}, & k\leq 0,
  \end{cases} \cr 
  \lambda_{v_-,v_-}(U^k) &= \begin{cases}
     \quad \sum_{i=1}^{k}(AU^i\otimes U^{k-i} - U^i\otimes AU^{k-i}), & k> 0, \\
     -\sum_{i=k+1}^{0}(AU^i\otimes U^{k-i} - U^i\otimes AU^{k-i}), & k\leq 0.
  \end{cases}  
\end{align*}
Again, this is neither coassociative, nor cocommutative, though it satisfies Sullivan's relation. 
\end{remark}

\appendix

\section{Local systems}\label{sec:local-systems}

We describe in this section the loop product and the
loop coproduct with general twisted coefficients. This allows us in particular to dispose of the usual orientability assumption for the underlying manifold. 
To the best of our knowledge, the Chas-Sullivan product on loop space homology was constructed for the first time on nonorientable manifolds by Laudenbach~\cite{Laudenbach-CS}, 
and the BV algebra structure 
by Abouzaid~\cite{Abouzaid-cotangent}. In this appendix we extend the definitions to more general local systems, we take into account the coproduct, and we discuss the adaptations to reduced homology and cohomology groups $\ol H_*\Lambda$ and $\ol H^*\Lambda$. We also discuss the formulation and properties of the isomorphism between symplectic homology and loop homology with twisted coefficients. 

\subsection{Conventions.} 

We use the following conventions from~\cite[\S9.7]{Abouzaid-cotangent}. Given a finite dimensional real vector space $V$, its \emph{determinant line} is the $1$-dimensional real vector space $\det V=\Lambda^{\max} V$. We view it as being a $\Z$-graded real vector space supported in degree $\dim_\R V$. To any $1$-dimensional graded real vector space $L$ we associate an \emph{orientation line} $|L|$, which is the rank $1$ graded free abelian group generated by the two possible orientations of $L$, modulo the relation that their sum vanishes. The orientation line $|L|$ is by definition supported in the same degree as $L$. When $L=\det V$ we denote its orientation line $|V|$. 

Given a $\Z$-graded line $\ell$ (rank $1$ free abelian group), its \emph{dual line} $\ell^{-1}=\mathrm{Hom}_\Z(\ell,\Z)$ is by definition supported in opposite degree as $\ell$. There is a canonical isomorphism $\ell^{-1}\otimes \ell\cong \Z$ induced by evaluation. 

Given a $\Z$-graded object $F$, we denote $F[k]$ the $\Z$-graded object obtained by shifting the degree down by $k\in\Z$, i.e. $F[k]_n=F_{n+k}$. For example, the shifted orientation line $|V|[\dim\, V]$ is supported in degree $0$. 
A linear map $f:E\to F$ between $\Z$-graded vector spaces or free abelian groups has \emph{degree $d$} if $f(E_n)\subset F_{n+d}$ for all $n$. In an equivalent formulation, the induced map $f[d]:E\to F[d]$ has degree $0$. For example, the dual of a vector space or free abelian group supported in degree $k$ is supported in degree $-k$. This is compatible with the grading convention for duals of $\Z$-graded orientation lines. Given a $\Z$-graded rank $1$ free abelian group $\ell$, we denote $\ul{\ell}$ the same abelian group with degree set to $0$. For example $\ul{|V|}=|V|[\dim\, V]$.

Given two oriented real vector spaces $U$ and $W$, we induce an orientation on their direct sum $U\oplus W$ by defining a positive basis to consist of a positive basis for $U$ followed by a positive basis for $W$. This defines a canonical isomorphism at the level of orientation lines 
$$
|U|\otimes |W|\cong |U\oplus W|. 
$$

Given an exact sequence of vector spaces 
$$
0\to U \to V\to W\to 0
$$ 
we induce an orientation on $V$ out of orientations of $U$ and $W$ by defining a positive basis to consist of a positive basis for $U$ followed by the lift of a positive basis for $W$. This defines a canonical isomorphism 
$$
|U|\otimes |W|\cong |V|. 
$$

The following example will play a key role in the sequel. 

\begin{example}[normal bundle to the diagonal]\label{ex:normal-bundle} 
Let $M$ be a manifold of dimension $n$. Consider the diagonal $\Delta\subset M\times M$ and denote $\nu \Delta$ its normal bundle. Let $p_{1,2}:M\times M\to M$ be the projections on the two factors, so that we have a canonical isomorphism $T(M\times M)\cong p_1^*TM \oplus p_2^*TM$. When restricted to $\Delta$ the projections coincide with the canonical diffeomorphism $p:\Delta \stackrel\simeq\longrightarrow M$. We obtain an exact sequence of bundles 
$$
0\to T\Delta \to p^*TM\oplus p^*TM \to \nu\Delta \to 0. 
$$
This gives rise to a canonical isomorphism $|\Delta|\otimes |\nu\Delta|\cong p^*|M|\otimes p^*|M|$ and, because $p^*|M|\otimes p^*|M|$ is canonically trivial, we obtain a canonical isomorphism 
\begin{equation*} 
|\Delta|\cong |\nu\Delta|. 
\end{equation*}
Explicitly, this isomorphism associates to the equivalence class of a basis $((v_1,v_1),\dots, (v_n,v_n))$, $v_i\in T_qM$ of $T_{(q,q)}\Delta$
the equivalence class of the basis $([(0,v_1)],\dots ,[(0,v_n)])$ of $\nu_{(q,q)}\Delta$. 
\end{example}

\subsection{Homology with local systems} 

By {\em local system} we mean a local system of $\Z$-graded rank $1$ free $\Z$-modules. On each path-connected component of the underlying space we think of such a local system in one of the following three equivalent ways: either as the data of the parallel transport representation of the fundamental groupoid, or as the data of the monodromy representation from the fundamental group $\pi$ to the multiplicative group $\{\pm 1\}$ together with the data of an integer (the degree), or as the data of a $\Z$-graded $\Z[\pi]$-module which is free and of rank $1$ as a $\Z$-module. 
Isomorphism classes of local systems on a path connected space $X$ are thus in bijective correspondence with $H^1(X;\Z/2)\times\Z$, where the first factor corresponds to the monodromy representation and the second factor to the grading. 
Here and in the sequel we identify the multiplicative group $\{\pm 1\}$ with the additive group $\Z/2$.
We refer to~\cite{Albers-Frauenfelder-Oancea} for a comprehensive discussion with emphasis on local systems on free loop spaces. One other point of view on local systems describes these as locally constant sheaves, but we will only marginally touch upon it in~\S\ref{sec:PD-local}. 

Given a local system $\nu$, we can change the coefficients to any commutative ring $R$ by considering $\nu_R=\nu\otimes_\Z R$. The monodromy of such a local system still takes values in $\{\pm 1\}$, and this property characterizes local systems of rank $1$ free $R$-modules which are obtained from local systems of rank $1$ free $\Z$-modules by tensoring with $R$. 

Let $X$ be a path connected space admitting a universal cover $\tilde X$. Denote its fundamental group at some fixed basepoint $\pi=\pi_1(X)$. Interpreting a local system $\nu$ on $X$ as a $\Z[\pi]$-module, one defines \emph{singular homology/cohomology with coefficients in $\nu$} in terms of singular chains on $\tilde X$ as 
$$
H_*(X;\nu)=H_*(C_*(\tilde X;\Z)\otimes_{\Z[\pi]}\nu), 
$$
$$
H^*(X;\nu)=H_*(\mathrm{Hom}_{\Z[\pi]}(C_*(\tilde X;\Z),\nu)).
$$
The homology/cohomology with local coefficients extended to a commutative ring $R$ are the $R$-modules 
$$
H_*(X;\nu_R) = H_*(C_*(\tilde X;\Z)\otimes_{\Z[\pi]}\nu_R),
$$
$$
H^*(X;\nu_R)=H_*\bigl(\mathrm{Hom}_{\Z[\pi]}(C_*(\tilde X;\Z),\nu_R)\bigr).
$$
In our grading convention the cohomology with constant coefficients is supported in \emph{nonpositive} degrees and equals the usual cohomology in the \emph{opposite} degree. The induced differential on the dual group $\mathrm{Hom}_{\Z[\pi]}(C_*(\tilde X;\Z),\nu_R)$ has degree $-1$.

The \emph{tensor product} $\nu_1\otimes\nu_2$ of two local systems is again a local system. Its $\Z[\pi]$-module structure is the diagonal one and its degree is the sum of the degrees of the factors. Note that viewing $\nu_1,\nu_2$ as elements in $H^1(X;\Z/2)\times\Z$, their tensor product is given by their sum $\nu_1+\nu_2$. Operations like cap or cup product naturally land in homology/cohomology with coefficients in the tensor product of the coefficients of the factors. 

Homology/cohomology with local coefficients behave functorially in the following sense. Given a continuous map $f:X\to Y$ and a local system $\nu$ on $Y$ described as a $\Z[\pi_1(Y)]$-module, the \emph{pullback local system} $f^*\nu$ on $X$ is defined by inducing a $\Z[\pi_1(X)]$-module structure via $f_*$. We then have canonical maps 
$$
f_*:H_*(X;f^*\nu)\to H_*(Y;\nu),\qquad f^*:H^*(Y;\nu)\to H^*(X;f^*\nu).
$$

The \emph{algebraic duality isomorphism} with coefficients in a field $\K$ takes the form  
$$
H^{-k}(X;\nu_\K^{-1})\stackrel\cong\longrightarrow H_k(X; \nu_\K)^\vee, \qquad k\in\Z. 
$$
The map is induced by the canonical evaluation of cochains on chains. We check that degrees fit in the case of graded local systems: given a local system $\nu_\K$ of degree $d$, and recalling our notation $\ul{\nu_\K}=\nu_\K[d]$ and $\ul{\nu_\K^{-1}}= \nu_\K^{-1}[-d]$, we have  
$$
H_k(X;\nu_\K) = H_{k-d}(X;\ul{\nu_\K}),\qquad H^{-k}(X;\nu_\K^{-1})=H^{-k+d}(X;\ul{\nu_\K^{-1}}),
$$ 
so $H_k(X;\nu_\K)^\vee$ and $H^{-k}(X;\nu_\K^{-1})$ both live in degree $d-k$.
\subsection{Poincar\'e duality}\label{sec:PD-local} 

Consider a manifold $M$ of dimension $n$. We denote by $|M|$ the local system on $M$ whose fiber at any point $q\in M$ is the orientation line $|T_qM|$, supported by definition in degree $n$. We refer to $|M|$ as the \emph{orientation local system of $M$}. The monodromy along a loop $\gamma$ is $+1$ if the loop preserves the orientation (i.e., the pullback bundle $\gamma^*TM$ is orientable), and $-1$ if the loop reverses it. The local system $|M|$ is trivial if and only if the manifold $M$ is orientable. A choice of orientation is equivalent to the choice of one of the two possible isomorphisms $\ul{|M|}\simeq \Z$. The local system $|M|\otimes \Z/2$ is  trivial, and this reflects the fact that any manifold is $\Z/2$-orientable. 

Suppose now that $M$ is closed. Then it carries a {\em fundamental class} $[M]\in H_n(M;\ul{|M|})=H_0(M;|M|^{-1})$. 

For any local system $\nu$ on $M$, the cap product with a fundamental class defines a {\em Poincar\'e duality isomorphism}
\begin{equation*}
   H^*(M;\nu) \stackrel{\cong}\longrightarrow H_*(M;\nu\otimes |M|^{-1}),\qquad \alpha\mapsto[M]\cap\alpha.
\end{equation*}

\begin{remark}Here is a description of the fundamental class using the interpretation of local systems as locally constant sheaves (see for example~\cite{Hatcher}, Lemma~3.27 and Example~3H.3). 
Let $M'\stackrel{p}\to M$ be the orientation double cover. Given the constant local system $\Z$ on $M'$, the pushforward $p_*\Z$ to $M$ has rank $2$ and can be decomposed as $\ul{|M|}\oplus\Z$ (the map $\Z\oplus \Z\to \Z\oplus \Z$, $(x,y)\mapsto (y,x)$ fixes the diagonal and acts by $-\mathrm{Id}$ on the anti-diagonal). 
The composition $H_*(M';\Z)\stackrel{p_*}\to H_*(M;p_*\Z)\stackrel\simeq\to H_*(M;\ul{|M|})\oplus H_*(M;\Z)$ is an isomorphism because $p_*$ is an isomorphism. Since $H_n(M;\Z)=0$ if $M$ is nonorientable, we obtain that $H_n(M;\ul{|M|})$ has rank $1$. A generator is the image of a generator in $H_n(M';\Z)$ via the above composition.
\end{remark}

\subsection{Thom isomorphism and Gysin sequence} 

Let $E\stackrel{p}\longrightarrow X$ be a real vector bundle of rank $r$, and denote $\dot E$ the complement of the zero section. Let $|E|$ be the local system on $X$ whose fiber at a point $x\in X$ is the orientation line $|E_x|$ of the fiber of $E$ at $x$. The local system $|E|$ is called \emph{the orientation local system of $E$} and is supported in degree $r$. The \emph{Thom class} is a generator 
$$
\tau\in H^{-r}(E,\dot E;\ul{p^*|E|})=H^0(E,\dot E;p^*|E|).
$$ 
The  
\emph{Thom isomorphism} takes the form 
$$
H_k(E,\dot E)\stackrel \simeq\longrightarrow H_{k-r}(X;\ul{|E|})=H_k(X;|E|), \qquad k\in\Z
$$
(cap product with $\tau$), respectively 
$$
H^k(X;|E|^{-1})=H^{k+r}(X;\ul{|E|^{-1}})\stackrel \simeq \longrightarrow H^k(E,\dot E), \qquad k\in\Z
$$
(cup product with $\tau$). 
More generally, for any local system $\nu$ on $X$ we have isomorphisms 
$$
H_k(E,\dot E;p^*\nu) \stackrel \simeq\longrightarrow H_{k-r}(X;\nu\otimes \ul{|E|})=H_k(X;\nu\otimes |E|),
$$
$$
H^k(X;\nu\otimes |E|^{-1}) = H^{k+r}(X;\nu\otimes \ul{|E|^{-1}})\stackrel \simeq \longrightarrow H^k(E,\dot E;p^*\nu).
$$
Pulling back the Thom class under the inclusion $i:X\to E$ of the zero section yields the {\em Euler class}
$$
   e = i^*\tau\in H^{-r}(X;\ul{|E|})=H^0(X;|E|). 
$$
Denote by $S\subset\dot E$ the sphere bundle with projection $\pi=p|_S:S\to X$. Then the long exact sequence of the pair $(E,\dot E)$ fits into the commuting diagram
$$
\xymatrix
@C=25pt
{
   \cdots H^k(E,\dot E) \ar[r] & H^k(E) \ar[d]_{i^*}^\cong \ar[r] & H^k(\dot E) \ar[r] & H^{k-1}(E;\dot E)\cdots \\ 
   \cdots H^k(X;|E|^{-1}) \ar[u]^{\cup\tau}_\cong \ar[r]^{\ \ \ \ \ \cup e} & H^k(X) \ar[r]^{\pi^*} & H^k(S) \ar@{=}[u] \ar[r]^{\pi_*\ \ \ \ \ } & H^{k-1}(X;|E|^{-1})\cdots \ar[u]^{\cup\tau}_\cong
}
$$
where the lower sequence is the {\em Gysin sequence}. More generally, for each local system $\nu$ on $X$ we get a Gysin sequence
\begin{align*}
 \cdots \, H^k(X;\nu\otimes|E|^{-1}) & \stackrel{\cup e}\longrightarrow   H^k(X;\nu) \\
 & \qquad \stackrel{\pi^*}\longrightarrow H^k(S;\pi^*\nu) \stackrel{\pi_*}\longrightarrow  H^{k-1}(X;\nu\otimes|E|^{-1})\, \cdots 
\end{align*}

\subsection{Spaces of loops with self-intersection} \label{sec:spaces-of-loops}

Let $M$ be a manifold of dimension $n$, $\Lambda=\Lambda M$ its space of free loops of Sobolev class $W^{1,2}$, and $\ev_s:\Lambda\to M$ the evaluation of loops at time $s$. We define 
$$
\cF=\{(\gamma,\delta)\in \Lambda\times \Lambda \mid \gamma(0)=\delta(0)\}\subset \Lambda\times\Lambda
$$
(pairs of loops with the same basepoint), and
$$
\cF_s = \{\gamma\in\Lambda \mid \gamma(s)=\gamma(0)\}\subset \Lambda, \qquad s\in (0,1)
$$
(loops with a self-intersection at time $s$). Denoting $f:\Lambda\times \Lambda\to M\times M$, $f=\ev_0\times \ev_0$ and $f_s:\Lambda\to M\times M$, $f_s=(\ev_0,\ev_s)$, we can equivalently write 
$$
\cF=f^{-1}(\Delta),\qquad \cF_s=f_s^{-1}(\Delta). 
$$
The maps $f$ and $f_s$ are smooth and transverse to the diagonal $\Delta$, 
so that $\cF$ and $\cF_s$ are Hilbert submanifolds of codimension $n$. Denoting $\nu\cF$ and $\nu\cF_s$ their normal bundles we obtain canonical isomorphisms
$$
\nu\cF\cong f^*\nu\Delta,\qquad \nu\cF_s\cong f_s^*\nu\Delta.
$$
In view of Example~\ref{ex:normal-bundle} we infer canonical isomorphisms
\begin{equation}\label{eq:nuF-nuFt}
|\nu\cF|\cong f^*|\Delta|\cong \ev_0^*|M|,\qquad |\nu\cF_s|\cong f_s^*|\Delta|\cong \ev_0^*|M|,
\end{equation}
where, in the first formula, $\ev_0:\cF\to M$ is the evaluation of pairs of loops at their common origin. 

Denote $i:\cF\hookrightarrow \Lambda\times \Lambda$ and $i_s:\cF_s\hookrightarrow \Lambda$ the inclusions. Recall the restriction maps~\eqref{eq:restr}. Define the cutting map at time $s$  
$$
c_s: \cF_s\to \cF,\qquad c_s(\gamma)=(\gamma|_{[0,s]},\gamma|_{[s,1]})
$$
and the concatenation map at time $s$
$$
g_s:\cF\to \cF_s, \qquad g_s(\gamma_1,\gamma_2)(t)=\left\{\begin{array}{ll}\gamma_1(\frac{t}{s}), & t\in [0,s],\\
\gamma_2(\frac{t-s}{1-s}), & t\in [s,1].\end{array}\right. 
$$
The maps $c_s$ and $g_s$ are smooth diffeomorphisms 
inverse to each other. The situation is summarized in the diagram
$$
\xymatrix
@C=50pt
{
\Lambda \times \Lambda & \ar[l]^-{i} \cF \ar@<.5ex>[r]^-{g_s}_-\sim & \ar@<1.1ex>[l]^-{c_s}\cF_s \ar[r]_-{i_s}& \Lambda.
}
$$
 
\begin{lemma} \label{lem:loc-sys-compatibility-with-products} 
Let $\nu$ be a local system (of rank $1$ free abelian groups) on $\Lambda$ supported in degree $0$. Denote $p_{1,2}:\Lambda\times \Lambda\to \Lambda$ the projections on the two factors. The following two conditions are equivalent:
\begin{equation}\label{eq:loc-sys-cut} 
c_s^*(p_1^*\nu\otimes p_2^*\nu)|_\cF\simeq \nu|_{\cF_s},
\end{equation}
and
\begin{equation}\label{eq:loc-sys-concat} 
(p_1^*\nu\otimes p_2^*\nu)|_\cF\simeq g_s^*(\nu|_{\cF_s}).
\end{equation}
\end{lemma}

\begin{proof} The first condition is $c_s^*i^*(p_1^*\nu\otimes p_2^*\nu)\simeq i_s^*\nu$. 
Since $g_s$ is a homeomorphism, this is equivalent to 
$$
g_s^*c_s^*i^*(p_1^*\nu\otimes p_2^*\nu)\simeq g_s^*i_s^*\nu. 
$$
In view of $c_sg_s=\mathrm{Id}_{\cF_s}$, this is the same as the second condition. 
\end{proof}

\begin{definition} \label{defi:loc_sys_compat_products} A degree $0$ local system $\nu$ on $\Lambda$ is \emph{compatible with products} if it satisfies the equivalent conditions of Lemma~\ref{lem:loc-sys-compatibility-with-products}. 
\end{definition}

A local system $\nu$ which is compatible with products must necessarily have degree $0$ (and rank $1$). Also, $\nu|_M$ must be trivial: restricting both sides of~\eqref{eq:loc-sys-cut} or~\eqref{eq:loc-sys-concat} to the constant loops yields $\nu|_M\otimes \nu|_M\simeq\nu|_M$.

\begin{remark}
Local systems which are compatible with products play a key role in the sequel definition of the loop product and loop coproduct with local coefficients. Condition~\eqref{eq:loc-sys-cut} is the one that ensures the coproduct is defined with coefficients twisted by $\nu$, whereas condition~\eqref{eq:loc-sys-concat} is the one that ensures the product is defined with coefficients twisted by $\nu$. That the two conditions are equivalent can be seen as yet another instance of Poincar\'e duality for free loops. 

We refer to Remark~\ref{rmk:associativity-local-systems} for an additional condition on the isomorphisms~\eqref{eq:loc-sys-concat} which is needed for the associativity of the product and coassociativity of the coproduct.  
\end{remark}

\begin{example}[Transgressive local systems] \label{example:transgressive}
Let $\Lambda=\sqcup_\alpha \Lambda_\alpha M$ be the decomposition of the free loop space into connected components, indexed by conjugacy classes $\alpha$ in the fundamental group. We view loops $\gamma:S^1\to \Lambda$ as maps $\gamma\times S^1:S^1\times S^1\to M$, $(u,t)\mapsto \gamma(u)(t)$. This induces a map $\pi_1(\Lambda_\alpha M)\to H_1(\Lambda_\alpha M;\Z)\to H_2(M;\Z)$, $[\gamma]\mapsto [\gamma\times S^1]$. Dually, and specializing to $\Z/2$-coefficients, any cohomology class $c\in H^2(M;\Z/2)$ determines a cohomology class $\tau_c\in H^1(\Lambda;\Z/2)=\prod_\alpha\mathrm{Hom}(\pi_1(\Lambda_\alpha M);\Z/2)$ via 
$$
\langle \tau_c,[\gamma]\rangle = \langle c,[\gamma\times S^1]\rangle. 
$$ 
We denote the corresponding local system on $\Lambda$ also by $\tau_c$. Degree $0$ local systems obtained in this way are called \emph{transgressive}~\cite{Abouzaid-cotangent}. 

Transgressive local systems are compatible with products. Indeed, the identity~\eqref{eq:loc-sys-concat} is a direct consequence of the equality $[g_s(\gamma_1,\gamma_2) \times S^1]=[\gamma_1\times S^1]+[\gamma_2\times S^1]$, which holds in $H_2(M;\Z)$ for all 
$(\gamma_1,\gamma_2):S^1\to\cF$.

The transgressive local system 
\begin{equation} \label{eq:sigma=tauw2}
\sigma=\tau_{w_2}
\end{equation} 
defined by the second Stiefel-Whitney class $w_2\in H^2(M;\Z/2)$ will play a special role in the sequel. 
 \end{example}

\begin{example} \label{example:w}Following Abouzaid~\cite{Abouzaid-cotangent}, define for each loop $\gamma\in\Lambda$ the \emph{shift} 
$$
w(\gamma)=\left\{\begin{array}{ll}0, & \mbox{if } \gamma \mbox{ preserves the orientation}, \\
-1, & \mbox{if } \gamma \mbox{ reverses the orientation}.
\end{array}\right. 
$$
Define the local system 
\begin{equation} \label{eq:tildeo}
\tilde o=\ev_0^*\ul{|M|}^{-w}
\end{equation}
to be trivial on the components where $\gamma$ preserves the orientation, and equal to $\ev_0^*\ul{|M|}$ on components where $\gamma$ reverses the orientation. 

The local system $\tilde o$ is compatible with products: the equality $w(\gamma_1)+w(\gamma_2)=w(g_s(\gamma_1,\gamma_2))$  holds in $\Z/2$ for all $(\gamma_1,\gamma_2)\in\cF$. Note that the local system $\tilde o$ is \emph{not} transgressive and, in case $M$ is nonorientable, it is nontrivial on all connected components $\Lambda_\alpha M$ whose elements reverse orientation. 
\end{example}

\begin{question}
Characterize in cohomological terms the local systems on $\Lambda$ which are compatible with products. For example, it follows from~\cite[Lemma~1]{Albers-Frauenfelder-Oancea} that, on a simply connected manifold, a local system $\nu$ is compatible with products if and only if $\nu|_M$ is trivial. A mild generalization is given by~\cite[Proposition~10]{Albers-Frauenfelder-Oancea}. 
\end{question}

\subsection{Loop product with local coefficients}

Following~\cite{Goresky-Hingston}, we view the loop product as being defined by going from left to right in the diagram 
$$
\Lambda\times \Lambda \hookleftarrow \cF \stackrel{g}\longrightarrow \Lambda,
$$
where $g=i_sg_s$ for some fixed $s\in(0,1)$. More precisely, the \emph{loop product with integer coefficients} is defined as the composition
\begin{align*}
H_i (\Lambda;\Z) \otimes H_j (\Lambda;\Z) & \stackrel{\epsilon\times}{\longrightarrow} H_{i+j}(\Lambda\times \Lambda;\Z) \\
& \longrightarrow  H_{i+j}(\nu \cF,\dot\nu\cF;\Z) \\
& \stackrel{\simeq}{\longrightarrow} H_{i+j}(\cF;\ev_0^*{|M|}) \\
& \stackrel{g_*}{\longrightarrow} H_{i+j}(\Lambda;\ev_0^*{|M|}).
\end{align*}
The first map is the homology cross-product corrected by a sign $\epsilon=(-1)^{n(i+n)}$ (\cite[Appendix~B]{Hingston-Wahl}), the second map is the composition of the map induced by inclusion $\Lambda\times\Lambda \hookrightarrow (\Lambda\times \Lambda,\Lambda\times \Lambda \setminus \cF)$ with excision and the tubular neighbourhood isomorphism, and the third map is the Thom isomorphism. In case $M$ is not orientable the loop product does not land in homology with integer coefficients and thus fails to define an algebra structure on $H_*(\Lambda;\Z)$. This can be corrected by using at the source homology with local coefficients.

\begin{definition}
Define on $\Lambda$ the local system 
$$
   \mu := \ev_0^*|M|^{-1}.
$$ 
The \emph{archetypal loop product} is the bilinear map 
$$
\bullet: H_i(\Lambda;\mu) \otimes H_j(\Lambda;\mu) \to H_{i+j}(\Lambda;\mu)
$$ 
defined as the composition
\begin{align*}
H_i (\Lambda;\mu) \otimes H_j (\Lambda;\mu) 
& \stackrel{\epsilon\times}{\longrightarrow} H_{i+j}(\Lambda\times \Lambda;p_1^*\mu\otimes p_2^*\mu) \\
& \longrightarrow H_{i+j}(\nu \cF,\dot\nu\cF;p_1^*\mu\otimes p_2^*\mu|_{\nu\cF}) \\
& \stackrel{\simeq}{\longrightarrow} H_{i+j}(\cF;(p_1^*\mu\otimes p_2^*\mu)|_\cF\otimes \ev_0^*|M|) \\ 
& \stackrel{g_*}{\longrightarrow} H_{i+j}(\Lambda;\mu).
\end{align*}
\end{definition} 
The description of the maps is the same as above, with $\epsilon=(-1)^{ni}$ because of the shift $H_i(\Lambda;\mu)=H_{i+n}(\Lambda;\ul{\mu})$. However, one still needs to check that the local systems of coefficients are indeed as written. For the first, second and third map the behavior of the coefficients follows general patterns. For the last map we use that 
$$
(p_1^*\mu\otimes p_2^*\mu)|_\cF\otimes \ev_0^*|M|\simeq g^*\mu,
$$
which is true for our specific $\mu=\ev_0^*|M|^{-1}$. 

The archetypal loop product is associative, graded commutative,
and it has a unit represented by the fundamental class 
$$
   [M]\in H_0(M;|M|^{-1})=H_n(M;\ul{|M|})
$$
from~\S\ref{sec:PD-local}. With our grading conventions, the archetypal loop product has degree $0$ and the local system $\mu$ is supported in degree $-n$.
In the case where $M$ is oriented we recover the usual loop product. 

More generally, the loop product can be defined with further twisted coefficients. 

\begin{definition} Let $\nu$ be a degree $0$ local system (of rank $1$ free $\Z$-modules) on $\Lambda$ which is compatible with products. The \emph{loop product with coefficients twisted by $\nu$} is the bilinear map
$$
\bullet: H_i(\Lambda;\nu\otimes\mu) \otimes H_j(\Lambda;\nu\otimes\mu) \to H_{i+j}(\Lambda;\nu\otimes\mu)
$$ 
(with $\mu=\ev_0^*|M|^{-1}$ as above) defined as the composition
\begin{align*}
H_i (\Lambda;\nu\otimes \mu) & \otimes H_j (\Lambda;\nu\otimes \mu) \\
& \stackrel{\epsilon\times}{\longrightarrow} H_{i+j}(\Lambda\times \Lambda;(p_1^*\nu\otimes p_2^*\nu) \otimes (p_1^*\mu\otimes p_2^*\mu)) \\
& \longrightarrow  
H_{i+j}(\nu \cF,\dot\nu\cF;(p_1^*\nu\otimes p_2^*\nu) \otimes (p_1^*\mu\otimes p_2^*\mu)|_{\nu\cF})  \\
& \stackrel{\simeq}{\longrightarrow} H_{i+j}(\cF;(p_1^*\nu\otimes p_2^*\nu) \otimes (p_1^*\mu\otimes p_2^*\mu)|_\cF\otimes \ev_0^*|M|) \\
& \stackrel{g_*}{\longrightarrow} H_{i+j}(\Lambda;\nu\otimes \mu).
\end{align*}
\end{definition}
As before, we have $\epsilon=ni$. For the last map we use the isomorphism $(p_1^*\mu\otimes p_2^*\mu)|_\cF\otimes \ev_0^*|M|\simeq g^*\mu$, and the isomorphism $(p_1^*\nu\otimes p_2^*\nu)|_\cF\simeq g^*\nu$ which expresses the compatibility with products for $\nu$. 

The loop product with twisted coefficients is graded commutative
and unital. Recalling that the compatibility with products for $\nu$ forces its restriction to $M$ to be trivial, the unit is again represented by the fundamental class 
$$
[M]\in H_0(M;\nu|_M\otimes|M|^{-1})=H_0(M;|M|^{-1})=H_n(M;\ul{|M|}).
$$
\begin{remark}\label{rmk:associativity-local-systems}
Associativity of the loop product with twisted coefficients depends 
on the following associativity condition on the isomorphisms~\eqref{eq:loc-sys-cut} and~\eqref{eq:loc-sys-concat}. Given $s,s'\in (0,1)$ denote $s''=(s'-ss')/(1-ss')$, so that $g_{s'}\circ (g_s\times\mathrm{id})=g_{ss'}\circ(\mathrm{id}\times g_{s''})$. Denoting $\Phi_s:(p_1^*\nu\otimes p_2^*\nu)|_\cF\stackrel\simeq\to g_s^*\nu|_{\cF_s}$ the isomorphism from~\eqref{eq:loc-sys-concat}, we require the associativity condition 
$$
\Phi_{ss'}\circ (\mathrm{Id}\otimes \Phi_{s''})=\Phi_{s'}\circ (\Phi_s\otimes\mathrm{Id}).
$$
This holds for the transgressive local systems from Example~\ref{example:transgressive} and for the local system in Example~\ref{example:w}. 

Also, because~\eqref{eq:loc-sys-cut} and~\eqref{eq:loc-sys-concat} are equivalent, this condition on~\eqref{eq:loc-sys-concat} will guarantee coassociativity of the coproduct, see below.  
\end{remark}

\subsection{Loop coproduct with coefficients} \label{sec:coproduct-coefficients}

Again following~\cite{Goresky-Hingston}, we view the primary coproduct on loop homology as being defined by going from left to right in the diagram 
$$
\Lambda \hookleftarrow \cF_s \stackrel{c_s}\longrightarrow \Lambda\times \Lambda
$$
for some fixed $s\in(0,1)$, where $c_s$ stands for $ic_s$ in the notation of~\S\ref{sec:spaces-of-loops}. We restrict in this section to coefficients in a field $\K$ and all local systems are accordingly understood in this category. The reason for this restriction is explained below. The \emph{primary coproduct with constant coefficients} is defined as the composition
$$
\xymatrix{
H_k (\Lambda;\K) \ar[r] & H_k(\nu\cF_s,\dot\nu\cF_s;\K) &  \ar[l]_-\simeq H_k(\cF_s;\ev_0^*|M|) 
}
$$
$$
\xymatrix{
\ar[r]^-{c_{s*}} & H_k(\Lambda\times\Lambda;p_1^*\ev_0^*|M|) \ar[r]^-{AW} & \bigoplus_{i+j=k}H_i(\Lambda;\ev_0^*|M|)\otimes H_j(\Lambda;\K).
}
$$
The first map is the composition of the map induced by inclusion $\Lambda\to (\Lambda,\Lambda\setminus \cF_s)$ with the excision isomorphism towards the homology rel boundary of a tubular neighbourhood of $\cF_s$. The second map is the Thom isomorphism. For the third map we use that $c_s^*p_1^*\ev_0^*=\ev_0^*$. The fourth map is the Alexander-Whitney diagonal map followed by the K\"unneth isomorphism.\footnote{The Alexander-Whitney diagonal map~\cite[VI.12.26]{Dold} takes values in $H_*(C_*(\Lambda;\ev_0^*|M|)\otimes C_*(\Lambda))$ with arbitrary coefficients. In order to further land in $H_*(\Lambda;\ev_0^*|M|)\otimes H_*(\Lambda)$ we need to restrict to field coefficients so the K\"unneth isomorphism holds. Alternatively, one needs to modify the target of the coproduct to be $H_*(\Lambda\times\Lambda;p_1^*\ev_0^*|M|)$, see~\cite{Hingston-Wahl}.}
Just like for the loop product, we see that if $M$ is nonorientable the primary coproduct fails to define a coalgebra structure on $H_*(\Lambda;\K)$. This is corrected by using homology with local coefficients as follows. 

\begin{definition}
Define on $\Lambda$ the local system
$$
   o := \ev_0^*|M| = \mu^{-1}.
$$
The \emph{archetypal primary coproduct} is the bilinear map 
$$
\vee_s: H_k(\Lambda;o) \to \bigoplus_{i+j=k}H_i(\Lambda;o) \otimes H_j(\Lambda;o)
$$ 
(for some fixed $s\in[0,1]$) defined as the composition
\begin{align*}
H_k (\Lambda;o) 
& \longrightarrow H_k(\nu\cF_s,\dot\nu\cF_s;o|_{\nu\cF_s}) \\
& \stackrel{\simeq}{\longleftarrow} H_k(\cF_s;o\otimes \ev_0^*|M|) \\
& \stackrel{c_{s*}}{\longrightarrow} H_k(\Lambda\times\Lambda;p_1^*o\otimes p_2^*o) \\
& \stackrel{AW}{\longrightarrow} \bigoplus_{i+j=k}H_i(\Lambda;o)\otimes H_j(\Lambda;o).
\end{align*}
\end{definition}

With our grading conventions this coproduct has degree $0$. Taking into account that $o=\ev_0^*|M|$ is supported in degree $n$, this results in the coproduct having the usual degree $-n$ in ungraded notation. In the orientable case it recovers the usual primary coproduct. 

Just like the product, the primary coproduct can be defined with further twisted coefficients. 

\begin{definition} Let $\nu$ be a degree $0$ local system (of rank one $\K$-vector spaces) on $\Lambda$ which is compatible with products. The \emph{primary coproduct with twisted coefficients} is the bilinear map 
$$
\vee_s: H_k(\Lambda;\nu\otimes o) \to \bigoplus_{i+j=k}H_i(\Lambda;\nu\otimes o) \otimes H_j(\Lambda;\nu\otimes o)
$$ 
(with $o=\ev_0^*|M|$ as above and some fixed $s\in[0,1]$) defined as the composition 
\begin{align*}
H_k (\Lambda;\nu\otimes o) 
& \longrightarrow H_k(\nu\cF_s,\dot\nu\cF_s;\nu\otimes o|_{\nu\cF_s}) \\ 
& \stackrel{\simeq}{\longleftarrow} H_k(\cF_s;\nu\otimes o^{\otimes 2}) \\
& \stackrel{c_{s*}}{\longrightarrow} H_k(\Lambda\times\Lambda;p_1^*(\nu\otimes o)\otimes p_2^*(\nu\otimes o)) \\
& \stackrel{AW}{\longrightarrow} \bigoplus_{i+j=k}H_i(\Lambda;\nu\otimes o)\otimes H_j(\Lambda;\nu\otimes o).
\end{align*}
\end{definition}
In the definition we use $c_s^*(p_1^*o\otimes p_2^*o)\simeq o\otimes o$, and the condition $c_s^*(p_1^*\nu\otimes p_2^*\nu)\simeq \nu|_{\cF_s}$ which is part of the condition of being compatible with products for $\nu$. 

The arguments of Goresky-Hingston~\cite[\S8]{Goresky-Hingston} which show that, in the orientable case, there is a secondary coproduct of degree $-n+1$ defined on relative homology $H_*(\Lambda,\Lambda_0;\K)$, apply \emph{verbatim} in the current setup involving local coefficients. As an outcome, we obtain the following. 

\begin{definition-proposition}
For any degree $0$ local system $\nu$ of rank one $\K$-vector spaces on $\Lambda$ which is compatible with products, there is a well-defined \emph{(secondary) loop coproduct with twisted coefficients} (abbreviate $o=\ev_0^*|M|$) 
$$
\vee: H_k(\Lambda,\Lambda_0;\nu\otimes o) \to \bigoplus_{i+j=k+1}H_i(\Lambda,\Lambda_0;\nu\otimes o) \otimes H_j(\Lambda,\Lambda_0;\nu\otimes o).
$$ 
\qed
\end{definition-proposition}

As explained in~\cite{Hingston-Wahl}, in order for this secondary coproduct to be coassociative in the case of a constant local system $\nu$ we need to correct the $ij$-component of the secondary product induced by the previously defined primary product by a sign $\epsilon=(-1)^{(n-1)(j-n)}$ (see~\cite[Definition~1.7]{Hingston-Wahl} and note the shift in grading $H_j(\Lambda,\Lambda_0;\nu\otimes o)=H_{j-n}(\Lambda,\Lambda_0;\nu\otimes \ul{o})$). With this correction the coproduct with constant local system $\nu$ is also graded cocommutative if gradings are shifted so that it has degree $0$. However, it has no counit (this would contradict the infinite dimensionality of the homology of $\Lambda$).

In the case of coefficients twisted by a local system $\nu$ which is nonconstant, coassociativity further requires that the isomorphisms $\Phi_s$ expressing compatibility with products for $\nu$ satisfy the condition from Remark~\ref{rmk:associativity-local-systems}. The coproduct is then also cocommutative. 

One obtains dually a cohomology product~\cite{Goresky-Hingston,
  Hingston-Wahl}. Note that, in contrast to the loop coproduct, the
dual loop cohomology product is defined with arbitrary coefficients
because its definition does not require the K\"unneth isomorphism. 

\begin{definition-proposition}
For any degree $0$ local system $\nu$ of rank $1$ free abelian groups on $\Lambda$ which is compatible with products, there is a well-defined \emph{cohomology product with twisted coefficients} (abbreviate $\mu=\ev_0^*|M|^{-1}$) 
$$
\oast: H^i(\Lambda,\Lambda_0;\nu\otimes \mu) \otimes H^j(\Lambda,\Lambda_0;\nu\otimes \mu) \to H^{i+j-1}(\Lambda,\Lambda_0;\nu\otimes \mu).
$$ 
\qed\end{definition-proposition}

The cohomology product with twisted coefficients is associative. It is also graded commutative when viewing it as a degree $0$ product on $H^{*-1}(\Lambda,\Lambda_0;\nu\otimes\mu)$. 

\subsection{Reduced- vs. loop homology relative to $\chi\cdot$point} \label{sec:reduced-relpoint}

We explain in this section the interplay between reduced loop homology and loop homology relative to $\chi\cdot$point in the presence of local coefficients. 

Recall the previous notation $o=\ev_0^*|M|=\mu^{-1}$, and let $\nu$ be a local system compatible with products. The arguments of~\S\ref{ss:A2+loop} carry over \emph{verbatim} to give a description of the loop product and coproduct in Morse homology with local coefficients in $\nu\otimes \mu$, respectively $\nu\otimes o$. For a definition of Morse homology with local coefficients we refer to~\cite[\S11.3]{Abouzaid-cotangent} or~\cite[\S7.2]{LS}. 

The \emph{reduced loop (co)homology groups} $\ol{MH}_*$ and $\ol{MH}^*$ are defined with local coefficients as follows. Recall that $\nu|_M$ is trivial. We consider the map $\eps$ given as the composition
$$
\xymatrix{
H^{*}(\Lambda;\nu\otimes \ul{\mu}) \ar[r]^-\eps \ar[d] & H_*(\Lambda;\nu\otimes \ul{\mu}) \\
H^0(M;\ul{\mu})\ar[r]_{\eps_0}& H_0(M;\ul{\mu}) \ar[u]
}
$$
where the vertical maps are restriction to, respectively inclusion of constant loops, and $\eps_0$ is induced by multiplication with the Euler characteristic. We then define 
$$
\ol{MH}^*(\Lambda;\nu\otimes \ul{\mu})=\ker \eps,\qquad \ol{MH}_*(\Lambda;\nu\otimes \ul{\mu})=\coker \, \eps. 
$$
Thus $\ol{MH}_*(\Lambda;\nu\otimes \mu)=\ol{MH}_{*-n}(\Lambda;\nu\otimes \ul{\mu})$ and $\ol{MH}_*(\Lambda;\nu\otimes o)=\ol{MH}_{*+n}(\Lambda;\nu\otimes \ul{\mu})$. 
Similarly to~\S\ref{ss:A2+loop}, the loop product descends to $\ol{MH}_*(\Lambda;\nu\otimes \ul \mu)$ because $\mathrm{im}\,\eps\subset H_*(\Lambda;\nu\otimes \ul{\mu})$ is an ideal.

The \emph{loop (co)homology groups relative to $\chi\cdot$point} are defined with local coefficients as follows. Recall again that $\nu|_M$ is trivial. Consider the embedding $\chi R q_0 \to MC_*(\Lambda;\nu\otimes \ul o)$ given by the inclusion $\chi R\hookrightarrow R$ and define the Morse chains relative to $\chi\cdot$point as $MC_*(\Lambda,\chi\text{pt};\nu\otimes \ul o)=MC_*(\Lambda;\nu\otimes \ul o)/\chi R q_0$. Similarly consider the projection $\pi:MC^*(\Lambda;\nu\otimes \ul \mu)\to R q_0$ and define the Morse cochains relative to $\chi\cdot$point to be $MC^*(\Lambda,\chi\text{pt};\nu\otimes \ul \mu)=\pi^{-1} \ker (R\stackrel{\cdot \chi}\longrightarrow R)$. The loop (co)homology groups relative to $\chi\cdot$point are 
$$
MH_*(\Lambda,\chi\text{pt};\nu\otimes \ul o) = H_*(MC_*(\Lambda,\chi\text{pt};\nu\otimes \ul o)),
$$
$$ 
MH^*(\Lambda,\chi\text{pt};\nu\otimes \ul \mu) = H^*(MC^*(\Lambda,\chi\text{pt};\nu\otimes \ul \mu)).
$$
The arguments of~\S\ref{ss:A2+loop} carry over \emph{verbatim} in order to show that the loop coproduct extends to $MH_*(\Lambda,\chi\text{pt};\nu\otimes \ul o)$ (for algebraic reasons we need to use field coefficients as in~\S\ref{sec:coproduct-coefficients}). Interpreted dually as a product on cohomology, this is defined with arbitrary coefficients on $MH^*(\Lambda,\chi\text{pt};\nu\otimes \ul \mu)$. 

The comparison between reduced loop homology and loop homology relative to a point goes as follows. Recalling that $\ul\mu=\ul o$, we have a commutative diagram 
$$
\xymatrix{
            & R q_0 \ar[r]^-{\cdot \chi}\ar[d]_-{\cdot \chi}& MC_*(\Lambda;\nu\otimes \ul\mu) \ar@{=}[d] & \\
0 \ar[r]& \chi R q_0 \ar[r]& MC_*(\Lambda;\nu\otimes \ul o) \ar[r]& MC_*(\Lambda,\chi\text{pt};\nu\otimes \ul\mu) \ar[r]& 0
}
$$
which induces 
$$
\xymatrix{
& R \ar[r]^-{\cdot \chi} \ar[d]_-{\cdot \chi} & MH_*(\Lambda;\nu\otimes \ul \mu) \ar@{=}[d] \ar[r]& \ol{MH}_*(\Lambda;\nu\otimes \ul \mu) \ar@{-->}[d] \\
\dots \ar[r]& \chi R \ar[r]& MH_*(\Lambda;\nu\otimes \ul o) \ar[r]& MH_*(\Lambda,\chi\text{pt};\nu\otimes \ul o) \ar[r]& \dots
}
$$
We thus get a canonical map 
\begin{equation}\label{eq:comparison-reduced-relpoint}
\ol{MH}_*(\Lambda;\nu\otimes \ul \mu)\longrightarrow  MH_*(\Lambda,\chi\text{pt};\nu\otimes \ul o)
\end{equation}
which is an isomorphism if and only if the map $\chi R \to MH_0(\Lambda;\nu\otimes \ul o)$ is injective. To study its injectivity we can restrict without loss of generality to the component of contractible loops, in which case the target of this map is $R$ if $\nu\otimes \ul o$ is trivial and $R/2R$ if $\nu\otimes \ul o$ is nontrivial on that component. We thus obtain injectivity of this map, and an isomorphism in~\eqref{eq:comparison-reduced-relpoint}, under any of the following conditions: 

(i) $\nu\otimes \ul o$ is trivial on the component of contractible loops.

(ii) $\chi=0$. 

(iii) $R$ is $2$-torsion.

\subsection{Isomorphism between symplectic homology and loop homology} \label{sec:iso-sympl-loop}

We spell out in this section the isomorphism between the symplectic homology of the cotangent bundle and the homology of the free loop space with twisted coefficients. 

For the next definition, recall the local systems 
$$
\sigma = \tau_{w_2},\qquad \tilde o = \ev_0^*\ul{|M|}^{-w}
$$
from~\eqref{eq:sigma=tauw2} and~\eqref{eq:tildeo}, as well as the orientation local systems 
$$
\mu=\ev_0^*|M|^{-1}=o^{-1}.
$$ 

\begin{definition}[Abouzaid~\cite{Abouzaid-cotangent}] \label{defi:local-system-eta} The \emph{fundamental local system for symplectic homology of the cotangent bundle} is the local system on $\Lambda$ given by 
$$
\eta = \sigma \otimes \mu \otimes \tilde o. 
$$
\end{definition}

The fundamental local system $\eta$ is supported in degree $-n$. Our previous discussion shows that the loop product is defined and has degree $0$ on $H_*(\Lambda;\eta)$, and the loop coproduct is defined and has degree $+1$ on $H_*(\Lambda,\Lambda_0;\eta^{-1})$. We can view the loop product as being defined on $H_*(\Lambda;\ul{\eta})$, where it has degree $-n$, and the loop coproduct as being defined (with field coefficients) on $H_*(\Lambda,\Lambda_0;\ul{\eta})$, where it has degree $1-n$. This point of view is useful when considering $\ol{H}_*(\Lambda;\ul{\eta})$, which is a common space of definition (to which the product descends and the coproduct extends). 

As proved in~\cite{AS-corrigendum,Abouzaid-cotangent}, the chain map $\Psi=\Psi^{\text{quadratic}}$ discussed in~\S\ref{sec:Psi} associated to a quadratic Hamiltonian acts as 
$$
\Psi:FC_*(H)\to MC_*(S_L;\ul{\eta})
$$
and induces an isomorphism $SH_*(D^*M)\stackrel\simeq\longrightarrow H_*(\Lambda;\ul{\eta})$. Given any local system $\nu$, the same map acts as $\Psi:FC_*(H;\nu)\to MC_*(S_L;\nu\otimes\ul{\eta})$ and induces an isomorphism $SH_*(D^*M;\nu)\stackrel\simeq\longrightarrow H_*(\Lambda;\nu\otimes \ul{\eta})$. 

Our filtered chain map $\Psi=\Psi^{\text{linear}}$ from~\S\ref{sec:Psi-symp} associated to a linear Hamiltonian is a chain isomorphism $FC_*(H)\stackrel\simeq\longrightarrow MC_*^{\le \mu}(E^{1/2};\ul{\eta})$, with $\mu$ the slope of the Hamiltonian. Given any local system $\nu$, we obtain a filtered chain isomorphism $FC_*(H;\nu)\stackrel\simeq\longrightarrow MC_*^{\le \mu}(E^{1/2};\nu\otimes \ul{\eta})$. 

In case the local system $\nu$ is compatible with products, the arguments of~\cite{AS,AS2,Abouzaid-cotangent} adapt in order to show that the map $\Psi$ intertwines the pair-of-pants product on the symplectic homology side with the homology product on the Morse side. The arguments of Theorem~\ref{thm:cont-loop-iso} adapt in order to show that the map $\Psi$ descends on homology relative to the constant loops, where it intertwines the continuation coproduct with the loop coproduct (with field coefficients). 

\begin{theorem}[{\cite{AS,AS2,AS-corrigendum,Abouzaid-cotangent}},
Theorem~\ref{thm:Psi-symp},
Theorem~\ref{thm:cont-loop-iso}]
Given any local system $\nu$ compatible with products, the filtered chain level map $\Psi$ induces filtered isomorphisms 
$$
\Psi_*:SH_*(D^*M;\nu)\stackrel\simeq\longrightarrow H_*(\Lambda;\nu\otimes \ul{\eta}),
$$
$$
\Psi_*^{>0}:SH_*^{>0}(D^*M;\nu)\stackrel\simeq\longrightarrow H_*(\Lambda,\Lambda_0;\nu\otimes \ul{\eta}).
$$
Moreover: 

-- $\Psi_*$ intertwines the pair-of-pants product with the Chas--Sullivan loop product, 

-- $\Psi_*^{>0}$ intertwines the continuation coproduct with the loop coproduct (with field coefficients).
\qed
\end{theorem}

There are also reduced versions of the map $\Psi$ which intertwine the product, and which also intertwine the coproducts provided both are defined using the same continuation data at the endpoints. The statements for the coproducts can moreover be interpreted as dual statements about products in cohomology.

\bibliographystyle{abbrv}
\bibliography{000_SHpair}

\end{document}